\documentclass[11pt, reqno]{amsart}

\usepackage{amsmath,amssymb,amsthm,mathtools,wasysym,calc,verbatim,enumitem,tikz,pgfplots,url,mathrsfs,bbm}
\usepackage[left=1in,right=1in,top=1.2in,bottom=1in]{geometry}
\usepackage{latexsym}

\usepackage{color}

\usepackage{amssymb}
\usepackage[centertags]{amsmath}
\usepackage{verbatim}

\usepackage{bbm}
\usepackage{amsthm}
\usepackage{graphicx}
\usepackage{float}
\usepackage{soul}

\usepackage[colorlinks=true,linkcolor=black ,citecolor=black]{hyperref}

\usepackage{amsthm}

\newtheorem{theorem}{Theorem}

\newtheorem{lemma}[theorem]{Lemma}

\newtheorem{proposition}[theorem]{Proposition}
\newtheorem{corollary}[theorem]{Corollary}

\newtheorem{claim}[theorem]{Claim}
\newtheorem{fact}[theorem]{Fact}

\theoremstyle{definition}
\newtheorem*{definition*}{Definition}
\newtheorem{definition}[theorem]{Definition}

\newcommand{\theoremname}{testing}
\theoremstyle{remark}
\newtheorem*{remark*}{Remark}
\newtheorem{remark}[theorem]{Remark}

\numberwithin{theorem}{section}

\renewcommand{\phi}{\varphi}
\renewcommand{\emptyset}{\varnothing}

\renewcommand{\leq}{\le}
\renewcommand{\geq}{\ge}

\newcommand{\eps}{\varepsilon}
\newcommand{\ga}{\gamma}

\newcommand{\la}{\lambda}

\newcommand{\cE}{\mathcal E}

\newcommand{\cZ}{\mathcal Z}
\newcommand{\zhat}{\widehat{z}}
\newcommand{\what}{\widehat{w}}
\newcommand{\bzhat}{\widehat{\mathbf{z}}}
\newcommand{\bwhat}{\widehat{\mathbf{w}}}

\newcommand{\bE}{\mathbb E}

\newcommand*{\dprime}{{\prime \prime}}

\newcommand*{\avg}[1]{\left\langle {#1} \right\rangle}

\def\La{{\Lambda}}
\def\Ga{{\Gamma}}
\def\la{{\lambda}}

\renewcommand{\le}{\leqslant}
\renewcommand{\ge}{\geqslant}

\newcommand{\bR}{\mathbb R}
\newcommand{\bC}{\mathbb C}
\newcommand{\bZ}{\mathbb Z}
\newcommand{\bD}{\mathbb D}
\newcommand{\bT}{\mathbb T}
\newcommand{\bN}{\mathbb N}

\newcommand{\bP}{\mathbb P}
\newcommand{\bS}{\mathbb S}

\newcommand{\Cov}{\operatorname{Cov}}

\title
[Fluctuations in the logarithmic energy]{Fluctuations in the logarithmic energy for zeros of random polynomials on the sphere}

\author{Marcus Michelen}
\address{\tiny{Marcus Michelen, Department of Mathematics, Statistics, and Computer Science, University of Illinois at Chicago}}
\email{michelen@uic.edu, michelen.math@gmail.com}
\author{Oren Yakir}
\address{\tiny{Oren Yakir, School of Mathematical Sciences, Tel Aviv University}}
\email{oren.yakir@gmail.com}

\begin{document}
		
	\begin{abstract}
		Smale's Seventh Problem asks for an efficient algorithm to generate a configuration of $n$ points on the sphere that nearly minimizes the logarithmic energy. As a candidate starting configuration for this problem, Armentano, Beltr\'an and Shub considered the set of points given by the stereographic projection of the roots of the random elliptic polynomial of degree $n$ and computed the expected logarithmic energy.  We study the fluctuations of the logarithmic energy associated to this random configuration and prove a central limit theorem.  Our approach shows that all cumulants of the logarithmic energy  are asymptotically linear in $n$, and hence the energy is well-concentrated on the scale of $\sqrt{n}$. 
	\end{abstract}	
\maketitle
\section{Introduction} 
Consider the random polynomial of degree $n\ge 1$ defined by
\begin{equation}
\label{eq:def_random_polynomial}
f_n(z) = \sum_{j=0}^{n} a_j \, \sqrt{\binom{n}{j}} \, z^j\, , \qquad { z\in \bC }
\end{equation}
where $a_0,\ldots,a_n$ are i.i.d.\ standard complex Gaussian random variables. This model for random polynomials,  usually referred to as elliptic polynomials, Kostlan polynomials or the $SU(2)$ model, was introduced in the physics literature by works of Bogomolny, Bohigas, Leboeuf~\cite{BBL1,BBL2} and Hannay~\cite{Hannay}, and from the mathematical point of view by Kostlan~\cite{Kostlan} and Shub, Smale~\cite{ShubSmale}. A remarkable feature of this model is that the random point process on the sphere 
$$\bS^2 = \{x\in \bR^3:\, \|x\|=1\}$$ 
given by the stereographic projection of the zero set of $f_n$ has a distribution which is invariant under isometries of $\bS^2$, see Figure~\ref{figure:zero_set_with_spherical}.  In fact, it is the unique Gaussian analytic function with this invariance property, see~\cite{Sodin} or~\cite[Theorem~2.5.2]{HKPV}.
\vspace*{-7mm}
\begin{figure}
\label{figure:zero_set_with_spherical}
\begin{center}	\scalebox{0.7}{\includegraphics{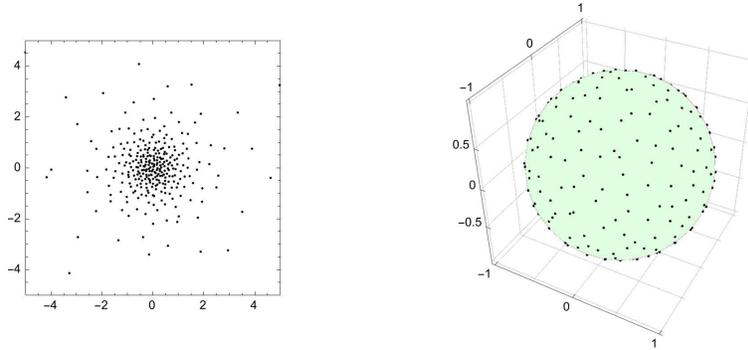}}
\end{center}
\caption{Realization of the zero set of the random elliptic polynomial of degree $n=300$ (left) and its stereographic projection (right).}
\end{figure} 
\subsection{Logarithmic energy}
Given a collection of points $x_1,\ldots,x_n \in \bS^2\subset \bR^3$ the discrete logarithmic energy is given by
\[
\mathcal{E}_n(x_1,\ldots,x_n) = \sum_{i\not=j} \log \frac{1}{\| x_i-x_j \|} \, .
\]
The asymptotic behavior of the minimum of this energy has been extensively studied, see for example the survey paper~\cite{BHS}. Smale's Seventh Problem~\cite{Smale-Problem} asks whether one can find an efficient algorithm which generates configurations of points on the sphere whose logarithmic energy is close to the minimum. By relating the discrete minimization problem to a continuous  variational problem in the $n = \infty$ limit, B\'etermin and Sandier \cite{BS} showed that there exists a constant $C_{{\sf{min}}}$ so that \begin{equation}
		\label{eq:minimum_log_energy_on_sphere}
		\min_{x_1\ldots,x_n \in \bS^2} \mathcal{E}_n(x_1,\ldots,x_n) = \Big(\frac{1}{2} - \log 2\Big) n^2 - \frac{1}{2} n \log n + C_{\sf{min}} n + o(n) \,.
\end{equation} as $n \to \infty$.  The exact value of $C_{{\sf{min}}}$ is not known, with the best-known current bounds giving
\[
-0.0569... \le C_{\sf{min}} \le 2\log 2 + \frac{1}{2} \log \frac{2}{3} + 3 \log \frac{\sqrt{\pi}}{\Ga(1/3)} = -0.0556...
\]
see~\cite{Lauritsen, Steinerberger}. It is conjectured that the upper bound on $C_{\sf{min}}$ is in fact its true value, see~\cite{BS, BHS}. Points $x_1,\ldots,x_n\in \bS^2$ which achieve the minimum in~\eqref{eq:minimum_log_energy_on_sphere} are called elliptic Fekete points.

Motivated by constructing near minima to the minimization problem \eqref{eq:minimum_log_energy_on_sphere}, it is natural to study random configurations of points on the sphere which are expected to have small logarithmic energy. If we put $n$ independent, uniformly distributed points on the sphere then it is easy to check that the expected logarithmic energy is 
\[
\Big(\frac{1}{2} - \log 2\Big)n^2 - \Big(\frac{1}{2} - \log 2\Big)n 
\]
which agrees with the minimal value~\eqref{eq:minimum_log_energy_on_sphere} up to a $\Theta(n\log n)$ term. One downside to using uniformly distributed points on the sphere is that there will be many pairs of points that are quite close to each other; this has the effect of introducing an extra term of order $n\log n$ over the minimum possible energy.  In order to obtain a random configuration of points with smaller expected logarithmic energy, one route is to look at a distribution of points that experiences repulsion between points.

An idea introduced by Armentano, Beltr\'{a}n and Shub~\cite{ArmentanoBeltranShub} is to consider the stereographic projection of the zeros of the random elliptic polynomial given by~\eqref{eq:def_random_polynomial}. As the zeros of random functions tend to repel each other, the points are unlikely to clump and hence the energy is expected to be small. Indeed, Armentano, Beltr\'{a}n and Shub~\cite{ArmentanoBeltranShub} proved that  the expected logarithmic energy for this model is given by  
\[
\Big(\frac{1}{2} - \log 2\Big)n^2 - \frac{1}{2} n \log n  - \Big(\frac{1}{2} - \log 2\Big)n \, .
\]
That is, the expected logarithmic energy for the elliptic zeros agrees with the minimal value~\eqref{eq:minimum_log_energy_on_sphere} up to a $\Theta(n)$ term.

\subsection{Main result}
Throughout this paper, we identify the Riemann sphere $\bS^2$ with the extended complex plane $\displaystyle \widehat{\bC} = \bC\cup\{\infty\}$ via {the} stereographic projection. The spherical distance on $\widehat{\bC}$ is the three-dimensional Euclidean distance after stereographic projection and is given explicitly by
\begin{equation}\label{eq:dS-def}
d_{\bS^2}(z,w) = \frac{2|z-w|}{\sqrt{1+|z|^2} \sqrt{1+|w|^2}}\, , \qquad z,w\in \bC 
\end{equation}
and $$d_{\bS^2}(z,\infty) = \frac{2}{\sqrt{1+|z|^2}}\, .$$ Denote by $\mathcal{Z} = \{\zeta_1,\ldots,\zeta_n\}$ the zero set of the random polynomial $f_n$ given by~\eqref{eq:def_random_polynomial}, and consider the corresponding logarithmic energy
\begin{equation}
\label{eq:def_of_log_energy}
\mathcal{E}_n = \mathcal{E}_n (\zeta_1,\ldots,\zeta_n) = -\sum_{i\not=j} \log d_{\bS^2}(\zeta_i,\zeta_j)\, .
\end{equation}
The main result of~\cite{ArmentanoBeltranShub} {(see also Lemma~\ref{lemma:expectation_of_E_n} below)} states that
\begin{equation}
\label{eq:expectation_for_Energy}
\bE \big[\mathcal{E}_n\big] = \Big(\frac{1}{2} - \log 2\Big)n^2 - \frac{1}{2} n \log n  - \Big(\frac{1}{2} - \log 2\Big)n \, .
\end{equation}
In this paper, we prove that the logarithmic energy is well concentrated around its mean {by proving a central limit theorem for $\cE_n$ at the scale $\sqrt{n}$.} Our main result is the following.
\begin{theorem}
\label{thm:main_result_clt}
Let $\mathcal{E}_n$ be the logarithmic energy~\eqref{eq:def_of_log_energy} of the roots of the random polynomial $f_n$ given by~\eqref{eq:def_random_polynomial}. There exists a constant $c_\ast>0$ such that
\begin{equation*}
{\sf Var} \big[\mathcal{E}_n\big] = c_\ast n \big(1+o(1)\big)
\end{equation*}
as $n\to \infty$.
Furthermore, the sequence of random variables
\[
\frac{\mathcal{E}_n - \bE[\mathcal{E}_n]}{\sqrt{c_\ast\, n}}
\]
converges in distribution as $n\to \infty$ to the standard (real) Gaussian law.
\end{theorem}  

The constant $c_\ast$ is somewhat explicit and is given as a sum of several integrals, see Section~\ref{sec:variance_asymptotics} (and, in particular, see~\eqref{eq:explicit_formula_for_c_ast}) for more details. Numerical computations of these integrals show that
\begin{equation*}
	c_\ast \approx 0.0907056 \, .
\end{equation*}
The proof of Theorem~\ref{thm:main_result_clt} will follow the method of moments, and in fact we will prove that for all $k\ge 1$ we have $$ \lim_{n \to \infty} \frac{\bE \left[ (\cE_n - \bE[\cE_n])^{2k}\right]}{n^k} = \frac{(2k)!}{2^{k}k!}c_\ast^k  \, .$$
By Chebyshev's inequality, this proves concentration of $\cE_n$.

	\begin{corollary}\label{cor:concentration}
	For each $A \geq 0$ there  exists a constant $C_A$ so that for all $T > 0$ and $n\ge 1$ we have
	\begin{equation*}
		\bP\Big(\big|\cE_n - \bE [\cE_n] \big| \geq T\sqrt{n} \, \Big) \leq C_A T^{-A} \,.
	\end{equation*}
	\end{corollary}
Conversely, the central limit theorem of Theorem \ref{thm:main_result_clt} shows a limiting probability for fluctuations on the scale $\sqrt{n}$, namely that for each fixed $T$ we have $$\lim_{n \to \infty} \bP\Big(\cE_n - \bE \big[\cE_n\big] < -T \sqrt{n} \Big) =  \frac{1}{\sqrt{2\pi c_\ast}}\int_{-\infty}^{{ -T}}  e^{-t^2 / (2c_\ast)} \, {\rm d}t \,.$$ 

In the context of Smale's Seventh Problem, Armentano, Beltr\'an and Shub~\cite{ArmentanoBeltranShub} introduced the idea of using random polynomials to construct near-minimizers for the logarithmic energy; they outline the strategy of starting with a configuration generated by the roots of $f_n$ and subsequently performing gradient descent in order to further minimize the energy.  The recent book of Borodachov, Hardin and Saff~\cite{borodachov-hardin-saff} contains descriptions of many other point configurations on the sphere and numerically compares their logarithmic energy, all of which are deterministic aside from uniformly random points.  Notably, \emph{all} of the point configurations they consider appear to have logarithmic energy varying by $\Omega(n)$ from the suspected minimizer (see~\cite[Figure 7.16]{borodachov-hardin-saff}), as is the case for the configuration considered here.  One potential benefit of starting with a random configuration of points, such as the configuration generated by random polynomials, is that the function ${\cE}_n: (\bS^2)^n \to \bR$ appears to have exponentially many local minima in $n$ (see the discussion in~\cite{Beltran-state-of-the-art}).  One possibility is that a random configuration such as one provided by random polynomials may provide a good starting point for a subsequent algorithm such as gradient descent. 

\subsection{Related Works}
As mentioned above, the expectation of $\mathcal{E}_n$ was already computed by Armentano, Beltr\'{a}n and Shub in~\cite{ArmentanoBeltranShub}. Recently, de la Torre and Marzo~\cite{DLTMarzo} gave a different approach for computing the expectation which allowed them, among other things, to study the expected Riesz energies for the random elliptic zeros; this amounts to replacing the logarithm in the definition of $\cE_n$ with a function of the form $z \mapsto z^s$. Another possible generalization is to consider the analogous expected energies for random sections of any positive line bundle over compact K\"{a}hler manifolds, see Feng and Zelditch~\cite{FZ}.

A natural question is to study the logarithmic energy for other point processes on the sphere which exhibits repulsion, with the hope of finding an example with energy closer to the minimum in \eqref{eq:minimum_log_energy_on_sphere}.  A natural candidate is the spherical ensemble, where one lets $A_n$ and $B_n$ be independent random matrices with i.i.d.\ standard complex Gaussian entries and stereographically projects the eigenvalues of $A_n^{-1}B_n$ onto the sphere.  Krishnapur~\cite{Krishnapur} showed that this ensemble is in fact a determinantal point process on the sphere, and so many computations become tractable. Indeed, Alishahi and Zamani~\cite{AZ} computed the asymptotic expected energies (both logarithmic and Riesz) as the number of points in the process tends to infinity. Interestingly, the expected logarithmic energy for the spherical ensemble is asymptotically larger than that of the elliptic zeros, which hints that the random zeros are closer to being elliptic Fekete points. An advantage for working with the spherical ensemble is that it generalizes to any dimension, and indeed Beltr\'{a}n, Marzo and Ortega-Cerd\`{a}~\cite{BMO} computed the asymptotic for the expected logarithmic energy in arbitrary dimension.

In terms of computing the variance or proving central limit theorems for these types of energies, not much seems to be known. In the context of independent uniform points on the sphere, the logarithmic (or Riesz) energy is a special case of a more general framework of studying U-statistics, see~\cite[Chap.~6]{Janson}. Indeed, a central limit theorem for U-statistics of a Poisson point process (in a general measure space) was proved by Reitzner and Schulte in~\cite{RS}.  Another work in this direction is by B\l aszczyszyn, Yogeshwaran and Yukich~\cite{BlazsYogeshYukich} which proved, among other things, a central limit theorem for local score functions for general stationary point processes in $\bR^d$. As the logarithmic energy is not a local quantity (the work \cite{BlazsYogeshYukich} requires the interaction to decay exponentially with distance), their result do not apply in our case. Perhaps more relevant to our work is the Wiener chaos expansion (sometimes referred to as the Hermite-It\^o expansion); an orthogonal decomposition of non-linear functionals of some underlying Gaussian processes, see~\cite[Chap.~2]{Janson}. In the context of Gaussian analytic functions the chaos technique was used, either implicitly or explicitly, in~\cite{Buckley-Nishry, NS-IMRN, ST} to prove normal fluctuations for linear statistics of the zeros. The main difference between the logarithmic energy (studied in this paper) and the previous works is that the logarithmic energy is \emph{not} a linear statistic of the roots but rather a singular function of the distinct pairs.  Nevertheless, we will use the Wiener chaos framework to obtain a simpler form for the limiting variance, see Section~\ref{sec:variance_asymptotics}. 

\subsection{Outline of the proof}
\label{ss:outline} The proof of Theorem \ref{thm:main_result_clt} begins by finding an expression for the logarithmic energy $\cE_n$ that reflects the spherical symmetry inherent in the problem. This takes the form of Lemma \ref{lemma:rewrite_energy_as_I_minus_S}, in which we write 
\begin{equation}\label{eq:cE-sketch}
	\cE_n - \Big(\frac{1}{2} - \log 2\Big) n^2 + \frac{n\log n}{2} - (\log 2)n = n \int_{\bC} \log |\widehat{f}_n(z)|\, {\rm d}\mu(z) - \sum_{\zeta \in \cZ_n} \log |D\widehat{f}_n(\zeta)|
\end{equation}
where $\cZ_n$ is the zero set of $f_n$ and $\widehat{f}_n, D\widehat{f}_n$ are given as
\begin{equation*}
	\widehat{f}_n(z) = \frac{f_n(z)}{(1+|z|^2)^{n/2}}\, , \qquad D\widehat{f}_n(z) = \frac{f_n^\prime(z)}{\sqrt{n} (1+|z|^2)^{n/2-1}}\, .
\end{equation*}
Here and everywhere, $m$ is the Lebesgue measure and \begin{equation}
	\label{eq:mu-def} {\rm d}\mu(z) = \frac{{\rm d}m(z)}{\pi(1+|z|^2)^2}
\end{equation}is the uniform (probability) measure on $\widehat \bC$.
From here, we see that the two terms on the right-hand side of~\eqref{eq:cE-sketch} both have distribution which is invariant under spherical isometries of the Riemann sphere. To get a feel for why a central limit should hold, note that a simple calculation shows that  the Gaussian process $f_n$ decorrelates on the scale of $1/\sqrt{n}$ with respect to the spherical metric. This means that for points $w,z$ with $d_{\bS^2}(w,z) \gg n^{-1/2}$, we expect $f_n$ to behave roughly independently when evaluated at $z$ and $w$. The two terms on the right-hand side~\eqref{eq:cE-sketch} may each be understood as similar to sums of $n$ nearly independent random variables.  Another key observation, is that the contribution to the right-hand side of~\eqref{eq:cE-sketch} coming from spherical caps of radius $\sim n^{-1/2}$ contribute a random quantity with fluctuations of constant order. The heuristics above suggests that the total fluctuations should be of order $\sqrt{n}$.
 
 \subsubsection{The Gaussian Entire Function }
One way to understand the scaling factor $1/\sqrt{n}$ is to see that the rescaled function
\begin{equation}\label{eq:fn-rescale}
	f_n\Big(\frac{z}{\sqrt{n}}\Big) = \sum_{j=0}^{n} a_j \,  \sqrt{n^{-j} \binom{n}{j}} \, z^j
\end{equation}
has a non-degenerate limit as $n\to \infty$.  Indeed, recall that the Gaussian Entire Function (G.E.F.) is defined by the random Taylor series
\begin{equation}\label{eq:GEF-def-intro}
	g(z) = \sum_{j=0}^{\infty} a_j \frac{z^j}{\sqrt{j!}} \, , \qquad  z\in \bC
\end{equation}
where, as before, $\{a_j\}$ is an i.i.d.\ sequence of standard complex Gaussian random variables.  A remarkable feature of the random zeros of the G.E.F.\ is its distribution invariance with respect to isometries of the plane $\bC$, see~\cite{NS-WhatIs, ST}. It is not difficult to check (see Claim~\ref{claim:local_convergence_of_polynomials_to_GEF} below), that the sequence of Gaussian processes $f_n(\cdot/\sqrt n)$ converges as $n\to \infty$ to the G.E.F.\ in some appropriate sense.  Using this observation, we will show that as $n$ becomes large the local contributions to the right-hand side of~\eqref{eq:cE-sketch} is close to a similar contribution coming from the zeros of the G.E.F., with the later contribution having typical fluctuations of constant order.  We note that our analysis could be made to work without making use of the G.E.F.\ as the scaling limit, and that the use of this scaling limit is for convenience and motivation for certain calculations.

\subsubsection{Intensity functions}
The $k$-point functions for $k\ge 1$ express correlations within $k$ points from the zero set $\mathcal{Z}_n$. It is the symmetric function
 \begin{equation*}
 	\rho = \rho_k : \big\{ Z=(z_1,\ldots,z_k)\in \bC^k : \, z_i\not=z_j \ \text{for } i\not=j \big\} \to \bR_{\ge 0}
 \end{equation*}
 defined via the relation
 \begin{equation*}
 	\bE\Big[ \prod_{j=1}^{k} \#(\mathcal{Z}_n\cap B_j) \Big] = \int_{B_1\times\ldots\times B_k} \rho(z_1,\ldots,z_k) \, {\rm d}m(z_1)\cdots{\rm d}m(z_k) \, ,
 \end{equation*}
 where $B_1,\ldots,B_k\subset \bC$ is any family of disjoint Borel sets. For each fixed $k\ge 1$, the $k$-point function for $\mathcal{Z}_n$ exists for all $n$ large enough, see for example~\cite[Corrolary~3.4.2]{HKPV}.  To justify the almost independence argument suggested in Section~\ref{ss:outline}, we will prove a (strong) clustering property for the  intensity functions of the random zero set $\mathcal{Z}=\mathcal{Z}_n$, which may be of independent interest. The clustering property is a quantitative manifestation of near independence for points from the process at large distances, see~\cite[\S4.4]{Ruelle}, in the sense that the correlations decay with the distance.
 
 \subsubsection{Clustering of the random zeros}
 Let $I = \{i_1,\ldots,i_\ell\} \subset  \{1,\ldots,k\}$ be a non-trivial subset, set $J = \{1,\ldots,k\} \setminus I$  and let $Z_I = \{z_{i_1},\ldots,z_{i_{\ell}}\}$ and define $Z_J$ analogously. We denote
 \[
 d(Z_I,Z_J) = \inf_{i\in I, j\in J} d_{\bS^2} (z_i,z_j) \, .
 \] 
 
 \begin{theorem}
 	\label{thm:clustering_intro}
 	For all $k\ge 2$ there exists constants $c_k,C_k,D_k>0$ such that the following holds. For any non-trivial partition $\{1,\ldots,k\} = I\sqcup J$ such that $d (Z_I,Z_J) \ge D_k/\sqrt{n}$ we have
 	\[
 	\big| \rho(Z) - \rho(Z_I) \rho(Z_J) \big| \le C_k \, n^{k} e^{-c_k n d(Z_I,Z_J)^2}
 	\]
 	for all $n\ge 1$ large enough.
\end{theorem}
	In fact, due to the non-linear nature of the logarithmic energy $\mathcal{E}_n$, we will need to prove this clustering property for more complicated densities (like the ones defined in Definition~\ref{definition:def_of_kac_rice_densities} below), and the clustering property for the $k$-point function is a  special  case of the more general result we will prove,  see remark after Theorem~\ref{thm:clustering_of_density}. In~\cite{NS-CMP} Nazarov and Sodin proved, among other things, the clustering property holds for the zero set of the G.E.F. Our proof for the clustering property for $\mathcal{Z}_n$ follows their ideas at certain steps, modulo some technicalities which come from the different nature of the densities we consider, see Section~\ref{sec:clustering_for_random_measure} for more details. Additionally, a necessary feature of our approach shows that the clustering property holds quantitatively as $n$ tends to infinity, and not only in the limiting case of the G.E.F. 
 
 	\subsubsection{Degeneracy near the diagonal}  When working with the $k$-point functions of $\mathcal{Z}_n$, we will need to understand their behavior near the diagonal. To illustrate the technical problem, we will focus here just on the case $k=2$. The classical \emph{Kac-Rice formula} (sometimes also attributed to Hammersley) reads that for $z\not=w,$
 	\begin{equation} \label{eq:2-pt-function-intro}
 		\rho_2(z,w) = \frac{\bE\big[|f_n^\prime(z)|^2 \, |f_n^\prime(w)|^2 \mid f_n(z) = f_n(w) = 0\big]}{\pi^2 \det \big[\Cov(f_n(z),f_n(w))\big]}
 	\end{equation}	
 	see~\cite[Corollary~3.4.2]{HKPV} or~\cite[Theorem~6.3]{Azais-Wschebor}. The issue now becomes apparent: if $z$ and $w$ are close then the Gaussian vector $\big(f_n(z),f_n(w)\big)$ becomes nearly-degenerate and both the numerator and the denominator in~\eqref{eq:2-pt-function-intro} tend to zero.  To handle this issue of degeneracy, we will consider divided differences of our polynomial $f_n$.  For instance, in our example, if $z$ and $w$ are very close to each other, rather than considering the nearly-degenerate Gaussian vector $\big(f_n(z),f_n(w)\big)$, we will consider instead $$\left(f_n(z), \frac{f_n(w) - f_n(z)}{w-z} \right) $$ which we will show is uniformly non-degenerate provided $z$ and $w$ are not too far apart. Bounds on the Kac-Rice density will follow from rewriting the density entirely in terms of these non-degenerate divided differences,  and showing that the singularities in the denominator are canceled out by a similar term in the numerator.
 
 	The use of divided differences to study random polynomials has been developed in recent years,  pioneered in the works by  Ancona-Letendre~\cite{AnconaLetendre} and then further developed by Gass~\cite{Gass} (see also Michelen-Sahasrabudhe~\cite{MichelenSahasrabudhe} for another instance). One benefit of this more combinatorial approach is that it allows one to control all moments  of the logarithmic energy, providing concentration of the form of Corollary~\ref{cor:concentration} in addition to a limit theorem.  It is also worth mentioning that in~\cite{NS-CMP}, Nazarov and Sodin also dealt with a similar degeneracy of the intensity functions  (of the random zeros of Gaussian analytic functions) near the diagonal using simple tools from complex analysis. As we will work with a representation of the divided differences using Cauchy integrals of rational functions (see~\eqref{eq:DD_as_cauchy_integral}  below), our approach can be thought of as a unification of the approaches considered above.
 	
	\subsection*{Organization of the paper}
	Section~\ref{sec:preliminaries} contains some preliminary results and the breakdown of the proof of Theorem~\ref{thm:main_result_clt} as described in the outline. To deal with the joint moments (or rather, the joint cumulants) of the two terms of the right-hand side of~\eqref{eq:cE-sketch} we will introduce Kac-Rice type densities, similar to the one described in~\eqref{eq:2-pt-function-intro} but adapted to the logarithmic nature of $\cE_n$. The bulk of the argument is to show that these densities satisfy an upper bound (Theorem~\ref{thm:local_bounds_on_dentity}) and approximately factor when evaluated on two collections of points of spherical distance $ \gg n^{-1/2}$ apart (Theorem~\ref{thm:clustering_of_density}). Section~\ref{sec:local_bounds} is devoted to the proof of Theorem~\ref{thm:local_bounds_on_dentity} and in Section~\ref{sec:clustering_for_random_measure} we prove Theorem~\ref{thm:clustering_of_density}. With these two statements in mind, a combinatorial argument which is carried out in Section~\ref{sec:cumulant_densities} shows that each cumulant of $\cE_n$ is asymptotically linear in $n$. From here, to conclude the desired central limit theorem we just need to show that the variance is not $o(n)$, and this is indeed shown in Section~\ref{sec:variance_asymptotics}. 

\subsection*{Notation}
We end the introduction by listing some notations that we will use across sections of the paper:
\begin{itemize}
	\item $\bC$, $\bR$ the complex plane and the real line; $\widehat{\bC}$ the Riemann sphere; $\bR_{\ge0}$ the non-negative reals;
	\item $\bC_{\not=}^m$ the set of all $m$-tuples of distinct complex numbers for $m\ge 1$;
	\item $d_{\bS^2}(z,w)$ the spherical distance between $z,w\in \widehat{\bC}$, defined via~\eqref{eq:dS-def};
	\item $\cZ$ the zero set of the random polynomial $f_n$ given by~\eqref{eq:def_random_polynomial}; 
	\item $\cE_n$ the logarithmic energy of the roots defined via~\eqref{eq:def_of_log_energy};
	\item $K(z,w) = (1+z\overline{w})^n$ the covariance kernel of $f_n$;
	\item $m$ the Lebesgue measure on $\bC$; $\mu$ the uniform probability measure on $\widehat\bC$, defined in~\eqref{eq:mu-def};
	\item $\widehat{f}_n$ and $D\widehat{f}_n$ denote normalized versions of $f_n$ and $f_n^\prime$, defined by~\eqref{eq:fhat-def} and~\eqref{eq:dhat-f};
	\item $\bE, {\sf Cov}, {\sf Var}$ the expectation, covariance and variance with respect to the underlying probability space;
	\item $\Cov(\mathbf{Z})$ the covariance matrix of a Gaussian random vector $\mathbf{Z}$;
\end{itemize}
We will use freely the Landau notations $O(\cdot),\Omega(\cdot)$ and $\Theta(\cdot)$ to denote inequalities up to non-asymptotic constants, and similarly right $a\lesssim b$ if $a=O(b)$. We also recall that a standard complex Gaussian is a random variable taking values in $\bC$ with the density $\frac{1}{\pi} e^{-|z|^2}$ with respect to the Lebesgue measure $m$.

	\section{Preliminaries}
	\label{sec:preliminaries}
	Recall that $f_n$ is the random polynomial given by~\eqref{eq:def_random_polynomial} with zero set $\mathcal{Z}$. In fact, $f_n$ is the mean-zero (complex) Gaussian process on $\bC$, with the covariance kernel
	\begin{equation}
		\label{eq:covariance_kernel_for_f}
		K(z,w) \stackrel{{\rm def}}{=} \bE\big[f_n(z) \overline{f_n(w)}\big] = \sum_{j=0}^{n} \binom{n}{j} (z\overline w)^j = \big(1+z\overline{w}\big)^n \, .
	\end{equation} 
	In view of~\eqref{eq:covariance_kernel_for_f}, it is natural to consider the following normalization
	\begin{equation}\label{eq:fhat-def}
		\widehat{f}_n(z) \stackrel{{\rm def}}{=} \frac{f_n(z)}{(1+|z|^2)^{n/2}} \, .
	\end{equation}
	We will also normalize the corresponding derivative as	
	\begin{equation}\label{eq:dhat-f}
		D\widehat{f}_n(z) \stackrel{{\rm def}}{=} \frac{f_n^\prime(z)}{\sqrt{n} (1+|z|^2)^{n/2-1}}\, .
	\end{equation} 

	We will see in Section \ref{ss:spherical-symmetry} that this scaling allows us to see the spherical symmetry underlying the function $f_n$.  
	
	\subsection{Rewriting the logarithmic energy}
	The first key step towards proving Theorem~\ref{thm:main_result_clt} will be rewriting the logarithmic energy~\eqref{eq:def_of_log_energy} in a way that reveals the spherical symmetry. Denote by
	\begin{align}
		\label{eq:def_of_I_n}
		\mathcal{I}_n &\stackrel{{\rm def}}{=} n\int_{\bC} \log|\widehat{f}_n(z)|\, {\rm d}\mu(z)\\ \label{eq:def_of_S_n} \mathcal{S}_n &\stackrel{{\rm def}}{=} \sum_{\zeta\in \mathcal{Z}} \log|D\widehat{f}_n(\zeta)|\, .
	\end{align}
	\begin{lemma}
		\label{lemma:rewrite_energy_as_I_minus_S}
		Let $\mathcal{E}_n$ be the logarithmic energy~\eqref{eq:def_of_log_energy} and let $\mathcal{I}_n$ and $\mathcal{S}_n$ be given by~\eqref{eq:def_of_I_n} and~\eqref{eq:def_of_S_n}, respectively. Then
		\[
		\mathcal{E}_n = \Big(\frac{1}{2} - \log 2\Big) n^2 - \frac{n\log n}{2} + \mathcal{I}_n - \mathcal{S}_n + (\log 2)n \, .
		\]
	\end{lemma}
	
		\begin{proof}
		Write $\cZ = \{\zeta_1,\ldots,\zeta_n\}$ and expand \begin{align*}
		\cE_n &= -\sum_{i\not=j} \log \frac{2|\zeta_i - \zeta_j|}{\sqrt{1 + |\zeta_i|^2}\sqrt{1 + |\zeta_j|^2}} \\ &= 2(n-1)\sum_{j = 1}^n \log \sqrt{1  + |\zeta_j|^2} - \sum_{i\not=j} \log |\zeta_i - \zeta_j| - \log (2) \,  n (n-1)\,.
		\end{align*}
		Noting that for each fixed $j$ we have $|f^\prime(\zeta_j)| = |a_n| \prod_{i \neq j} |\zeta_j - \zeta_i|$ yields 
		\begin{equation}
		\label{eq:log_energy_as_three_terms_from_ABS}
		\mathcal{E}_n = 2(n-1) \sum_{j = 1}^n \log \sqrt{1+|\zeta_j|^2} - (\log 2) \, n (n-1) - \sum_{j = 1}^n \log|f^\prime(\zeta_j)| + n \log|a_n|\,.
		\end{equation}
		This equality appears as \cite[Proposition~1]{ArmentanoBeltranShub}, with an extra factor $\log 2$ due to different normalization of the spherical metric. The lemma follows from the elementary calculation 
		\begin{equation}
		\label{eq:Jensen_on_the_sphere}
		\int_{\bC} \log|f_n(z)| \, {\rm d}\mu(z) = \log|a_n| + \sum_{j =1}^n \log\sqrt{1+|\zeta_j|^2} \, .
		\end{equation}
		Indeed, by plugging~\eqref{eq:Jensen_on_the_sphere} into~\eqref{eq:log_energy_as_three_terms_from_ABS} we get that
		\begin{align*}
		\mathcal{E}_n + &\log (2) \, n(n-1) \\ &= \left(n-2\right) \sum_{j = 1}^n \log \sqrt{1+|\zeta_j|^2} - \sum_{j=1}^n \log|f^\prime(\zeta_j)| + n \int_{\bC} \log|f(z)| \, {\rm d} \mu(z) \\ 
		&= n \int_{\bC} \log \left(\frac{|f(z)|}{(1+|z|^2)^{n/2}}\right) \, {\rm d}\mu(z) + \frac{n^2}{2} \int_{\bC} \log \left(1+|z|^2\right) \, {\rm d}\mu(z)  - \sum_{j = 1}^n \log \left(\frac{|f^\prime(\zeta_j)|}{(1+|\zeta_j|^2)^{n/2 -1}} \right) \\ 
		&= \frac{n^2}{2} - \frac{n \log n}{2} + \mathcal{I}_n - \mathcal{S}_n \, ,
		\end{align*}
		where in the last equality we used
		\begin{equation*}
		\int_{\bC} \log\big(1+|z|^2\big) \, {\rm d}\mu(z) = \int_{0}^{\infty} 2t \, \frac{\log(1+t^2)}{(1+t^2)^2} \,{\rm d}t = 1 \, .
		\end{equation*} 
		To conclude the proof of the lemma, it remains to prove~\eqref{eq:Jensen_on_the_sphere}. Since $f_n$ is a polynomial, it is enough if we prove that for all $\zeta\in \bC$ 
		\[
		\int_{\bC} \log|z-\zeta| \, {\rm d}\mu (z) = \frac{1}{2} \log\big(1+|\zeta|^2\big)\, .
		\]
		Recalling Jensen's formula
		\[
		\frac{1}{2\pi} \int_{0}^ {2\pi} \log|te^{i\theta} - \zeta| \, {\rm d}\theta = \begin{cases}
		\log|\zeta| & t < |\zeta| \, , \\ \log t & t\ge |\zeta| \, ;
		\end{cases}
		\]
		we can compute in polar coordinates and get 
		\begin{align*}
		\int_{\bC} \log|z-\zeta| \, {\rm d}\mu(z) &= \int_{0}^{\infty} \frac{2t}{(1+t^2)^2} \bigg(\frac{1}{2\pi} \int_{0}^{2\pi} \log|te^{i\theta} - \zeta| \, {\rm d}\theta \bigg) \, {\rm d}t \\ &= \log|\zeta| \int_{0}^{|\zeta|} \frac{2t}{(1+t^2)^2} \, {\rm d}t	+ \int_{|\zeta|}^{\infty} \log(t) \, \frac{2t}{(1+t^2)^2} \, {\rm d}t \\ &= \log|\zeta| \frac{|\zeta|^2}{1+|\zeta|^2} + \frac{1}{2}\log(1+|\zeta|^2) - \log|\zeta| + \frac{\log|\zeta|}{1+|\zeta|^2} \\ &= \log(1+|\zeta|^2) \, .
		\end{align*}
		This proves~\eqref{eq:Jensen_on_the_sphere} and we are done.
	\end{proof}
	
	In view of Lemma~\ref{lemma:rewrite_energy_as_I_minus_S}, the central limit theorem for $\mathcal{E}_n$ will follow if we can prove that $\big(\mathcal{I}_n,\mathcal{S}_n\big)$, after proper normalization, converges in distribution as $n\to \infty$ to a non-trivial Gaussian vector in $\bR^2$.  We prove the latter statement by computing the joint cumulants of $\mathcal{I}_n$ and $\mathcal{S}_n$, which will be possible through some ``Kac-Rice type" densities described in Section~\ref{sec:kac_rice_densities} below. The limiting variance follows from this computation, and for pedagogical reasons we record this limit as a separate theorem.
	\begin{theorem}
		\label{thm:variance_after_split}
		There exist constants $c_1,c_2,c_3>0$ such that
		\begin{equation*}
			\lim_{n\to \infty} \frac{\sf Var\left[\mathcal{I}_n\right]}{n} = c_1 \, , \quad \lim_{n\to \infty} \frac{\sf Var\left[\mathcal{S}_n\right]}{n} = c_2\, , \quad \lim_{n\to \infty} \frac{\sf Cov\left[\mathcal{I}_n,\mathcal{S}_n\right]}{n} = c_3 \, .
		\end{equation*}
		Furthermore, we have the relation $c_\ast = c_1 + c_2 - 2c_3>0$, where $c_\ast$ is the limiting constant from Theorem~\ref{thm:main_result_clt}. 
	\end{theorem}
	The proof of Theorem~\ref{thm:variance_after_split} is given in Section~\ref{sec:variance_asymptotics}, where we also give some integral expression of the limiting constant $c_\ast$. The fact that the limiting constant is positive is then proved in Appendix~\ref{sec:variance_is_positive}. For the completeness of this paper and the reader's convenience, we also provide the computation for the expectation (the main result from~\cite{ArmentanoBeltranShub}), see Lemma~\ref{lemma:expectation_of_E_n} below.
	
	\subsection{Spherical symmetry} \label{ss:spherical-symmetry}
	Recall that the law of the random zero set $\mathcal{Z}$ is invariant under rotations of the Riemann sphere. We wish to establish the same property also for the joint law of the random variables~\eqref{eq:def_of_I_n} and~\eqref{eq:def_of_S_n}. Every isometry of $\widehat{\bC}$ is of the form
	\begin{equation}
		\label{eq:isometries_of_sphere}
		\tau(z) = \frac{\alpha z+\beta}{\overline{\alpha} - \overline{\beta}z}
	\end{equation} 
	where $\alpha,\beta\in \bC$ satisfy $|\alpha|^2+|\beta|^2=1$. Given such an isometry $\tau$, we set
	\begin{equation*}
		u_\tau(z) \stackrel{{\rm def}}{=} \bigg(\frac{\overline\alpha-\overline \beta z}{|\overline\alpha-\overline \beta z|} \bigg)^n\, .
	\end{equation*} 
	Recalling that $f_n$ is the random polynomial given by~\eqref{eq:def_random_polynomial}, it is straightforward to check that the random function $f_n^\tau$ defined by
	\begin{equation}
		\label{eq:distribution_of_polynomial_after_isometry}
		f_n^\tau(z) \stackrel{{\rm def}}{=} u_\tau(z)^{-1} \bigg(\frac{1+|z|^2}{1+|\tau(z)|^2}\bigg)^{n/2} f_n(\tau(z)) \, ,
	\end{equation}
	has the same law as $f_n$, (see \cite[\S2.3]{HKPV}). In particular, $f_n^\tau$ is  a polynomial. 
	 
	Suppose that $Z=(Z_1,Z_2)\in \bC^{a+b}$ is a mean-zero complex Gaussian random vector with $Z_1\in \bC^a$ and $Z_2\in \bC^b$, $a,b\ge 1$. The covariance matrix can be partitioned in the obvious way
	\[
	\Cov(Z_1,Z_2) = \begin{bmatrix}
	\Sigma_{11} & \Sigma_{12} \\ \Sigma_{12}^\ast & \Sigma_{22}
	\end{bmatrix}\, .
	\]
	We will frequently use the basic fact that on the event $\{Z_1 = \mathbf{0}\}$, where $\mathbf{0}=(0,\ldots,0)\in \bC^a$, the conditional law of $Z_2$ is also complex Gaussian in $\bC^b$ with mean-zero and covariance matrix
	$
	\Sigma_{22} - \Sigma_{12}^\ast \Sigma_{11}^{-1} \Sigma_{12}\, ,
	$
	(see for instance~\cite[Ex.~2.1.3]{HKPV} or~\cite[\S2.13, Theorem~2]{Shiryaev}).  
	
	\begin{lemma}
		\label{lemma:spherical_invariance}
		Let $\ell,m\in \bZ_{\ge 0}$ and suppose $w_1,\ldots,w_\ell,z_1,\ldots,z_m \in \bC$. Then the law of the random vector
		\begin{equation*}
			\big(|\widehat{f}_n(w_1)|,\ldots,|\widehat{f}_n(w_\ell)| , |D\widehat{f}_n(z_1)|,\ldots,|D\widehat{f}_n(z_m)|\big)
		\end{equation*}
		conditioned on the event $\big\{f(z_1)=\ldots=f(z_m)=0\big\}$ is rotation invariant. That is, for any isometry $\tau$ of the form~\eqref{eq:isometries_of_sphere}, the above law does not change if $z_1,\ldots,z_m$ are replaced by $\tau(z_1),\ldots,\tau(z_m)$ and $w_1,\ldots,w_\ell$ are replaced by $\tau(w_1),\ldots\tau(w_\ell)$. 
	\end{lemma}
	\begin{proof}
		Let $\tau$ be an isometry of the form~\eqref{eq:isometries_of_sphere}. Recall that our random polynomial $f_n$ has the same law as $f_n^\tau$ which is given by~\eqref{eq:distribution_of_polynomial_after_isometry}. Therefore, conditioned on the event $\big\{f(z_1)=\ldots=f(z_m)=0\big\}$, the Gaussian random vector
		\begin{equation*}
		\big(f(w_1),\ldots,f(w_\ell) , f^\prime(z_1),\ldots,f^\prime(z_m)\big)
		\end{equation*}
		has the same law as the Gaussian vector
		\begin{equation*}
		\big(f_n^\tau(w_1),\ldots,f_n^\tau(w_\ell) , (f_n^\tau)^\prime(z_1),\ldots,(f_n^\tau)^\prime(z_m)\big)
		\end{equation*}
		conditioned on the event $\big\{f_n(\tau(z_1)) = \ldots = f_n(\tau(z_m)) = 0  \big\}$. If $f_n(\zeta)=0$ for some point $\zeta\in \bC$, then 
		\[
		(f_n^\tau)^\prime(\zeta) = u_\tau(\zeta)^{-1} \bigg(\frac{1+|\zeta|^2}{1+|\tau(\zeta)|^2}\bigg)^{n/2} f_n^\prime(\tau(\zeta)) \, \tau^\prime(\zeta)\, .
		\] 
		Direct computation shows that
		\begin{equation*}
			\tau^\prime(z) = \frac{1}{\big(\overline{\alpha} - \overline{\beta}z\big)^2}\, , \qquad |\tau^\prime(z)| = \frac{1+|\tau(z)|^2}{1+|z|^2}\, .
		\end{equation*}
		Since $|u_\tau(z)| = 1$ for all $z\in \bC$, we conclude that the random vector
		\begin{equation*}
		\big(|f(w_1)|,\ldots,|f(w_\ell)| , |f^\prime(z_1)|,\ldots,|f^\prime(z_m)|\big)
		\end{equation*}
		conditioned on the event $\big\{f(z_1)=\ldots=f(z_m)=0\big\}$ have the same law as the random vector
		\begin{align*}
		\left(|f_n^\tau(w_1)|,\ldots,|f_n^\tau(w_\ell)| ,  \left(\frac{1+|z_1|^2}{1+|\tau(z_1)|^2}\right)^{n/2-1} |f_n^\prime(\tau(z_1))|, \ldots,\left(\frac{1+|z_m|^2}{1+|\tau(z_m)|^2}\right)^{n/2-1} |f_n^\prime(\tau(z_m))|\right)
		\end{align*}
		conditioned on the event $\left\{f_n(\tau(z_1)) = \ldots = f_n(\tau(z_m)) = 0  \right\}$. Recalling the definitions of $\widehat{f}_n$ and $D\widehat{f}_n$  concludes the proof of the lemma.
	\end{proof}
	\subsection{Kac-Rice type densities} 
	\label{sec:kac_rice_densities}
	Recall that $\mathcal{Z}$ is the zero set of our random polynomial $f_n$. For $m\in \bN$, we denote by $\mathcal{Z}^m$ the set of distinct $m$-tuples of zeros from $\mathcal{Z}$, with the agreement that $\mathcal{Z}^1 = \mathcal{Z}$. 
	
	We denote by $\mathcal{P}(k,m)$ the collection of all unordered partitions of the set $[k] = \{1,\ldots,k\}$ into $m$ non-empty blocks, and by $\mathcal{P}(k)$ the collection of all partitions of $[k]$. For $\pi\in \mathcal{P}(k,m)$, we denote the blocks by $\{\pi_1,\ldots,\pi_m\}$ and let $|\pi_\ell|$ be the size of the $\ell$'th block. For $\pi \in \mathcal{P}(k,m)$, we set
	\begin{equation}
		\label{eq:def_S_n_of_partition}
		\mathcal{S}_n^{(\pi)} \stackrel{{\rm def}}{=} \sum_{(\zeta_1,\ldots,\zeta_{m})\in \mathcal{Z}^{|\pi|}} \log^{|\pi_1|} |D\widehat{f}_n(\zeta_1)| \cdots \log^{|\pi_m|} |D\widehat{f}_n(\zeta_m)| \, .
	\end{equation} 
	The sums of the form~\eqref{eq:def_S_n_of_partition} appear when considering the moments of $\mathcal{S}_n$, since
	\begin{equation*}
		(\mathcal{S}_n)^k = \sum_{\pi \in \mathcal{P}(k)}\mathcal{S}_n^{(\pi)} \, .
	\end{equation*} 
	It is evident that $\mathcal{S}_n^{(\pi)}$, as a function of $\pi\in \mathcal{P}(k,m)$, depends only on the numbers $\{|\pi_1|,\ldots,|\pi_m|\}$, which leads us to consider densities of the following form.
	\begin{definition}
		\label{definition:def_of_kac_rice_densities}
		For $m,\ell\in \bZ_{\ge 0}$ and $p=(p_1,\ldots,p_m)\in \bZ_{\geq 0}^m$, we consider the function $\rho_{\ell,m,p}:\bC_{\not=}^\ell\times \bC_{\not=}^m \to \bR$ given by
		\begin{equation}
			\label{eq:def_of_kac_rice_densities}
			\rho_{\ell,m,p}(\mathbf{w},\mathbf{z}) \stackrel{{\rm def}}{=} n^{\ell+m} \frac{\La_{\ell,m,p}(\mathbf{w},\mathbf{z})}{\det \Big[\Cov\big(\widehat{f}_n(z_j)\big)_{j=1}^{m}\Big] } \, ,
		\end{equation}
		where $\mathbf{w}= (w_1,\ldots,w_\ell)$, $\mathbf{z}= (z_1,\ldots,z_m)$ and
		\begin{equation*}
			\La_{\ell,m,p}(\mathbf{w},\mathbf{z}) \stackrel{{\rm def}}{=} \bE\bigg[ \prod_{t=1}^{\ell} \log|\widehat{f}_n(w_t)| \cdot \prod_{j=1}^{m} \Big(\log^{p_j}|D\widehat{f}_n(z_j)|\cdot|D\widehat{f}_n(z_j)|^2\Big) \, \Big\vert \, f_n(z_j) = 0, \, \forall j\in[m]\bigg]\, .
		\end{equation*}
	\end{definition}
	The family of functions given by~\eqref{eq:def_of_kac_rice_densities} will allow us to compute the joint moments of $\mathcal{I}_n$ and $\mathcal{S}_n$.  We note that by plugging $\ell = 0$ and $p_1=\ldots=p_m=0$ in the above definition  we recover the usual $m$-point function for the point process $\mathcal{Z}$ on the sphere $\widehat{\bC}$ described in the introduction, see~\cite[\S3.3]{HKPV}.
	\begin{proposition}
		\label{prop:formula_for_joint_moments}
		For all $\ell,k\in \bZ_{\ge0}$ and $\pi\in \mathcal{P}(k,m)$ we have
		\begin{equation*}
			\bE\big[(\mathcal{I}_n)^{\ell} \, \mathcal{S}_n^{(\pi)}\big] = \int_{\bC^{\ell}} \int_{\bC^{m}} \rho_{\ell,m,p}(\mathbf{w},\mathbf{z}) \, {\rm d}\mu^{\otimes m}(\mathbf{z}) \, {\rm d}\mu^{\otimes \ell}(\mathbf{w})
		\end{equation*}
		where $p_j = |\pi_j|$ for $1\le j\le m$, and $\rho_{\ell,m,p}$ is given by Definition~\ref{definition:def_of_kac_rice_densities}.
	\end{proposition}  
	The proof of Proposition~\ref{prop:formula_for_joint_moments} is quite standard and follows the same scheme as in the usual Kac-Rice type results.  As such, we defer the proof to Appendix~\ref{sec:proof_of_kac_rice_formula}. With Proposition~\ref{prop:formula_for_joint_moments} at our disposal, we can already recover the main result from~\cite{ArmentanoBeltranShub}, that is, compute the expectation of $\mathcal{E}_n$. Although it will not play a role in what follows, we provide this computation for the sake of completeness.
	\begin{lemma}
		\label{lemma:expectation_of_E_n}
		Let $\mathcal{E}_n$ be the logarithmic energy~\eqref{eq:def_of_log_energy}. We have
		\[
		 \bE \big[\mathcal{E}_n\big] = \Big(\frac{1}{2} - \log 2\Big)n^2 - \frac{1}{2} n \log n  - \Big(\frac{1}{2} - \log 2\Big)n \, .
		\]
	\end{lemma}
	\begin{proof}
		Let $Z$ be a standard complex Gaussian. We have
		\begin{equation*}
			\bE\big[\log|Z|\big] = \int_{\bC} \log|z| \, e^{-|z|^2} \frac{{\rm d} m(z)}{\pi} = \frac{1}{2} \int_{0}^{\infty} (\log t) e^{-t} \, {\rm d}t = \frac{1}{2} \Ga^\prime(1) = - \frac{\gamma}{2} 
		\end{equation*}
		where $\Ga$ is the Gamma function and $\gamma$ is Euler's constant. Furthermore, as $\bE \big|\log|Z|\big| < \infty$, we can apply Fubini's theorem and get
		\begin{equation}
			\label{eq:expectation_of_I_n}
			\bE\big[\mathcal{I}_n\big] = n\int_{\bC} \bE \big[\log|\widehat{f}_n(z)|\big]\, {\rm d}\mu(z) = -n \frac{\gamma}{2} \, .
		\end{equation}
		To compute $\bE[\mathcal{S}_n]$ we will use Proposition~\ref{prop:formula_for_joint_moments} with $\ell = 0$ and $k=1$. For all $z\in \bC$ the Gaussian vector $\big( \widehat{f}_n(z),D\widehat{f}_n(z)\big)$ has mean zero and covariance matrix
		\begin{equation*}
			\begin{pmatrix}
			1 & z\sqrt{n} \\  \overline z \sqrt{n} & 1+n|z|^2
			\end{pmatrix}\, .
		\end{equation*}
		Therefore, the distribution of $D\widehat{f}_n(z)$ conditioned on the event that $\{f_n(z) = 0\}$ is that of a standard complex Gaussian. For $\La$ which is given by~\eqref{eq:def_of_kac_rice_densities}, we conclude that
		\begin{align*}
			\La_{0,1,1}(z) &= \bE\Big[ \log|D\widehat{f}(z)| \cdot |D\widehat{f}(z)|^2 \ \big\vert \, f(z) = 0 \Big]  = \frac{1}{\pi}\int_{\bC} (\log|z|) \, |z|^2 \, e^{-|z|^2} {\rm d}m(z)  \\
			&= \frac{1}{2}\int_{0}^{\infty} \log(t) \, t\,  e^{-t} \, \rm d t = \frac{1-\gamma}{2} \, .
		\end{align*}
		Proposition~\ref{prop:formula_for_joint_moments} now gives
		\begin{equation}
			\label{eq:expectation_of_S_n}
			\bE\big[\mathcal{S}_n\big] = n\frac{1-\gamma}{2} \int_{\bC} {\rm d} \mu (z) = n \, \frac{1-\gamma}{2} \, .
		\end{equation}
		 By linearity of the expectation, we combine~\eqref{eq:expectation_of_I_n} and~\eqref{eq:expectation_of_S_n} together with Lemma~\ref{lemma:rewrite_energy_as_I_minus_S} and get the desired formula for the expected energy.	
	\end{proof}
	To give effective bounds on the cumulants of $\mathcal{E}_n$, we will need to bound from above expectations of the form which appear in Proposition~\ref{prop:formula_for_joint_moments}. This, in turn, will follow from the clustering property of the densities $\rho_{\ell,m,p}$, as described in the introduction. 
	\subsection{Local bounds and clustering}
	Throughout this section, we fix $\ell,k\in \bZ_{\ge0}$ and $\pi\in \mathcal{P}(k,m)$. As before, we denote the length of the blocks of $\pi$ as $p_j = |\pi_j|$ for $1\leq j \le m$, and let $\rho_{\ell,m,p}$ be the density given by~\eqref{eq:def_of_kac_rice_densities}. The first ingredient we will need is the following local bound, at the scale of $n^{-1/2}$.  
	\begin{theorem}
		\label{thm:local_bounds_on_dentity}
		Let $n\ge 2$ and let $K\subset \bC$ be a compact set. There exist a constant $C=C(k,\ell,K)$ such that for all configurations of pairwise distinct points $(\mathbf{w},\mathbf{z})\in K^{\ell+m}$ we have 
		\[
		\Big|\rho_{\ell,m,p}\Big(\frac{\mathbf{w}}{\sqrt{n}},\frac{\mathbf{z}}{\sqrt{n}}\Big)\Big| \le C n^{\ell+m}\bigg( \prod_{1\le j \le m} \Delta_j |\log^{p_j} (\Delta_j)|\bigg) \cdot \bigg(\prod_{1\le t\le \ell} |\log(\Delta^\prime_t)|\bigg) 
		\]
		where
		\begin{equation}
			\label{eq:def_of_Delta_and_Delta_prime}
			\Delta_j \stackrel{{\rm def}}{=} \prod_{\substack{1\le j^\prime\le m \\ j^\prime \not= j}} |z_j - z_{j^\prime}|^2 \quad \text{and} \quad \Delta_t^{\prime} \stackrel{{\rm def}}{=} \prod_{1\le j\le m} \max\Big\{|w_t - z_j|^2, \frac{1}{2}\Big\} \, ,
		\end{equation}
		for $j\in \{1,\ldots,m\}$ and $t\in\{1,\ldots,\ell\}$.
	\end{theorem}
	\begin{remark} \label{remark:log-abs}
		As will be clear from the proof of Theorem~\ref{thm:local_bounds_on_dentity}, a slightly stronger statement is true; one may replace in the definition of $\rho_{\ell,m,p}$ all appearances of $\log$ with $|\log|$ and the same bound (as stated in Theorem~\ref{thm:local_bounds_on_dentity}) still holds. We will use this remark in Section~\ref{sec:clustering_for_random_measure}.
	\end{remark}
	
	To proceed with the clustering, we need to introduce several notations. Let $\mathbf{z} = (z_1,\ldots,z_m)$ and $\mathbf{w} = (w_1,\ldots,w_\ell)$ be vectors in $\bC_{\not=}^m$ and $\bC_{\not=}^\ell$, which will serve as the variables in the density $\rho_{\ell,m,p}$ given by~\eqref{eq:def_of_kac_rice_densities}. For a non-empty subset $I=\{i_1,\ldots,i_j\} \subset [m]$, we set $\mathbf{z}_I = \{z_{i_1} , \ldots,z_{i_j} \}$ and $p_I = \{p_{i_1},\ldots,p_{i_j}\}$. Similarly, for a non-empty subset $I^\prime = \{i_1^\prime,\ldots,i_t^\prime\} \subset [\ell]$, we set $\mathbf{w}_{I^\prime} = \{w_{i_1^\prime},\ldots, w_{i_t^\prime} \}$. 
	
	For any two vectors $A=(A_1,\ldots,A_m)$ and $B=(B_1,\ldots,B_\ell)$ with complex entries we define the \emph{spherical distance} between them as
	\begin{equation}
		\label{eq:def_spherical_distance_between_sets}
		\text{\sf dist} \{A,B\} \stackrel{{\rm def}}{=} \min\left\{d_{\bS^2}(a_i,b_j) \mid i \in [m], j \in [\ell]\right\}  \, .
	\end{equation}
	\begin{theorem}
		\label{thm:clustering_of_density}
		For all $\ell,m\in \bZ_{\ge 0}$ there exist constants $C,D,{ c}>0$ so that the following holds. Let $(\mathbf{w},\mathbf{z})$ be any configuration in $\bC_{\neq}^{\ell + m}$ and suppose we can partition $[m]$ and $[\ell]$ into non-empty subsets $I\sqcup J = [m]$ and $I^\prime\sqcup J^\prime= [\ell]$ such that
		\[
		d\stackrel{{\rm def}}{=} \text{\normalfont \sf dist} \left\{ \left(\frac{\mathbf{w}_{I^\prime}}{\sqrt n},\frac{\mathbf{z}_I}{\sqrt n}\right) , \left(\frac{\mathbf{w}_{J^\prime}}{\sqrt n},\frac{\mathbf{z}_J}{\sqrt n}\right) \right\} \ge \frac{D}{\sqrt{n}} \, .
		\]
		Then
		\begin{equation*}
		\Big|\rho_{\ell,m,p}\Big(\frac{\mathbf{w}}{\sqrt{n}},\frac{\mathbf{z}}{\sqrt{n}}\Big) - \rho_{|I^\prime|,|I|,p_I}\Big(\frac{\mathbf{w}_{I^\prime}}{\sqrt{n}},\frac{\mathbf{z}_I}{\sqrt{n}}\Big)\rho_{|J^\prime|,|J|,p_J}\Big(\frac{\mathbf{w}_{J^\prime}}{\sqrt{n}},\frac{\mathbf{z}_J}{\sqrt{n}}\Big)\Big|\le C n^{\ell+m} \bigg(\prod_{1\le t\le \ell} |\log(\Delta^\prime_t)|\bigg) \, { e^{-cnd^2}} \, .
		\end{equation*}
	\end{theorem}
		 \begin{remark*}
				By looking at the special case of Theorem~\ref{thm:clustering_of_density} with $\ell=0$, $m\ge 1$ and $p=(0,\ldots,0)$, we get the (quantitative) clustering property for the  usual $m$-point function of the point process of zeros $\cZ$ as described in the introduction,  see Theorem~\ref{thm:clustering_intro}. Proving this special case follows along the same lines of the proof of Theorem~\ref{thm:clustering_of_density}, which is provided in Section~\ref{sec:clustering_for_random_measure}, and is in fact easier as some technical issues arising from logarithmic terms can be avoided.
			\end{remark*} 
	\subsection{Cumulants}
	Recall that for a real random variable $X$, the $k$th cumulant $s_k(X)$ is defined\footnote{For this definition of cumulants, one must assume that the variable has a finite exponential moment; one may define the $k$th cumulant assuming the variable only has a finite $k$th moment via \eqref{eq:def_joint_cumulant}.}  via the formal relation
	\begin{equation*}
		\log \bE\big[e^{tX}\big] = \sum_{k\ge 1} \frac{s_k(X)}{k!} \, t^k \, .
	\end{equation*}
	For random variables $X_1,\ldots,X_N$ we define their (simple) joint cumulant as
	\begin{equation}
		\label{eq:def_joint_cumulant}
		s(X_1,\ldots,X_N) \stackrel{{\rm def}}{=} \sum_{\pi\in \mathcal{P}(N)} (|\pi|-1)! \, (-1)^{|\pi|-1} \bigg(\prod_{I\in \pi} \bE\Big[\prod_{i\in I} X_i\Big]\bigg)	
	\end{equation}
	provided that the right hand-side is finite (here $|\pi|$ is the number of blocks in the partition $\pi$). In fact, \eqref{eq:def_joint_cumulant} is just the homogeneous coefficient of degree $1$ in the Taylor expansion of the function 
	\begin{equation}\label{eq:cumulants-gen-fcn}
	(t_1,\ldots,t_N) \mapsto \log \bE\big[e^{t_1 X_1 +\ldots + t_N X_N}\big]
	\end{equation}
	around the origin, see~\cite[Chapter~II, \S12]{Shiryaev}. The inverse formula is
	\begin{equation}
		\label{eq:inversion_formula_moments_and_joint_cumulants}
		\bE\left[X_1\cdots X_N\right] = \sum_{\pi \in \mathcal{P}(N)} \prod_{I\in \pi} s\left(\{X_i\}_{i\in I}\right) \, .
	\end{equation}
	The proof of Theorem~\ref{thm:main_result_clt} is based on bounds on the cumulants of $\mathcal{E}_n$, which in turn is based on bounds on the joint cumulants of $\mathcal{I}_n$ and $\mathcal{S}_n$. For $a,b\in \bZ_{\ge 0}$ we set
	\begin{equation}
		\label{eq:def_joint_cummulant_gamma_a_b}
		\gamma_{n}(a,b) \stackrel{{\rm def}}{=} s\big(\underbrace{\mathcal{S}_n,\ldots,\mathcal{S}_n}_{a \, \text{times}} \, , \, \underbrace{\mathcal{I}_n,\ldots,\mathcal{I}_n}_{b \, \text{times}}\big)
	\end{equation}
	where the joint cumulant is given by~\eqref{eq:def_joint_cumulant}. Using Proposition~\ref{prop:formula_for_joint_moments} and the clustering properties for $\rho_{\ell,m,p}$ given by Theorems~\ref{thm:local_bounds_on_dentity} and~\ref{thm:clustering_of_density}, we will prove the following theorem in Section~\ref{sec:cumulant_densities}. 
	\begin{theorem}
		\label{thm:limit_joint_cumulant}
		For $a,b\in \bZ_{\ge 0}$ and $\gamma_{n}(a,b)$ given by~\eqref{eq:def_joint_cummulant_gamma_a_b} there is a constant $c_{a,b} \in \bR$ so that 		\begin{equation*}
			\lim_{n\to \infty} \frac{\gamma_{n}(a,b)}{n} = c_{a,b}\,.
		\end{equation*}
	\end{theorem}
	The proof of Theorem \ref{thm:limit_joint_cumulant} in fact provides an expression for the constants $c_{a,b}$ (see Remark \ref{remark:cumulants}).  
	Combining Theorems~\ref{thm:variance_after_split} and~\ref{thm:limit_joint_cumulant}, proves Theorem~\ref{thm:main_result_clt}.
	\begin{proof}[Proof of Theorem~\ref{thm:main_result_clt}]
		Since
		\[
		{\sf Var} \big[\mathcal{I}_n - \mathcal{S}_n\big] = {\sf Var} \big[\mathcal{I}_n\big] + {\sf Var} \big[\mathcal{S}_n\big] - 2 \, {\sf Cov} \big[\mathcal{I}_n,\mathcal{S}_n\big] \, ,
		\]
		we can combine Lemma~\ref{lemma:rewrite_energy_as_I_minus_S} with Theorem~\ref{thm:variance_after_split} and get that
		\begin{equation*}
			\lim_{n\to \infty} \frac{{\sf Var}\big[\mathcal{E}_n\big]}{n} = c_{\ast} > 0\, .
		\end{equation*}
		For all $k\ge 2$, \eqref{eq:cumulants-gen-fcn} implies a binomial identity for cumulants
		\begin{equation*}
			s_k\big(\mathcal{E}_n\big) = s_k\big(\mathcal{I}_n- \mathcal{S}_n\big) = \sum_{j=0}^{k}\binom{k}{j} (-1)^{j} \gamma_{n}\big(j,k-j\big)
		\end{equation*}
		where $\gamma_{n}$ is given by~\eqref{eq:def_joint_cummulant_gamma_a_b}. Theorem~\ref{thm:limit_joint_cumulant} implies that
		\begin{equation*}
			\lim_{n\to \infty} \frac{s_k\big(\mathcal{E}_n\big)}{n}
		\end{equation*}
		exists and is finite for all $k\ge 2$. By the scaling relation $s_k\big(n^{-1/2} \mathcal{E}_n\big) = n^{-k/2} s_k\big(\mathcal{E}_n\big)$, we see that for all $k\ge 3$
		\[
		\lim_{n\to \infty} s_k\big(n^{-1/2}\mathcal{E}_n\big) = 0\, . 
		\]
		By the method of moments, we conclude that the sequence $\big(\mathcal{E}_n - \bE\big[\mathcal{E}_n\big]\big)/\sqrt{n}\,$ converges in distribution to the Gaussian law of mean $0$ and variance $c_\ast$ as $n\to \infty$ and the proof is complete.
	\end{proof}

	\section{Local bounds on the density. Proof of Theorem~\ref{thm:local_bounds_on_dentity}}
	\label{sec:local_bounds}
	\subsection{Local limit: Understanding the G.E.F}
	Recall that the Gaussian Entire Function (abbreviated G.E.F.) is the random Taylor series given by
	\begin{equation}
		\label{eq:def_of_GEF}
		g(z) \stackrel{{\rm def}}{=} \sum_{j\ge 0} \frac{a_j}{\sqrt{j!}} z^j \, ,
	\end{equation}
	where $\{a_j\}$ are i.i.d.\ standard complex Gaussians. First, we show the standard fact that $g(\cdot)$ is the limit of $f(\cdot/\sqrt{n})$ in a suitable sense. 
	\begin{claim}
		\label{claim:local_convergence_of_polynomials_to_GEF}
		For each $k,\ell\in \bZ_{\ge0}$ and $z,w\in \bC$ we have
		\[
		\lim_{n\to \infty} n^{-(k+\ell)/2} \bE\bigg[f_n^{(k)}\Big(\frac{z}{\sqrt{n}}\Big) \overline{f_n^{(\ell)}\Big(\frac{w}{\sqrt{n}}\Big)}\bigg] = \bE\Big[ g^{(k)}(z) \overline{g^{(\ell)} (w)}\Big] \, .
		\]
		Furthermore, the above convergence is uniform for $z,w$ in a compact set.
	\end{claim}	
	\begin{proof}
		We have
		\[
		\bE\bigg[f_n\Big(\frac{z}{\sqrt{n}}\Big) \overline{f_n\Big(\frac{w}{\sqrt{n}}\Big) }\bigg] = K_n\Big(\frac{z}{\sqrt{n}},\frac{w}{\sqrt{n}}\Big) \stackrel{\eqref{eq:covariance_kernel_for_f}}{=} \Big(1+\frac{z\overline{w}}{n}\Big)^n \xrightarrow{n\to \infty} e^{z\overline{w}} 
		\] 
		uniformly for $z,w$ in a compact set. It remains to note that $\bE\big[g(z) \overline{g(w)}\big] = \exp(z\overline{w})$, which is evident from~\eqref{eq:def_of_GEF}, to establish the case $k=\ell=0$ of the claim. Note that both $K_n(z,w)$ and $\exp(z\overline w)$ are holomorphic in $z$ and anti-holomorphic in $w$, and thus a simple application of Cauchy's theorem gives the desired statement for $k,\ell \in \mathbb{Z}_{\ge 0}$.
	\end{proof}
	In view of Claim~\ref{claim:local_convergence_of_polynomials_to_GEF}, to prove the local bounds on the density $\rho=\rho_{\ell,m,p}$, it will be sufficient to prove the analogous statement for the G.E.F.\ replacing our original polynomial. We denote by $\widehat{g}(z) = e^{-|z|^2/2} g(z)$ and $D \widehat{g}(z) = e^{-|z|^2/2} g^\prime(z)$.
	\begin{definition}
		For $m,\ell\in \bZ_{\ge 0}$ and $p=(p_1,\ldots,p_m)\in \bN^m$, let
		\begin{equation}
		\label{eq:def_of_kac_rice_densities_GEF}
		\rho_{\ell,m,p}^{G}(\mathbf{w},\mathbf{z}) = \frac{\La_{\ell,m,p}^{G}(\mathbf{w},\mathbf{z})}{\det \big[\Cov\big(\widehat{g}(z_j)\big)_{j=1}^{m}\big] } \, ,
		\end{equation}
		where $\mathbf{w}= (w_1,\ldots,w_\ell)$, $\mathbf{z}= (z_1,\ldots,z_m)$ and
		\begin{equation*}
		\La_{\ell,m,p}(\mathbf{w},\mathbf{z}) \stackrel{{\rm def}}{=} \bE\bigg[ \prod_{t=1}^{\ell} \log|\widehat{g}(w_t)| \prod_{j=1}^{m} \Big(\log^{p_j}|D \widehat{g}(z_j)|\, |D\widehat{g}(z_j)|^2\Big) \, \Big\vert \, g(z_j) = 0, \, \forall j\in[m]\bigg]\, .
		\end{equation*}
	\end{definition}
	The main result of this section, Theorem~\ref{thm:local_bounds_on_dentity}, will follow easily from the analogous statement for the G.E.F.
	\begin{theorem}
		\label{thm:local_bounds_density_for_GEF}
		Let $K\subset \bC$ be a compact set. There exist a constant $C=C(m,\ell,p,K)$ such that for all configurations of pairwise distinct point $(\mathbf{w},\mathbf{z})\in K^{\ell+m}$ we have 
		\[
		|\rho_{\ell,m,p}^{G}(\mathbf{w},\mathbf{z})| \le C \bigg( \prod_{1\le j \le m} \Delta_j \, |\log^{p_j} (\Delta_j)|\bigg) \cdot \bigg(\prod_{1\le t\le \ell} |\log(\Delta^\prime_t)|\bigg)  \, .
		\]
		where $\rho_{\ell,m,p}^{G}$ is given by~\eqref{eq:def_of_kac_rice_densities_GEF} and $\Delta_j, \Delta^\prime_t$ are given by~\eqref{eq:def_of_Delta_and_Delta_prime}.
	\end{theorem}
	\subsection{Divided differences}\label{ss:divided-differences}
	For an entire function $f:\bC\to \bC$ and a collection of distinct points $\mathbf{z}=(z_1,\ldots,z_m)$ define the divided difference of $f$ with respect to $\mathbf{z}$ via
	\begin{equation}
		\label{eq:def_DD_for_f}
		f[z_1,\ldots,z_m]  \stackrel{{\rm def}}{=} \sum_{j=1}^{m} \frac{f(z_j)}{\prod_{j^\prime\not=j} (z_j - z_{j^\prime})}\, .
	\end{equation} 
	The divided difference of $f$ are in fact the coefficients in Newton's form for polynomial interpolation, see~\cite{MilneThomoson} for basic facts on the divided difference. 
	
	It will be convenient for us later on to write~\eqref{eq:def_DD_for_f} in a matrix form. Indeed, for distinct $z_1,\ldots,z_m \in \bC$ we consider the lower triangular $m\times m$ matrix
	\begin{equation*}
		M(\mathbf{z}) \stackrel{{\rm def}}{=} \bigg( \prod_{j=1}^{i-1} (z_j-z_i) \bigg)_{1\le i,j\leq m} 
	\end{equation*}
	with the agreement that the empty product is equal to 1. It is not hard to check that~\eqref{eq:def_DD_for_f} can be written as the matrix product
	\begin{equation}
		\label{eq:DD_in_matrix_form}
		\begin{pmatrix}
		f(z_1) \\ f(z_2) \\ \vdots \\ f(z_m)
		\end{pmatrix} = M(\mathbf{z}) \begin{pmatrix}
		f[z_1] \\ f[z_1,z_2] \\ \vdots \\ f[z_1,\ldots,z_m]
		\end{pmatrix}
	\end{equation}
	see for instance~\cite[Lemma~2.13]{Gass}. Since $M(\mathbf{z})$ is lower triangular, we have
	\[
	\det M(\mathbf{z}) = \prod_{i<j} (z_j-z_i) \, . 
	\]
	Suppose now that $z_1,\ldots,z_m \in K$ where $K\subset \bC$ is a compact set and let $\ga$ be a simple closed contour that bounds a simply connected domain containing $K$. A simple residue computation shows that
	\begin{equation}
		\label{eq:DD_as_cauchy_integral}
		f[z_1,\ldots,z_m] = \frac{1}{2\pi {\rm i}} \int_{\ga} \frac{f(z)}{(z-z_1)\cdots(z-z_m)}\, {\rm d}z\, .
	\end{equation}
	Using~\eqref{eq:DD_as_cauchy_integral}, we extend by continuity the definition~\eqref{eq:def_DD_for_f} of $f[z_1,\ldots,z_m]$ to non-distinct tuples. Namely, for a tuple of points $z_1,\ldots,z_m\in K$ with $z_j$ appearing $n_j$ times for $j=1,\ldots,m$, we have
	\begin{equation*}
		f[z_1,\ldots,z_m] = \frac{1}{2\pi {\rm i}} \int_{\ga} \frac{f(z)}{(z-z_1)^{n_1}\cdots(z-z_m)^{n_m}}\, {\rm d}z\, .
	\end{equation*}
	Note that the mapping $(z_1,\ldots,z_m) \mapsto f[z_1,\ldots,z_m]$ is smooth and that
	\[
	f[\underbrace{z,\dots,z}_{m \, \text{times}}] = \frac{f^{(m-1)}(z)}{(m-1)!} \, .
	\]

		\begin{claim}\label{claim:converge-DD}
			Let $K\subset \bC$ be a compact set. For any $z_1,\ldots,z_m,w_{1},\ldots,w_{\ell} \in K$ we have \begin{equation*}
				\lim_{n\to\infty}
				n^{-(m+\ell - 2)/2} \, \bE \bigg[f_n\Big[\frac{z_1}{\sqrt{n}}, \ldots,\frac{z_m}{\sqrt{n}}  \Big] \overline{f_n\Big[\frac{w_1}{\sqrt{n}}, \ldots,\frac{w_\ell}{\sqrt{n}}  \Big]} \bigg] = \bE \left[ g[z_1,\ldots,z_k] \overline{g[w_1,\ldots,w_\ell]}\right]
			\end{equation*}
		where the convergence is uniform in $K$.
		\end{claim}
		\begin{proof}
			Let $\gamma$ be a simple closed contour whose interior contain $K$. We have $$f_n\Big[\frac{z_1}{\sqrt{n}}, \ldots,\frac{z_m}{\sqrt{n}}  \Big] =\frac{ n^{(m-1)/2}}{2\pi {\rm i}} \int_{\gamma}f_n\Big(\frac{z}{\sqrt{n}}\Big) \frac{{\rm d}z}{\prod_j(z - z_j)}\,.$$
			By writing a similar expression for the divided difference along $(w_j/\sqrt{n})$ and applying Claim~\ref{claim:local_convergence_of_polynomials_to_GEF} we get the desired claim
		\end{proof}
	To rewrite the densities in terms of the divided differences, we need to understand the effect of the conditioning on the divided differences. This understanding is made clear by the following simple Lemma~\ref{lemma:DD_conditioned_on_roots}. We note that a similar statement also appears in Gass~\cite[Lemma~2.11]{Gass}.
	\begin{lemma}
		\label{lemma:DD_conditioned_on_roots}
		Let $f:\bC\to \bC$ be an entire function and let $z_1,\ldots,z_m,y \in \bC$. Suppose that $f(z_1) = f(z_2) = \ldots = f(z_m) = 0$, then
		\begin{equation*}
			f(y) = f[z_1,\ldots,z_m,y] \prod_{\ell =1}^{m} (y-z_\ell) \quad \text{and} \quad f^\prime(z_1) = f[z_1,\ldots,z_m, z_1]\prod_{\ell =2}^{m} (z_1 - z_\ell) \, .
		\end{equation*}
	\end{lemma}
	\begin{proof}
		When the points $z_1,\ldots,z_m,y$ are distinct, the first equality follows immediately from~\eqref{eq:def_DD_for_f}. When the points $z_1,\ldots,z_m$ are not distinct, then the equality follows from continuity of $f[z_1,\ldots,z_m,y]$ as seen from~\eqref{eq:DD_as_cauchy_integral}. To see the second equality, simply apply the first with $y=z_1+h$ for some $h\not=0$, divide both sides by $h$ and then take $h\to 0$.
	\end{proof}

	Our next observation, which is the key step in proving Theorem~\ref{thm:local_bounds_density_for_GEF}, is that the Gaussian vector of divided differences for the G.E.F. is uniformly non-singular and of bounded variance.
	\begin{lemma}
		\label{lemma:DD_for_GEF_non_singular}
		Let $K\subset \bC$ be a compact set and let $g$ given by~\eqref{eq:def_of_GEF}. There exist constants $c=c(K,m)>0$ and $C = C(K,m)$ such that for all points $z_1,\ldots,z_m\in K$ we have
		\begin{equation}\label{eq:g-cov-structure}
			\det\bigg[ \Cov \Big(g[z_1,\ldots,z_j]\Big)_{j=1}^{m} \bigg] \ge c\,,  \qquad \bE |g[z_1,\ldots,z_m]|^2 \leq C  
		\end{equation} 
		and for all $n$ we have 		\begin{equation}\label{eq:f_n-cov-structure}
			\det\bigg[ \Cov \left(f_n\left[\frac{z_1}{\sqrt{n}},\ldots,\frac{z_j}{\sqrt{n}}\right]\right)_{j=1}^{m} \bigg] \ge cn^{m(m-1)/2}\,,  \qquad \bE \left|f_n\left[\frac{z_1}{\sqrt{n}},\ldots,\frac{z_m}{\sqrt{n}}\right]\right|^2 \leq C n^{{m-1}}
		\end{equation} 
	\end{lemma}
	\begin{proof}
		Note that assuming the bounds in \eqref{eq:g-cov-structure} hold for some constant $c$ the bounds in \eqref{eq:f_n-cov-structure} hold for some smaller constant $c$ by Claim~\ref{claim:converge-DD}.  Thus we focus our attention on \eqref{eq:g-cov-structure}.
		
		For the latter bound in \eqref{eq:g-cov-structure}, choose the contour $\gamma$ in \eqref{eq:DD_as_cauchy_integral} to be of distance, say, at least $1$ from $K$ and apply the triangle inequality.  To prove the former bound, 
		seeking a contradiction assume there exist $z_1,\ldots,z_m\in K$ such that
		\[
		\det\bigg[ \Cov \Big(g[z_1,\ldots,z_j]\Big)_{j=1}^{m} \bigg] = 0\, .
		\]
		Thus, there exist complex numbers $\alpha_1,\ldots,\alpha_m$, not all of which are zero, such that
		\begin{equation*}
			\sum_{\ell =1}^{m} \alpha_\ell \, g[z_1,\ldots,z_\ell] = 0 \qquad \text{almost surely.}
		\end{equation*}
		That is, $\displaystyle \int_{\ga} g(z) \, r(z) \, {\rm d}z = 0$, where
		\[
		r(z) \stackrel{{\rm def}}{=} \frac{1}{2\pi{\rm i}} \sum_{j=1}^{m} \frac{\alpha_j}{(z-z_1)\cdots(z-z_j)}\, .
		\]
		The definition~\eqref{eq:def_of_GEF} of $g$ implies that
		\begin{equation}
			\label{eq:coefficients_of_rational_function_in_taylor}
			\int_{\ga} z^j \,r(z) \, {\rm d}z =0 \qquad \text{for all } j\ge 0 \, ,
		\end{equation}
		but since $r(z)$ is a non-trivial rational function that vanishes at infinity we get a contradiction. The fact that we can choose a uniform lower bound throughout $K$ follows from the fact that $$(z_1,\ldots,z_m) \mapsto \det\bigg[ \Cov \Big(g[z_1,\ldots,z_j]\Big)_{j=1}^{m} \bigg]$$ is continuous.
 	\end{proof}
	\subsection{Proof of Theorem~\ref{thm:local_bounds_on_dentity}}

	We are now ready to prove Theorem \ref{thm:local_bounds_on_dentity}.  The idea is to rewrite all terms in $\rho$ in terms of appropriately normalized divided differences of $f_n$ and then { use} Lemma \ref{lemma:DD_for_GEF_non_singular} to handle the resulting quantities.

	\begin{proof}[Proof of Theorem~\ref{thm:local_bounds_on_dentity}]
		
		Throughout this proof, for a complex number $z$ we write $\zhat$ for $z/\sqrt{n}$, and extend this notation to vectors.  We also note that for $z \in K$ we have $(1 + |\zhat|^2)^n = \Theta(1)$, and so we will safely ignore these corrective multiplicative terms and only lose a constant depending on $K$.   Starting with the denominator, the multiplicativity of the determinant together with~\eqref{eq:DD_in_matrix_form} implies that
		\begin{align*}
			\det \left[\Cov\left({f}_n(\zhat_j)\right)_{j=1}^{m}\right] &\gtrsim  \left|\det \left[M(\bzhat)\right]\right|^2 \det\left[ \Cov \left(f_n[\zhat_1,\ldots,\zhat_j]\right)_{j=1}^{m} \right]\,.
		\end{align*}
	We { also note that} \begin{equation}
		|\det[M(\bzhat)]| = n^{-m(m-1)/4} |\det[M(\mathbf{z})]|
	\end{equation}
and so, Lemma~\ref{lemma:DD_for_GEF_non_singular} gives the bound
		\begin{align}
			\label{eq:first_bound_on_rho_GEF}
			\frac{\left|\rho_{\ell,m,p}\left(\bwhat,\bzhat\right)\right|}{n^{m + \ell}}  &\lesssim \frac{\bE\left[ \prod_{t=1}^{\ell} \big|\log|\widehat{f}_n(\what_t)|\big| \prod_{j=1}^{m} \left(\big|\log^{p_j}| D\widehat{f}_n(\zhat_j)|\big||D\widehat{f}_n(\zhat_j)|^2\right) \, \big\vert \, f_n(\zhat_j) = 0\right]}{\left|\det [M(\mathbf{z})]\right|^2}
		\end{align}
		By Lemma~\ref{lemma:DD_conditioned_on_roots}, conditioned on the event that $\big\{f_n(z_j) = 0, \, \forall j \in [m] \big\}$, we have that
		\begin{align}
			\label{eq:product_of_GEF_derivatives_is_determinant}
			\prod_{j=1}^{m} |D\widehat{f}_n(\zhat_j)|^2 \nonumber &\lesssim n^{-m} \prod_{j=1}^{m} |f_n'(\zhat_j)|^2 = n^{-m}\prod_{j =1}^{m} \Big(\left|f_n[\zhat_1,\ldots,\zhat_m,\zhat_j]\right|^2 \prod_{j^\prime\not= j} |\zhat_j - \zhat_{j^\prime}|^2\Big) \\ \nonumber
			&= n^{-m} \left|\det[M(\bzhat)]\right|^4  \prod_{j =1}^{m} \left|f_n[\zhat_1,\ldots,\zhat_m,\zhat_j]\right|^2 \\
			&= \left|\det[M(\mathbf{z})]\right|^4  \prod_{j =1}^{m} \left|\frac{f_n[\zhat_1,\ldots,\zhat_m,\zhat_j]}{n^{m/2}}\right|^2  .
		\end{align}
		Furthermore, on the event that $\big\{f_n(\zhat_j) = 0, \, \forall j \in [m] \big\}$, Lemma~\ref{lemma:DD_conditioned_on_roots} also gives that
		\begin{align*}
			\log|\widehat{f}_n(\what_t)| &= \log\left|\frac{f_n[\zhat_1,\ldots,\zhat_m,\what_t]}{n^{m/2}}\right| + \log \prod_{j=1}^{m} |w_t - z_j| + \Theta(1) \\ \log^{p_j} |D\widehat{f}_n(\zhat_j)| &= \bigg(   \log\left|\frac{f_n[\zhat_1,\ldots,\zhat_m,\zhat_j]}{n^{m/2}}\right|   + \log\prod_{j^\prime\not=j} |z_j-z_{j^\prime}| + \Theta(1) \bigg)^{p_j} \, .
		\end{align*}
		Combining the above with~\eqref{eq:product_of_GEF_derivatives_is_determinant} with the elementary inequality $(|\alpha| + |\beta|)^p\le 2^{p-1}(|\alpha|^p + |\beta|^p)$ for all $p\ge1$ yields
		\begin{align*}
			\bE\bigg[& \prod_{t=1}^{\ell} \left|\log|\widehat{f}_n(\what_t)|\right| \prod_{j=1}^{m} \left(\big|\log^{p_j}| D\widehat{f}_n(\zhat_j)|\big| \, |D\widehat{f}_n(\zhat_j)|^2\right) \, \big\vert \, f_n(\zhat_j) = 0\bigg]			\\ 
			&\le 2^{p_1+\ldots+p_m} \left|\det[M(\mathbf{z})]\right|^4 \, \bE\Bigg[\prod_{t=1}^{\ell}\left(\Big|\log\Big|\frac{f_n[\mathbf{z},w_t]}{n^{m/2}}\Big|\Big| + |\log(\Delta_t^\prime)| + \Theta(1)\right) \\ &    \qquad \times\prod_{j=1}^{m} \left(\left(\Big|\log^{p_j}\Big| \frac{f_n[\mathbf{z},z_j]}{n^{m/2}}\Big|\Big|+\left|\log^{p_j}(\Delta_{j})\right| + \Theta(1)\right) \left|\frac{f_n[\mathbf{z},z_j]}{n^{m/2}}\right|^2 \right)\, \Big\vert \, f_n(z_j) = 0, \, \forall j\Bigg]\, ,
		\end{align*}
		where we used the notation $f_n[\mathbf{z},y] \stackrel{{\rm def}}{=} f_n[z_1,\ldots,z_m,y]$ for $y\in \bC$. Finally, by Lemma~\ref{lemma:DD_for_GEF_non_singular}, we know that the Gaussian vector of divided differences is uniformly non-degenerate in a compact set; we can thus expand out the product and bound each term by H\"older's inequality. This yields the existence of a constant $C=C(K,\ell,m,p)>0$ so that
		\begin{align*}
			\bE\bigg[ \prod_{t=1}^{\ell} \big|\log|\widehat{f}_n(\what_t)|\big|& \prod_{j=1}^{m} \left(\big|\log^{p_j}| D\widehat{f}_n(\zhat_j)|\big| \,|D\widehat{f}_n(\zhat_j)|^2\right) \, \Big\vert \, f_n(\zhat_j) = 0\bigg]	 \\ &\le C \left|\det[M(\mathbf{z})]\right|^4 \,\bigg( \prod_{1\le j \le m} |\log^{p_j} (\Delta_j)|\bigg) \cdot \bigg(\prod_{1\le t\le \ell} |\log(\Delta^\prime_t)|\bigg) \, .
		\end{align*}
		Plugging the above inequality into~\eqref{eq:first_bound_on_rho_GEF}, we get that
		\begin{align*}
			|\rho_{\ell,m,p}(\mathbf{w},\mathbf{z})| & \le C \left|\det[M(\mathbf{z})]\right|^2 \,\bigg( \prod_{1\le j \le m} |\log^{p_j} (\Delta_j)|\bigg) \cdot \bigg(\prod_{1\le t\le \ell} |\log(\Delta^\prime_t)|\bigg) \\ & = C\bigg( \prod_{1\le j \le m} \Delta_j \, |\log^{p_j} (\Delta_j)|\bigg) \cdot \bigg(\prod_{1\le t\le \ell} |\log(\Delta^\prime_t)|\bigg)
		\end{align*}
		as desired.
	\end{proof}

	\section{Clustering of the densities. Proof of Theorem~\ref{thm:clustering_of_density}}
	\label{sec:clustering_for_random_measure}
	The starting point is to write the density~\eqref{eq:def_of_kac_rice_densities} in a slightly more convenient form. Denote by $\varphi_n:\bC^{2m+\ell} \to \bR_{\ge0}$ the joint density for the complex Gaussian random variables
	\begin{equation}
		\label{eq:complex_gaussians_we_consider}
		\widehat{f}_n\Big(\frac{z_1}{\sqrt{n}}\Big), \ldots, \widehat{f}_n\Big(\frac{z_m}{\sqrt{n}}\Big), \, D\widehat{f}_n\Big(\frac{z_1}{\sqrt{n}}\Big),\ldots , D\widehat{f}_n\Big(\frac{z_m}{\sqrt{n}}\Big), \, \widehat{f}_n\Big(\frac{w_1}{\sqrt{n}}\Big),\ldots,\widehat{f}_n\Big(\frac{w_\ell}{\sqrt{n}}\Big) ,
	\end{equation}
	that is
	\begin{equation}
		\label{eq:gaussian_density_formula}
		\varphi_n(\eta_1,\eta_2,\eta^\prime) = \frac{1}{\pi^{2m+\ell} \det \Ga_{n}} \exp\Big( - \avg{\Ga_{n}^{-1} \widetilde\eta, \widetilde\eta } \Big)
	\end{equation}
	where $\widetilde\eta=(\eta_1,\eta_2,\eta^\prime)^{{\sf T}}\in \bC^{2m+\ell}$ with $\eta_1,\eta_2\in \bC^m$, $\eta^\prime\in \bC^\ell$ and $\Ga_{n}$ is the covariance matrix of the (complex) Gaussian variables~\eqref{eq:complex_gaussians_we_consider}. The density~\eqref{eq:def_of_kac_rice_densities} is then given by
	\begin{equation}
		\label{eq:Kac_Rice_densities_non_conditional}
		\rho_{\ell,m,p} \Big(\frac{\mathbf{w}}{\sqrt{n}},\frac{\mathbf{z}}{\sqrt{n}}\Big) = n^{\ell+m} \int_{\bC^{m+\ell}} \varphi_n(\mathbf{0},\eta,\eta^\prime)  \Big(\prod_{t=1}^{\ell} \log|\eta_t^\prime| \Big) \Big(\prod_{j=1}^{m} |\eta_j|^2 \, \log^{p_j}|\eta_j|\Big)  \, {\rm d}m(\eta,\eta^\prime) \, ,
	\end{equation}
	where $\mathbf{0}=(0,\ldots,0)$ and $m$ is the Lebesgue measure on $\bC^{m+\ell}$. The goal of this section is to prove that whenever the points $\mathbf{w},\mathbf{z}$ can be split into two separated sets, then the Gaussian density~\eqref{eq:gaussian_density_formula} approximately factors, in the precise technical sense as stated in Theorem~\ref{thm:clustering_of_density}. To achieve this, we start by proving that the covariance matrix $\Ga_{n}=\Ga_{n}(\mathbf{w},\mathbf{z})$ factors in the operator sense.
	
	\subsection{Preliminaries}
	Fix for the moment a configuration of points $\mathbf{z}\in \bC_{\not=}^m$  and $\mathbf{w}\in \bC_{\not=}^\ell$. For $\alpha = (\alpha^1,\alpha^2,\alpha^\prime)^{\sf T}\in \bC^{2m+\ell}$ we consider the random variable
	\begin{equation}
		\label{eq:def_of_L_n}
		L_n = L_n(\alpha) \stackrel{{\rm def}}{=} \sum_{j=1}^{m}\Big[\alpha_j^1 \, \widehat{f}_n\Big(\frac{z_j}{\sqrt{n}}\Big)+\alpha_j^2 \, D\widehat{f}_n\Big(\frac{z_j}{\sqrt{n}}\Big) \Big] + \sum_{t=1}^{\ell} \alpha_t^\prime\, \widehat{f}_n\Big(\frac{w_t}{\sqrt{n}}\Big) \, .
	\end{equation}
	The relevance of these random variables becomes clear by the following simple claim.
	\begin{claim}
		\label{claim:L_n_represents_Ga_n_in_L_2}
		For all $\alpha\in \bC^{2m+\ell}$ we have $\avg{\Ga_{n} \, \alpha,\alpha} = \bE|L_n(\alpha)|^2$\, .
	\end{claim}
	\begin{proof}
		This follows simply from the fact that $\Ga_n$ is the covariance matrix of~\eqref{eq:complex_gaussians_we_consider} and a direct inspection of $\bE|L_n(\alpha)|^2$.
	\end{proof}
	Since $\widehat{f}(\cdot/\sqrt{n})$ converges to the normalized G.E.F.\ as $n\to \infty$, we are led to consider the limiting random variable
	\begin{equation}
		\label{eq:def_of_L_infty}
		L_\infty(\alpha) \stackrel{{\rm def}}{=} \sum_{j=1}^{m}\big[\alpha_j^1 \, \widehat{g}(z_j)+\alpha_j^2 \, D\widehat{g}(z_j) \big] + \sum_{t=1}^{\ell} \alpha_t^\prime\, \widehat{g}(w_t)
	\end{equation}	
	where, as before, $\widehat{g}= e^{-|z|^2/2} g(z)$ and $D \widehat{g}(z) = e^{-|z|^2/2} g^\prime(z)$ with $g$ given by~\eqref{eq:def_of_GEF}.
	\begin{claim}
		\label{claim:L_n_converge_to_L_infty}
		We have $\displaystyle \lim_{n\to\infty} \bE|L_n(\alpha)|^2 = \bE|L_\infty(\alpha)|^2$, with uniform convergence for $(\mathbf{w},\mathbf{z})$ inside a compact set.
 	\end{claim}
 	\begin{proof}
 		By Claim~\ref{claim:L_n_represents_Ga_n_in_L_2} it suffices to prove that
 		\[
 		\lim_{n\to \infty} \avg{\Ga_{n} \, \alpha,\alpha} = \avg{\Ga_\infty\,\alpha,\alpha} \, ,
 		\]
 		where $\Ga_\infty$ is the covariance matrix for the Gaussian variables
 		\begin{equation}
 			\label{eq:gaussians_we_consider_GEF}
 			\widehat{g}(z_1), \ldots, \widehat{g}(z_m), \, D\widehat{g}(z_1),\ldots , D\widehat{g}(z_m), \, \widehat{g}(w_1),\ldots,\widehat{g}(w_\ell) .
 		\end{equation}
 		(That is, the limit for the random vector~\eqref{eq:complex_gaussians_we_consider} as $n\to \infty$.) The later fact follows easily from Claim~\ref{claim:local_convergence_of_polynomials_to_GEF}.
 	\end{proof}
 	For $r>0$ we denote by $r\bD = \{|z|\le r\}$ and $r\bT = \{|z|=r\}$. Suppose that the vector $(\mathbf{w},\mathbf{z})$ which participates in the definition of~\eqref{eq:def_of_L_n} has all its coordinates inside $r\bD$. Then, by Cauchy's theorem, we have that
 	\begin{equation}
 		\label{eq:L_n_as_cauchy_integral}
 		L_n(\alpha) = \frac{1}{2\pi {\rm i}} \int_{2r\bT} f_n\Big(\frac{z}{\sqrt{n}}\Big) R_n^\alpha (z) \, {\rm d} z 
 	\end{equation}
 	where  
 	\begin{multline}
 		\label{eq:def_of_R_n}
 		R_n^\alpha(z) \stackrel{{\rm def}}{=} \sum_{j=1}^{m}\Big[ \frac{\alpha_j^1}{(1+|z_j|^2/n)^{n/2}} \, \frac{1}{z-z_j} + \frac{\alpha_j^2}{(1+|z_j|^2/n)^{n/2-1}} \,  \frac{1}{(z-z_j)^2} \Big] \\ + \sum_{t=1}^{\ell} \frac{\alpha_t^\prime}{(1+|w_t|^2/n)^{n/2}} \, \frac{1}{z-w_t}\, .
 	\end{multline}
 	\begin{lemma}
 		\label{lemma:rational_function_and_L_n_are_comparable}
 		There exist a positive constant $C=C(m,\ell,r)$ such that for large enough $n$ and for all $(\mathbf{w},\mathbf{z}) \in (r\bD)^{\ell+m}$ we have
 		\[
 		\max_{|z|=2r} |R_n^\alpha(z)| \le C \, \bE|L_n(\alpha)|^2\,.
 		\]
 	\end{lemma}
 	\begin{remark*}
 		In fact, it is easy to see from~\eqref{eq:L_n_as_cauchy_integral} that the complementary inequality $\displaystyle \bE|L_n|^2\lesssim \max_{|z|=2r} |R_n^\alpha(z)|$ also holds, but we will not use this fact in what follows.
 	\end{remark*}
 	\begin{proof}[Proof of Lemma~\ref{lemma:rational_function_and_L_n_are_comparable}]
 		We start with the observation that, as $n\to \infty$, the rational function $R_n^\alpha$ converges to
 		\begin{equation}
 			\label{eq:rational_function_for_L_infty}
 			R^\alpha(z) = \sum_{j=1}^{m}e^{-|z_j|^2/2} \Big[  \frac{\alpha_j^1}{z-z_j} + \frac{\alpha_j^2}{(z-z_j)^2} \Big] + \sum_{t=1}^{\ell} e^{-|w_t|^2/2} \frac{\alpha_t^\prime}{z-w_t}\, ,
 		\end{equation}
 		where this convergence occurs uniformly inside a compact set which contains $2r\bT$. By Claim~\ref{claim:L_n_converge_to_L_infty}, the desired inequality will follow once we prove that
 		\begin{equation}
 			\label{eq:rational_function_and_L_infty_are_comparable}
 			\max_{|z|=2r} |R^\alpha(z)|^2 \le C \, \bE|L_\infty(\alpha)|^2
 		\end{equation}
 		where $L_\infty(\alpha)$ is given by~\eqref{eq:def_of_L_infty}. We also note the simple relation
 		\begin{equation}
 			\label{eq:L_infty_as_cauchy_integral}
 			L_\infty(\alpha) = \frac{1}{2\pi {\rm i}} \int_{2r\bT} g(z)R^\alpha(z) \, {\rm d} z\, .
 		\end{equation}
 		The upper bound~\eqref{eq:rational_function_and_L_infty_are_comparable} follows from~\cite[Claim~2.1]{NS-CMP}, we sketch the argument here for the sake of completeness. Let $\mathcal{R}$ denote the set of all rational functions of degree at most $(2m+\ell)$ vanishing at infinity and having all of their poles in $r\bD$, and note that the function $R^\alpha$ given by~\eqref{eq:rational_function_for_L_infty} belongs to this class. Then it is easy to check (see~\cite[Claim~2.2]{NS-CMP}) that any sequence $\{R_k\}\subset \mathcal{R}$ with $\max_{|z|=2r} |R_k(z)| \le 1$ has a subsequence that converges uniformly on $\{|z|=2r\}$ to $R\in \mathcal{R}$. 
 		
 		Assume by contradiction that~\eqref{eq:rational_function_and_L_infty_are_comparable} does not hold. Then there exists a sequence of random variables $\{L_\infty^k\}_{k\ge 1}$ all of the form~\eqref{eq:def_of_L_infty}, along with the corresponding sequence of rational functions $\{R_k\}_{k\ge 1}$ of the form~\eqref{eq:rational_function_for_L_infty} with
 		\[
 		\max_{|z|=2r} |R_k(z)| = 1 \, , \qquad \lim_{k\to\infty}\bE|L_\infty^k|^2 = 0\, .
 		\]
 		By the previous paragraph, we can choose a subsequence such that $R_k$ converge uniformly on $2r\bT$ to a non-zero rational function $R\in \mathcal{R}$. Setting
 		\[
 		L = \frac{1}{2\pi {\rm i}} \int_{2r\bT} g(z)R(z) \, {\rm d} z\, ,
 		\]
 		we see from~\eqref{eq:L_infty_as_cauchy_integral} that
 		\begin{equation*}
 			\lim_{k\to \infty} \bE\big|L - L_\infty^k\big|^2 \le (2\pi r)^2 \cdot \int_{2r\bT} \bE|g(z)|^2 \, |{\rm d}z| \cdot \lim_{k\to\infty} \big(\max_{|z|=2r} |R(z) - R_k(z)|^2\big) = 0\, .
 		\end{equation*}
 		That is, $L=0$ almost surely. In view of~\eqref{eq:coefficients_of_rational_function_in_taylor}, we get that $R$ must be the constant zero function and this is a contradiction.
 	\end{proof}
 	\subsection{Identifying the scale}
 	For the random polynomial $f_n$ as defined by~\eqref{eq:def_random_polynomial} and $\xi\in \bC$, define the random function $T_\xi f_n$ by
 	\begin{equation*}
 		T_\xi \, f_n(z) \stackrel{{\rm def}}{=} f_n\left(\frac{z+\xi}{1-z\overline \xi}\right) \frac{(1-z\overline \xi)^n}{(1+|\xi|^2)^{n/2}}\, .
 	\end{equation*} 
 	This is a special case of composing our polynomial $f_n$ with the isometries of the Riemann sphere of the form~\eqref{eq:isometries_of_sphere}, and normalizing to have the same distribution, similar as we did in~\eqref{eq:distribution_of_polynomial_after_isometry}. Geometrically, the function $T_\xi\, f_n$ can be interpreted as considering our polynomial $f_n$ in a new coordinate system, with $\xi$ being rotated to the origin. Indeed, a direct computation shows that 
 	\[
 	\bE\Big[T_\xi \, f_n(z) \overline{T_\xi \, f_n(w)}\Big] = K_n(z,w)\, , \qquad \forall z,w\in \bC
 	\] 
 	where $K_n$ is given by~\eqref{eq:covariance_kernel_for_f}. Since $T_\xi \, f_n$ is a Gaussian process on $\bC$, we conclude that it has the same distribution as $f_n$. 
 	
 	For the vector $(\mathbf{w},\mathbf{z})\in (r\bD)^{\ell+m}$, recall the random variable $L_n=L_n(\alpha)$ of the form~\eqref{eq:def_of_L_n}. For $\xi \in \bC$ we set
 	\begin{equation}
 		\label{eq:rotaion_of_L_n}
 		T_\xi (L_n) \stackrel{{\rm def}}{=} \frac{1}{2\pi{\rm i}} \int_{2r\bT} T_\xi \, f_n\Big(\frac{z}{\sqrt{n}}\Big) R_n^\alpha (z) \, {\rm d} z
 	\end{equation}
 	where $R_n^\alpha$ is the rational function given by~\eqref{eq:def_of_R_n}. Note that, in view of~\eqref{eq:L_n_as_cauchy_integral}, the random variables $L_n$ and $T_\xi (L_n)$ have the same distribution. Let $L_n^1$ and $L_n^2$ be random variables of the form~\eqref{eq:def_of_L_n}. We start by estimating $\bE\Big[T_{\xi_1} (L_n^1) \,  \overline{T_{\xi_2} (L_n^2)} \Big]$ when the spherical distance between $\xi_1, \xi_2 \in \bC$ is large.  We note that this comes down to upper bounding $\bE\Big[T_{\xi_1} f_n(\lambda_1)  \,  \overline{T_{\xi_2} f_n(\lambda_2)} \Big]$ for $\xi_1,\xi_2,\lambda_1,\lambda_2 \in \bC$; if we had $\xi_1 = 0$ for instance, then this would decay in terms of the spherical distance between $0$ and $\xi_2$ provided $\lambda_1,\lambda_2$ are in some compact ball on the scale of $1/\sqrt{n}$.  The computation in this lemma shows that this holds for arbitrary $\xi_1,\xi_2$ as well.
 	
	\begin{lemma}
		\label{lemma:exponential_decay_with_two_variables_L_n}
		Let $m,\ell \in \bZ_{\ge 0}$, $r>0$ and suppose that $d_{\bS^2} (\xi_1,\xi_2) \ge 100 r/\sqrt{n}$. Suppose that $L_n^1(\alpha)$ and $L_n^2(\beta)$ are random variables of the form~\eqref{eq:def_of_L_n}, and suppose that the corresponding rational functions $R_{n,1}^\alpha$ and $R_{n,2}^\beta$ given by~\eqref{eq:def_of_R_n} have all their poles inside $r\bD$. Then
		\begin{align}
			\label{eq:inequality_in_exponential_decay_with_two_variables_L_n}
			\Big|\bE\Big[T_{\xi_1} (L_n^1) \,  \overline{T_{\xi_2} (L_n^2)} \Big]\Big| \le C(r,m,\ell) \, \exp\Big(-\frac{n}{16}\, d_{\bS^2}(\xi_1,\xi_2)^2 \Big) \Big( \bE\big|L_n^1\big|^2 + \bE\big|L_n^2\big|^2\Big)\, ,
		\end{align}
		for all $n\ge 2$ large enough.
	\end{lemma}
	\begin{proof}
		Using the representation~\eqref{eq:rotaion_of_L_n}, we see that
		\begin{align*}
			&\Big|\bE\Big[T_{\xi_1} (L_n^1) \,  \overline{T_{\xi_2} (L_n^2)} \Big]\Big| \\  
			&= \bigg|\iint_{2r\bT\times 2r\bT} R_{n,1}^\alpha(z_1) \overline{R_{n,2}^\beta (z_2)} \, \bE\bigg[T_{\xi_1} \, f_n\left(\frac{z_1}{\sqrt{n}}\right) \overline{T_{\xi_2} \, f_n\left(\frac{z_2}{\sqrt{n}}\right)} \, \bigg] \, {\rm d}z_1 {\rm d}z_2\bigg| \\ 
			& \lesssim \max_{|z_1|=|z_2|=2r} \bigg|\bE\bigg[T_{\xi_1}\, f_n\Big(\frac{z_1}{\sqrt{n}}\Big) \,  \overline{T_{\xi_2}\, f_n\Big(\frac{z_2}{\sqrt{n}}\Big)} \bigg]\bigg|\Big(\max_{|z|=2r}|R_{n,1}^\alpha(z)|^2 + \max_{|z|=2r}|R_{n,2}^\beta (z)|^2\Big) \, ,
		\end{align*}
		and by Lemma~\ref{lemma:rational_function_and_L_n_are_comparable} we get that
		\begin{align}
			\label{eq:first_inequality_for_two_variables_L_n}
			\Big|\bE\Big[T_{\xi_1} (L_n^1) \,  \overline{T_{\xi_2} (L_n^2)} \Big]\Big|  \lesssim \max_{z_1,z_2 \in2r \mathbb{T}} \bigg|\bE\bigg[T_{\xi_1}\, f_n\Big(\frac{z_1}{\sqrt{n}}\Big) \,  \overline{T_{\xi_2}\, f_n\Big(\frac{z_2}{\sqrt{n}}\Big)} \bigg]\bigg| \,  \Big( \bE\big|L_n^1(\alpha)\big|^2 + \bE\big|L_n^2(\beta)\big|^2\Big) \, .
		\end{align}
		A computation shows that for all $\la_1,\la_2,\xi_1,\xi_2 \in \bC$ we have 
		\begin{align*}
			\bE\Big[T_{\xi_1} \, f_n(\la_1) \overline{T_{\xi_2} \, f_n(\la_2)}\Big] &= \Bigg(\frac{\big(1-\la_1 \overline{\xi}_1\big)\big(1-\overline{\la}_2 \xi_2\big)}{\sqrt{1+|\xi_1|^2}\sqrt{1+|\xi_2|^2}} \bigg(1+ \Big(\frac{\la_1+\xi_1}{1-\la_1\overline{\xi}_1}\Big) \Big(\frac{\overline{\la}_2+\overline{\xi}_2}{1-\overline{\la}_2 \xi_2}\Big)\bigg)\Bigg)^n \\ &= \Bigg(\frac{\big(1-\la_1 \overline{\xi}_1\big)\big(1-\overline{\la}_2 \xi_2\big) + \big(\la_1+\xi_1\big)\big(\overline{\la}_2+\overline{\xi}_2\big)}{\sqrt{1+|\xi_1|^2}\sqrt{1+|\xi_2|^2}} \Bigg)^n \\ &= \Bigg(\frac{1+\xi_1\overline\xi_2 + \la_1\overline\la_2\big(1+\overline\xi_1\xi_2\big) - \la_1\big(\overline{\xi}_1-\overline{\xi}_2\big) + \overline \la_2\big(\xi_1-\xi_2\big)}{\sqrt{1+|\xi_1|^2}\sqrt{1+|\xi_2|^2}} \Bigg)^n\, . 
		\end{align*}
		We also note that
		\begin{equation*}
			\frac{|1+\xi_1\overline{\xi}_2|}{\sqrt{1+|\xi_1|^2}\sqrt{1+|\xi_2|^2}} = \bigg(1-\frac{d_{\bS^2}(\xi_1,\xi_2)^2}{4}\bigg)^{1/2}\, .
		\end{equation*}
		Therefore, by the triangle inequality, we obtain that
		\begin{align*}
			\max_{|z_1|=|z_2|=2r} \left|\bE\left[T_{\xi_1}\, f_n\Big(\frac{z_1}{\sqrt{n}}\Big) \,  \overline{T_{\xi_2}\, f_n\Big(\frac{z_2}{\sqrt{n}}\Big)} \right]\right|  &\le \bigg(\bigg(1-\frac{d_{\bS^2}(\xi_1,\xi_2)^2}{4}\bigg)^{1/2}\Big(1+\frac{4r^2}{n}\Big) + \frac{2r}{\sqrt{n}} d_{\bS^2}(\xi_1,\xi_2)\bigg)^n \\ & \leq \exp\Big( - \frac{n}{16}\, d_{\bS^2}(\xi_1,\xi_2)^2 \Big)
		\end{align*}
		provided that $d_{\bS^2}(\xi_1,\xi_2) \ge 100 r/\sqrt{n}$. Plugging into~\eqref{eq:first_inequality_for_two_variables_L_n}, we get the lemma.
	\end{proof}
	We will also use the following simple claim,  the argument being borrowed from \cite[Claim~3.2]{NS-CMP} and adapted to the spherical metric. For $y\in \bC$ and $r>0$, we denote by
	\[
	D(y,r) \stackrel{{\rm def}}{=} \{ x\in \widehat\bC \mid d_{\bS^2}(y,x) \le r \}
	\]
	the (closed) spherical cap of radius $r$ centered at $y$.  
	\begin{claim}
		\label{claim:identifying_the_scale_geometrically}
		Suppose we are given $x_1,\ldots,x_{\ell+m} \in \bC$ and $\psi:(0,\infty) \to (0,\infty)$ some increasing function. Then for all large enough $n$, there exist $r>0$ and $y_1,\ldots,y_N \in \bC$ (with $N\le \ell+m$) such that the collection of spherical caps $\{D(y_i,\frac{r}{\sqrt{n}})\}_{1\le i\le N}$ cover $x_1,\ldots,x_{\ell+m}$ and
		\[
		d_{\bS^2}(y_{i},y_{i^\prime}) \ge \frac{\psi(r)}{\sqrt{n}} \quad \text{for all } \ i\not=i^\prime.
		\]
		Furthermore, $r$ is bounded from above by a constant depending only on $m,\ell$ and $\psi$  (and \underline{not} on $n$). 
	\end{claim}
	\begin{proof}
		We start the procedure by taking $r_1 = 1$ and $y_i = x_i$ for all $i=1,\ldots,\ell+m$. If all pairwise distances $d_{\bS^2}(y_i,y_{i^\prime}) \ge \psi(r_1)/\sqrt{n}$ then we are done. If not, choose a pair of centers $y_i$ and $y_{i^\prime}$ for which $d_{\bS^2}(y_i,y_{i^\prime}) < \psi(r_1)/\sqrt{n}$ and replace them by one center $\widetilde{y} \in \bC$, which is the unique point that satisfies
		\[
		d_{\bS^2}(\widetilde{y},y_i)^2 = d_{\bS^2}(\widetilde{y},y_{i^\prime})^2 = 2-\sqrt{4-d_{\bS^2}(y_i,y_{i^\prime})^2}\, .
		\]
		(The point $\widetilde{y}$ is the midpoint of the geodesic\footnote{On the sphere, geodesics between points are unique unless the points are antipodal; in the case of antipodal points, choose a $\widetilde{y}$ to be an arbitrary midpoint.} which connects $y_i$ and $y_{i^\prime}$.) By increasing the radii of the spherical caps to $$r_2 = r_1 + \sqrt{2-\sqrt{4-\psi(r_1)^2/n}}$$ we get a new configuration which still covers $x_1,\ldots,x_{\ell+m}$. If in the resulting configuration all distances between centers of spherical caps is at least $\psi(r_2)/\sqrt{n}$ then we are done. Otherwise, we continue with our procedure and replace two close centers by a single one, while increasing the radii to $$r_3 = r_2 + \sqrt{2-\sqrt{4-\psi(r_2)^2/n}} \, .$$ We repeat this process until we obtain the desired configuration, and the final radius is bounded from above by the $(\ell+m)$-th term in the recursive sequence $$r_j = r_{j-1} + \sqrt{2-\sqrt{4-\psi(r_{j-1})^2/n}}$$ which is clearly bounded from above by a constant which does not depend on $n$ (in fact, it is monotonically decreasing in $n$).	
	\end{proof}
	Fix for a moment $(\mathbf{w},\mathbf{z})\in \bC^{\ell+m}$ and a random variable $L_n = L_n(\alpha)$ of the form~\eqref{eq:def_of_L_n}. Given a non-trivial subset of indices $I\subset[m]$ and $I^\prime\subset[\ell]$, we set
	\begin{equation}
		\label{eq:def_of_L_n_partial_sum}
		L_n^I \stackrel{{\rm def}}{=} \sum_{j\in I}\left[\alpha_j^1 \, \widehat{f}_n\left(\frac{z_j}{\sqrt{n}}\right)+\alpha_j^2 \, D\widehat{f}_n\left(\frac{z_j}{\sqrt{n}}\right) \right] + \sum_{t\in I^\prime} \alpha_t^\prime\, \widehat{f}_n\left(\frac{w_t}{\sqrt{n}}\right)\, .
	\end{equation}
	Clearly, if $[m] = I\sqcup J$ and $[\ell] = I^\prime\sqcup J^\prime$ is some partition of the indices, we have
	\[
	L_n = L_n^I + L_n^J 
	\]
	where $L_n^I$ and $L_n^J$ are given by~\eqref{eq:def_of_L_n_partial_sum}. We conclude this section with a quantitative estimate on how $L_n^I$ and $L_n^J$ decorrelate, once the spherical distance between $(\mathbf{w}_{I^\prime},\mathbf{z}_I)$ and $(\mathbf{w}_{J^\prime},\mathbf{z}_J)$ becomes large.
	\begin{lemma}
		\label{lemma:exponential_decay_of_correlation_for_seperated_L_n}
		For all $\ell,m\in \bZ_{\ge 0}$ there exist constants $C,D>0$ with the following property. Let $(\mathbf{w},\mathbf{z}) \in \bC^{\ell+m}$, and suppose we can partition $[m]$ and $[\ell]$ into non-empty subsets $I\sqcup J = [m]$ and $I^\prime\sqcup J^\prime= [\ell]$ such that
		\[
		d = \text{\normalfont \sf dist} \Big\{  \Big(\frac{\mathbf{w}_{I^\prime}}{\sqrt n},\frac{\mathbf{z}_I}{\sqrt n}\Big) , \Big(\frac{\mathbf{w}_{J^\prime}}{\sqrt n},\frac{\mathbf{z}_J}{\sqrt n}\Big) \Big\} \ge \frac{D}{\sqrt{n}}\, ,
		\]
		where {\normalfont \sf dist} is defined via~\eqref{eq:def_spherical_distance_between_sets}.
		Then
		\begin{equation*}
			\Big|\bE\big[L_n^I \, \overline{L_n^J} \, \big]\Big| \le C e^{- n d^2/{16} } \Big( \bE \big|L_n^I\big|^2 + \bE \big|L_n^J\big|^2 \Big) \, , 
		\end{equation*}
		where $L_n^I$ and $L_n^J$ are given by~\eqref{eq:def_of_L_n_partial_sum}. 
	\end{lemma}
	\begin{proof}
		We apply the construction of Claim~\ref{claim:identifying_the_scale_geometrically} to the point $(x_1,\ldots,x_{\ell+m}) = \frac{1}{\sqrt{n}}(\mathbf{w},\mathbf{z}) \in \bC^{\ell+m}$ with
		\begin{equation*}
			\psi(r) = 100 \, r + \sqrt{16 \log\big(10\, C(r) \, (m+\ell)\big)}
		\end{equation*}
		where $C(r) = C(r, m,\ell)$ is the constant appearing in~\eqref{eq:inequality_in_exponential_decay_with_two_variables_L_n}. We obtain $r\ge 1$ and $y_1,\ldots,y_N\in \bC$ such that $\{D(y_s,\frac{r}{\sqrt{n}})\}_{1\le s\le N}$ cover the entries of the vector $n^{-1/2}(\mathbf{w},\mathbf{z})$. Furthermore, $r\le r_0(\ell,m,\psi)$ with $r_0$ independent of the points $x_1,\ldots,x_{\ell+m}$ and the degree $n$. We can now take $D = 102\,  r_0(\ell,m,\psi)$ and check that the inequality stated in the lemma holds.
		
		First, we note that the partitions $[m] = I\sqcup  J$ and $[\ell] = I^\prime \sqcup J^\prime$ induce a partition $[\ell+m] = \widetilde{I}\sqcup \widetilde{J}$, which in turn splits the complex numbers $x_1,\ldots,x_{\ell+m}$ into two disjoint sets
		\[
		\mathbf{x}_{\widetilde I} \stackrel{{\rm def}}{=} \{ x_i : \, i\in \widetilde I\} \quad \text{and} \quad \mathbf{x}_{\widetilde J} \stackrel{{\rm def}}{=} \{ x_j : \, j\in \widetilde J\} \, ,
		\]
		with $\text{\sf dist}\{\mathbf{x}_{\widetilde I},\mathbf{x}_{\widetilde J} \} \ge D/\sqrt{n}$. By our choice of $D$, we deduce that no spherical cap $D(y_s,\frac{r}{\sqrt{n}})$ contains two points $x_i$ and $x_j$ with $i\in \widetilde I$ and $j\in \widetilde J$. As we can always assume that each spherical cap contains at least one point from $x_1,\ldots,x_{\ell+m}$, we conclude that the partition $[\ell+m] = \widetilde{I}\sqcup \widetilde{J}$ induces a partition $[N] = \mathrm{I} \sqcup \mathrm{J}$ of the spherical caps. For $1\le s\le N$, we set
		\begin{equation*}
			L_{n,s} \stackrel{{\rm def}}{=} \sum_{j:\, z_j \in D(y_s,\frac{r}{\sqrt{n}})}\Big[\alpha_j^1 \, \widehat{f}_n\Big(\frac{z_j}{\sqrt{n}}\Big)+\alpha_j^2 \, D\widehat{f}_n\Big(\frac{z_j}{\sqrt{n}}\Big) \Big] + \sum_{t:\, w_t \in D(y_s,\frac{r}{\sqrt{n}})} \alpha_t^\prime\, \widehat{f}_n\Big(\frac{w_t}{\sqrt{n}}\Big)\, ,
		\end{equation*}
		and note that, in view of~\eqref{eq:def_of_L_n_partial_sum} we have
		\[
		L_n^I = \sum_{s\in \mathrm{I}} L_{n,s} \, .
		\]
		Expanding the second moment, we see that
		\begin{equation*}
			\bE\big|L_n^I\big|^2 = \sum_{s\in \mathrm{I}} \bE\big|L_{n,s}\big|^2 + \sum_{\substack{s\not = s^\prime \\ s,s^\prime \in \mathrm{I}}} \bE\big[L_{n,s} \, \overline{L_{n,s^\prime}}\big] \, . 
		\end{equation*}
		Since $\psi(r) \ge 100 r$, we can apply Lemma~\ref{lemma:exponential_decay_with_two_variables_L_n} and obtain that for $s\not=s^\prime$
		\begin{equation*}
			\big|\bE\big[L_{n,s} \, \overline{L_{n,s^\prime}}\big]\big| \le C(r) e^{-\frac{\psi(r)^2}{16}} \Big(\bE|L_{n,s}|^2 + \bE|L_{n,s^\prime}|^2 \Big) \le \frac{1}{10(m+\ell)}  \Big(\bE|L_{n,s}|^2 + \bE|L_{n,s^\prime}|^2 \Big)\, .
		\end{equation*}
		We obtain that
		\begin{equation*}
			\bE\big|L_n^I\big|^2 \ge \frac{1}{2} \sum_{s\in \mathrm{I}} \bE\big|L_{n,s}\big|^2\, ,
		\end{equation*}
		and similarly
		\begin{equation*}
			\bE\big|L_n^J\big|^2 \ge \frac{1}{2} \sum_{s\in \mathrm{J}} \bE\big|L_{n,s}\big|^2\, .
		\end{equation*}
		On the other hand, for $s\in \mathrm{I}$ and $s^\prime\in \mathrm{J}$, we have
		\[
		\sqrt{n} \, d_{\bS^2} (y_s,y_{s^\prime}) \ge D - 2r \ge 100 r 
		\]
		and we can apply Lemma~\ref{lemma:exponential_decay_with_two_variables_L_n} once more to get
		\begin{align*}
			\big|\bE\big[L_{n}^I \, \overline{L_{n}^J}\big]\big| & \le \sum_{\substack{s\in \mathrm{I} \\ s^\prime \in \mathrm{J}}} \big|\bE\big[L_{n,s} \, \overline{L_{n,s^\prime}}\big]\big| \\ & \lesssim \sum_{\substack{s\in \mathrm{I} \\ s^\prime \in \mathrm{J}}} \exp\Big(-\frac{n}{16} d_{\bS^2}(y_s,y_{s^\prime})^2 \Big) \, \Big(\bE|L_{n,s}|^2 + \bE|L_{n,s^\prime}|^2 \Big)  \\ &\lesssim \exp\Big(-
			 \frac{n}{16} \, d^2 \Big) \Big( \bE \big|L_n^I\big|^2 + \bE \big|L_n^J\big|^2 \Big)
		\end{align*}
		as desired.
	\end{proof}
	\subsection{Proof of Theorem~\ref{thm:clustering_of_density}} \label{ss:proof-clustering}
	We are ready to prove the main result of Section~\ref{sec:clustering_for_random_measure}. To simplify our notation, we denote by $$\varepsilon = \varepsilon(m,\ell,d,n) \stackrel{{\rm def}}{=} C \exp\left(- n d^2 / {16}\right)$$
	where $C=C(m,\ell)$ is the positive constant from Lemma~\ref{lemma:exponential_decay_of_correlation_for_seperated_L_n}. Recall that $\Ga = \Ga_n(\mathbf{w},\mathbf{z})$ is the covariance matrix for the Gaussian random variables given by~\eqref{eq:complex_gaussians_we_consider}. For a partition on the vector $(\mathbf{w},\mathbf{z})$ as given in the statement of Theorem~\ref{thm:clustering_of_density}, we denote by $\Ga^{I,J}$ the covariance matrix of the Gaussian random variables~\eqref{eq:complex_gaussians_we_consider}, when we set the correlations between the points $(\mathbf{w}_{I^\prime},\mathbf{z}_I)$ and $(\mathbf{w}_{J^\prime},\mathbf{z}_J)$ to be zero (this corresponds to sampling two independent copies of the random polynomial $f_n$ for the different tuples). With this notation in mind, Claim~\ref{claim:L_n_represents_Ga_n_in_L_2} together with Lemma~\ref{lemma:exponential_decay_of_correlation_for_seperated_L_n} implies that
	\begin{equation}
		\label{eq:inequality_in_operator_sense_for_correlations}
		(1-\eps) \, \Ga^{I,J} \le \Ga \le (1+\eps) \, \Ga^{I,J}
	\end{equation} 
	where the inequality in understood in the operator sense.	
	
	Recall the formula~\eqref{eq:Kac_Rice_densities_non_conditional} for the density $\rho_{\ell,m,p}$, which reads
	\begin{equation*}
		\rho_{\ell,m,p} (\mathbf{w},\mathbf{z})= n^{\ell+m} \int_{\bC^{m+\ell}} \varphi_n(\mathbf{0},\eta,\eta^\prime)  \Big(\prod_{t=1}^{\ell} \log|\eta_t^\prime| \Big) \Big(\prod_{j=1}^{m} |\eta_j|^2 \, \log^{p_j}|\eta_j|\Big)  \, {\rm d}m(\eta,\eta^\prime) \, ,
	\end{equation*}
	where $\varphi_n$ is the density of a complex Gaussian vector with covariance matrix $\Ga$. Similarly, we denote by $\varphi_n^{I,J}$ to be the density of a complex Gaussian vector with covariance matrix $\Ga^{I,J}$, that is
	\begin{equation*}
		\varphi_n^{I,J}(\eta_1,\eta_2,\eta^\prime) = \frac{1}{\pi^{2m+\ell} \det\big( \Ga^{I,J}\big)} \exp\Big( - \avg{(\Ga^{I,J})^{-1} \,\widetilde \eta,\widetilde \eta } \Big)\, ,
	\end{equation*}
	where $\widetilde\eta=(\eta_1,\eta_2,\eta^\prime)^{{\sf T}}\in \bC^{2m+\ell}$. In view of~\eqref{eq:Kac_Rice_densities_non_conditional}, we have the simple relation
	\begin{align}
		\label{eq:Kac_rice_densities_two_independent_copies}
	\rho_{|I^\prime|,|I|,p_I}&\Big(\frac{\mathbf{w}_{I^\prime}}{\sqrt{n}},\frac{\mathbf{z}_I}{\sqrt{n}}\Big)\rho_{|J^\prime|,|J|,p_J}\Big(\frac{\mathbf{w}_{J^\prime}}{\sqrt{n}},\frac{\mathbf{z}_J}{\sqrt{n}}\Big) \\ \nonumber &= n^{\ell+m} \int_{\bC^{m+\ell}} \varphi_n^{I,J}(\mathbf{0},\eta,\eta^\prime)  \Big(\prod_{t=1}^{\ell} \log|\eta_t^\prime| \Big) \Big(\prod_{j=1}^{m} |\eta_j|^2 \, \log^{p_j}|\eta_j|\Big)  \, {\rm d}m(\eta,\eta^\prime) \, .
	\end{align}
	In view of~\eqref{eq:Kac_Rice_densities_non_conditional} and~\eqref{eq:Kac_rice_densities_two_independent_copies}, the proof of Theorem~\ref{thm:clustering_of_density} will follow from quantitative estimates of $(\varphi_n - \varphi_n^{I,J})$, which is based on~\eqref{eq:inequality_in_operator_sense_for_correlations}.
	\begin{proof}[Proof of Theorem~\ref{thm:clustering_of_density}]
		We first observe that, in view of~\eqref{eq:inequality_in_operator_sense_for_correlations}, we have 
		\begin{equation}
			\label{eq:inequality_for_gaussian_determinants}
			\left(1-	\eps\right)^{2m+\ell} \det\left(\Ga^{I,J}\right) \le \det \left(\Ga\right) \le \left(1+\eps\right)^{2m+\ell} \det\left(\Ga^{I,J}\right)\, .
		\end{equation}
		We also need to compare the integrals against the Gaussian function, and we start with the upper bound. Let
		\begin{equation*}
			\Upsilon_{m+\ell} \stackrel{{\rm def}}{=} \big\{ \sigma \mid \sigma :[m+\ell] \to \{\pm1\}\big\}
		\end{equation*}
		We denote by $\bD^{+1} \stackrel{{\rm def}}{=} \bD$ and $\bD^{-1} \stackrel{{\rm def}}{=} \bC\setminus \bD$, which make the following identity obvious
		\[
		\bC^{\ell+m} = \bigcup_{\sigma\in \Upsilon_{m+\ell}} \bD^{\sigma(1)} \times \ldots \times \bD^{\sigma(m+\ell)} \, .
		\] 
		Fix for a moment some $\sigma\in \Upsilon_{m+\ell}$ and let $\mathcal{A}_\sigma =\bD^{\sigma(1)} \times \ldots \times \bD^{\sigma(m+\ell)}$. On each domain $\mathcal{A}_\sigma$ the function
		\[
		(\eta^\prime,\eta) \mapsto \Big(\prod_{t=1}^{\ell} \log|\eta_t^\prime| \Big) \, \Big(\prod_{j=1}^{m} |\eta_j|^2 \, \log^{p_j}|\eta_j|\Big)
		\]
		does not change sign. In view of~\eqref{eq:inequality_in_operator_sense_for_correlations}, there exist a sign $\tau=\tau(\sigma) \in \{-1,+1\}$ such that, if we let $\widetilde{\eta} = (\mathbf{0},\eta,\eta^\prime)$,
		\begin{multline}
			\label{eq:clustering_upper_bound_over_single_sigma_domain}
			\int_{\mathcal{A}_\sigma} \Big(\prod_{t=1}^{\ell} \log|\eta_t^\prime| \Big) \Big(\prod_{j=1}^{m} |\eta_j|^2 \, \log^{p_j}|\eta_j|\Big) e^{- \avg{\Ga^{-1} \, \widetilde\eta, \widetilde\eta }} \, {\rm d}m(\eta,\eta^\prime) \\ \le \int_{\mathcal{A}_\sigma} \Big(\prod_{t=1}^{\ell} \log|\eta_t^\prime| \Big) \Big(\prod_{j=1}^{m} |\eta_j|^2 \, \log^{p_j}|\eta_j|\Big) e^{- (1+\tau\eps)^{-1}\avg{(\Ga^{I,J})^{-1} \, \widetilde\eta, \widetilde\eta }} \, {\rm d}m(\eta,\eta^\prime) \\ = (1+\tau\eps)^{2m+\ell} \int_{(1+\tau \eps)\mathcal{A}_\sigma} \Big(\prod_{t=1}^{\ell} \log|(1+\tau\eps)\eta_t^\prime| \Big) \Big(\prod_{j=1}^{m} \log^{p_j}|(1+\tau\eps)\eta_j|\Big) \\ \times \Big(\prod_{j=1}^{m} |\eta_j|^2\Big) \, e^{- \avg{(\Ga^{I,J})^{-1} \, \widetilde\eta, \widetilde\eta }} \, {\rm d}m(\eta,\eta^\prime) \, .
		\end{multline}
		We make a few simple observations. For any $\tau \in \{-1,1\}$ we have
		\[
		\big|\log\big((1+\tau\eps) x\big) - \log x \big| \lesssim \eps
		\]
		uniformly in $x>0$. We also have the relation
		\[
		\Big|1- \big(1+\eps\big)^{2m+\ell}\Big| \lesssim \eps\, .
		\]
		Finally, as $|\log(1+x)| \lesssim \eps$ uniformly in $x\in [1-\eps,1+\eps]$ we can also deduce that
		\begin{align*}
			\int_{(1+\tau \eps)\mathcal{A}_\sigma \triangle \mathcal{A}_\sigma} & \Big(\prod_{t=1}^{\ell} |\log|\eta_t^\prime|| \Big) \Big(\prod_{j=1}^{m} |\log^{p_j}|\eta_j||\Big)  \Big(\prod_{j=1}^{m} |\eta_j|^2\Big) \, e^{- \avg{(\Ga^{I,J})^{-1} \, \widetilde\eta, \widetilde\eta }} \, {\rm d}m(\eta,\eta^\prime) \\ &\lesssim \eps^{m+\ell}\int_{(1+\tau \eps)\mathcal{A}_\sigma \triangle \mathcal{A}_\sigma}  \Big(\prod_{j=1}^{m} |\eta_j|^2\Big) \, e^{- \avg{(\Ga^{I,J})^{-1} \, \widetilde\eta, \widetilde\eta }} \, {\rm d}m(\eta,\eta^\prime)\, .
		\end{align*}
		Combining this observations with~\eqref{eq:inequality_for_gaussian_determinants}, we can sum the inequality~\eqref{eq:clustering_upper_bound_over_single_sigma_domain} over $\sigma\in \Upsilon_{m+\ell}$ and obtain that
		\begin{multline*}
			\rho_{\ell,m,p}\Big(\frac{\mathbf{w}}{\sqrt{n}},\frac{\mathbf{z}}{\sqrt{n}}\Big) - \rho_{|I^\prime|,|I|,p_I}\Big(\frac{\mathbf{w}_{I^\prime}}{\sqrt{n}},\frac{\mathbf{z}_I}{\sqrt{n}}\Big)\rho_{|J^\prime|,|J|,p_J}\Big(\frac{\mathbf{w}_{J^\prime}}{\sqrt{n}},\frac{\mathbf{z}_J}{\sqrt{n}}\Big) \\ =n^{\ell+m} \sum_{\sigma\in \Upsilon_{m+\ell}} \int_{\mathcal{A}_\sigma} \Big(\varphi_n(\widetilde{\eta}) - \varphi_n^{I,J}(\widetilde \eta)\Big)   \Big(\prod_{t=1}^{\ell} \log|\eta_t^\prime| \Big) \Big(\prod_{j=1}^{m} |\eta_j|^2 \, \log^{p_j}|\eta_j|\Big)  \, {\rm d}m(\eta,\eta^\prime) \\ \lesssim \eps n^{\ell+m}\max_{\vec{\mathrm{p}}}\int_{\bC^{m+\ell}} \varphi_n^{I,J}(\widetilde \eta)   \Big(\prod_{t=1}^{\ell} \big|\log|\eta_t^\prime|\big| \Big) \Big(\prod_{j=1}^{m} |\eta_j|^2 \, \big|\log^{\mathrm{p}_j}|\eta_j|\big|\Big)  \, {\rm d}m(\eta,\eta^\prime) \, ,
		\end{multline*}
		where the maximum runs over all integer vectors $\vec{\mathrm{p}} = (\mathrm{p}_1,\ldots,\mathrm{p}_m)$ with $\mathrm{p}_j\le p_j$ for all $j=1,\ldots,m$. In view of Theorem~\ref{thm:local_bounds_on_dentity} (and Remark \ref{remark:log-abs}), the above maximum is bounded by
		\[
		\bigg(\prod_{1\le t\le \ell} |\log(\Delta^\prime_t)|\bigg)
		\]
		and this gives the desired upper bound. The lower bound follows in a symmetric way by repeating the same argument with $\tau(\sigma)$ replaced by $-\tau(\sigma)$ for all $\sigma\in \Upsilon_{m+\ell}$. Combining both lower and upper bounds proves Theorem~\ref{thm:clustering_of_density}. 
	\end{proof}

	\section{Cumulants densities and bounds. Proof of Theorem~\ref{thm:limit_joint_cumulant}}
	\label{sec:cumulant_densities}
	In this section we will use the clustering property for the Kac-Rice densities $\rho_{\ell,m,p}$ to establish the existence of limits for the cumulants of the logarithmic energy $\mathcal{E}_n$. We refer the reader to~\cite[Chapter~II, \S12]{Shiryaev} for basic properties of the cumulants of random variables and their relation to the moments. For a finite set $A$, let $\mathcal{P}(A)$ denote all partitions of $A$ into non-trivial blocks, and recall that $\mathcal{P}(k) = \mathcal{P}([k])$. We denote by $|\pi|$ the number of blocks in the partition $\pi$. We start this section by recording two simple claims on the relation between the moments and the cumulants that follow from applying the M\"obius inversion formula on the lattice of set partitions (see \cite[\S 3]{Speed}).
	\begin{claim}[\cite{Speed}]
		\label{claim:cumulants_and_moments_relation}
		For $k\ge 1$, let $(\mathfrak{m}_{J})_{J\subset [k]}$ and $(\mathfrak{s}_{J})_{J\subset[k]}$ be two collections of real numbers indexed by subsets of $[k] = \{1,\ldots,k\}$. Then the following two identities are equivalent:
		\begin{align}
			\label{eq:cumulants_and_moments_general_form}
			\nonumber
			\forall J\subset [k], \quad \mathfrak{m}_J &= \sum_{\pi\in \mathcal{P}(J)} \prod_{I\in \pi} \mathfrak{s}_{I} \\ \forall J\subset [k], \quad \mathfrak{s}_J &= \sum_{\pi\in \mathcal{P}(J)} (-1)^{|\pi|-1} \big(|\pi|-1\big)! \prod_{I\in \pi} \mathfrak{m}_I\, .
		\end{align}
		Furthermore, if there exists a partition $\pi\in \mathcal{P}(k)$ with $|\pi|\ge 2$ so that
		\[
		\forall J\subset [k]\, , \quad \mathfrak{m}_J = \prod_{I\in \pi} \mathfrak{m}_{I\cap J}
		\]
		then~\eqref{eq:cumulants_and_moments_general_form} implies that $\mathfrak{s}_{[k]} = 0$.
	\end{claim}

	If we assume that the ``moments" $(\mathfrak{m}_J)_{J\subset [k]}$ approximately factor, then we will get a quantitative bound for the ``cumulant" $\mathfrak{s}_{[k]}$.
	\begin{claim}
		\label{claim:approximate_factor_for_moments_implies_small_cumulant}
		For each $k\ge 1$ there exist a constant $C=C(k)>0$ so that the following holds. For $0<\eps\le1\le M$, let $(\mathfrak{m}_J)_{J\subset [k]}$ be a collection of real numbers, all in the interval $[-M,M]$. Suppose there exists a partition $\pi \in \mathcal{P}(k)$ with $|\pi|\ge 2$ so that for all $J\subset [k]$ we have
		\[
		\Big|\mathfrak{m}_J - \prod_{I\in \pi} \mathfrak{m}_{I\cap J}\Big|\le \eps \, ,
		\]
		then
		\[
		\bigg|\sum_{\la\in \mathcal{P}(k)} (-1)^{|\la|-1} \big(|\la|-1\big)! \prod_{J\in \la} \mathfrak{m}_J\bigg| \le C M^{k-1} \eps\, .
		\] 
	\end{claim}
	\begin{proof}
		Put $\widetilde{\mathfrak{m}}_{J} \stackrel{{\rm def}}{=} \prod_{I\in \pi} \mathfrak{m}_{J\cap I}$ for $J\subset [k]$. By Claim~\ref{claim:cumulants_and_moments_relation} we have that
		\[
		\sum_{\la \in \mathcal{P}(k)} (-1)^{|\la|-1} (|\la|-1)! \prod_{J\in \la}\widetilde{\mathfrak{m}}_{J} = 0\, .
		\]
		We thus obtain
		\begin{align*}
			\bigg|\sum_{\la\in \mathcal{P}(k)} (-1)^{|\la|-1} \big(|\la|-1\big)! \prod_{J\in \la} \mathfrak{m}_J\bigg| &\le \sum_{\la\in \mathcal{P}(k)} \big(|\la|-1\big)! \Big|\prod_{J\in \la} \mathfrak{m}_J - \prod_{J\in \la} \widetilde{\mathfrak{m}}_J \Big| \\ & \le \sum_{\la\in \mathcal{P}(k)} \big(|\la|-1\big)! (2M)^{|\la|-1} \eps \\ &\lesssim M^{k-1} \eps\, .
		\end{align*}
	\end{proof}
	\subsection{Cumulant densities}
	Let $a,b\in \mathbb{Z}_{\ge 0}$ be fixed throughout. Recall that $\gamma_{n}(a,b)$ given by~\eqref{eq:def_joint_cummulant_gamma_a_b}, is the joint cumulant for $\mathcal{I}_n$ and $\mathcal{S}_n$ of order $(a,b)$. The goal of this section is to express this joint cumulant in term of the moments given by Proposition~\ref{prop:formula_for_joint_moments}. To do so, we will need to introduce some more combinatorial notation. Recall that $\mathcal{P}(a,m)$ is the set of all partitions $\pi\in \mathcal{P}(a)$ with $|\pi|=m$. We always think of $\mathcal{P}(0)$ as the empty set. We shall construct partitions of $\mathcal{P}(a+b)$ by first partitioning $[a]$, and then completing the partition adding elements of $[b]$, either to existing blocks or by constructing new blocks. The following definition becomes relevant.
	
	\begin{definition}
		For integers $m,b\in \mathbb{Z}_{\ge0}$, we denote by $\mathcal{P}_b(m)$ as the set of all partitions of $\{x_1,\ldots,x_m,y_1,\ldots,y_b\}$ such that the $x_j$'s are in distinct blocks. For $\pi\in \mathcal{P}_b(m)$, let $|\pi|$ be the total number of blocks, and note that $|\pi|\ge m$. We further denote by $\mathcal{P}_b(m,t)$ to be the set of all partitions in $\mathcal{P}_b(m)$ with $|\pi|=t$. Finally, we denote by $\ell_j(\pi)$, $1\le j\le t$, as the number of $y$-terms in the $j$'th block for $\pi \in \mathcal{P}_b(m,t)$. We will always enumerate the blocks of $\pi\in \mathcal{P}_b(m,t)$ so that the first block contains $x_1$, the second block contains $x_2$ and so forth, while the rest of the $t-m$ blocks are enumerated arbitrarily.
 	\end{definition}
 	With this definition in mind, we can write the following formula for bivariate cumulants.
 	\begin{claim}
 		\label{claim:bivariate_cumulants_general}
 		Let $X,Y$ be random variables with all moments being finite. Then
 		\begin{align*}
 			s\big(\underbrace{X,\ldots,X}_{a \, \text{\normalfont times}} , \underbrace{Y,\ldots,Y}_{b \, \text{\normalfont times}}\big)  &= \sum_{m=1}^{a} \sum_{\la\in \mathcal{P}(a,m)} \sum_{t=m}^{m+b} \sum_{\pi\in \mathcal{P}_b(m,t)} (-1)^{t-1} (t-1)! \prod_{j=1}^{t} \bE\big[X^{|\la_j|} Y^{\ell_j(\pi)}\big]
 		\end{align*}
 		where we interpret $\lambda_j = \emptyset$ for $j\ge |\la|$.
 	\end{claim}
 	\begin{proof}
 		By~\eqref{eq:def_joint_cumulant}, we have
 		\begin{equation}
 			\label{eq:joint_cumulants_for_X_Y_general_first_equality}
 			s(\underbrace{X,\ldots,X}_{a \, \text{times}} \, , \, \underbrace{Y,\ldots,Y}_{b \, \text{times}}) = \sum_{\sigma\in \mathcal{P}(a+b)} (|\sigma|-1)! \, (-1)^{|\sigma|-1} \left(\prod_{I\in \sigma} \bE\Big[X^{x_I} Y^{|I|-x_I}\Big]\right)
 		\end{equation}
 		where $x_I$ is the number of indices from $[a]$ which appear in the block $I\in \sigma$. To construct a partition $\sigma\in \mathcal{P}(a+b)$, we start with a partition $\la\in \mathcal{P}(a,m)$ and label its blocks as $\{\la_1,\ldots,\la_m\}=\{x_1,\ldots,x_m\}$. We now need to add this partition the elements $\{a+1,\ldots,a+b\} = \{y_1,\ldots,y_b\}$, which amounts to considering a partition $\pi \in \mathcal{P}_b(m)$. Every partition $\sigma\in \mathcal{P}(a+b)$ can be uniquely decomposed in this way, with $|\sigma| = |\pi|$ and with $|\la_j|$ being the number of indices from $[a]$ appearing in the $j$'th block. Partitioning the sum in~\eqref{eq:joint_cumulants_for_X_Y_general_first_equality} and summing over all possible lengths of $\la$ and $\pi$ yields the desired claim.
 	\end{proof}
 	 For partitions $\sigma,\la\in \mathcal{P}(a)$, we say that $\sigma$ is a \emph{refinement} of $\la$ and denote by $\sigma\preceq \la$, if all blocks of $\sigma$ are contained in blocks of $\la$. Whenever $\sigma\preceq \la = \{\la_1,\ldots,\la_m\}$, we will denote by $\la_j(\sigma) \in \mathcal{P}(|\la_j|)$ the unique partition on $[|\la_j|]$ induced by $\sigma$, simply by collecting the corresponding blocks of $\sigma$ inside the block $\la_j$.
 	
 	Recall that for all $k\ge 0$
 	\begin{equation}
 	\label{eq:open_brackets_of_S_k_as_sum_over_partitions}
 	(\mathcal{S}_n)^k = \sum_{\pi \in \mathcal{P}(k)}\mathcal{S}_n^{(\pi)} \, .
 	\end{equation} 
 	where $\mathcal{S}_n^{(\pi)}$ is given by~\eqref{eq:def_S_n_of_partition}.
 	\begin{lemma}
 		\label{lemma:formula_for_gamma_n_a_b_in_terms_of_partitions}
 		For $\gamma_{n}(a,b)$ given by~\eqref{eq:def_joint_cummulant_gamma_a_b} we have
 		\begin{equation*}
 			\gamma_{n}(a,b) = \sum_{\sigma\in \mathcal{P}(a)} \sum_{\substack{\la \in \mathcal{P}(a) \\ \sigma \preceq \la}} \sum_{\pi \in \mathcal{P}_b(|\la|)} (|\pi|-1)! (-1)^{|\pi|-1} \prod_{j=1}^{|\pi|} \bE\Big[ \mathcal{S}_n^{(\la_j(\sigma))} \, \mathcal{I}_n^{\ell_j(\pi)}\Big]
 		\end{equation*}
 		where, as before, we interpret $\lambda_j = \emptyset$ for $j\ge |\la|$ and $\mathcal{S}_n^{(\emptyset)} = 1$.
 	\end{lemma}
 	\begin{proof}
 		Applying Claim~\ref{claim:bivariate_cumulants_general} with $X = \mathcal{S}_n$ and $Y=\mathcal{I}_n$ yields the formula
 		\begin{align}
 			\label{eq:first_equality_for_gamma_n_a_b}
 			\nonumber \gamma_{n}(a,b) &= \sum_{m=1}^{a} \sum_{\la\in \mathcal{P}(a,m)} \sum_{t=m}^{m+b} \sum_{\pi\in \mathcal{P}_b(m,t)} (-1)^{t-1} (t-1)! \prod_{j=1}^{t} \bE\big[\mathcal{S}_n^{|\la_j|} \,  \mathcal{I}_n^{\ell_j(\pi)}\big] \\ &= \sum_{m=1}^{a} \sum_{\la\in \mathcal{P}(a,m)} \sum_{t=m}^{m+b} \sum_{\pi\in \mathcal{P}_b(m,t)} (-1)^{t-1} (t-1)! \prod_{j=1}^{t} \Big(\sum_{\sigma^j \in \mathcal{P}(|\la_j|)} \bE\big[\mathcal{S}_n^{(\sigma^j)} \,  \mathcal{I}_n^{\ell_j(\pi)}\big] \Big) \, ,
 		\end{align}
 		where the last equality is~\eqref{eq:open_brackets_of_S_k_as_sum_over_partitions} together with the linearity of the expectation. Taken together, the collection $\{\sigma^j\}_{1\le j \le m}$ is a partition $\sigma\in \mathcal{P}(a)$ such that $\sigma \preceq \la$. Equivalently, $\la$ can be thought of a partition of $[|\sigma|]$, and under this new description we have $\sigma^j = \la_j(\sigma)$. Expanding the brackets gives
 		\begin{align*}
 			\prod_{j=1}^{t} \Big(\sum_{\sigma^j \in \mathcal{P}(|\la_j|)} \bE\big[\mathcal{S}_n^{(\sigma^j)} \,  \mathcal{I}_n^{\ell_j(\pi)}\big] \Big) &= \sum_{\sigma^1\in \mathcal{P}(|\la_1|)} \cdots \sum_{\sigma^m\in \mathcal{P}(|\la_m|)} \prod_{j=1}^t \bE\big[\mathcal{S}_n^{(\sigma^j)} \,  \mathcal{I}_n^{\ell_j(\pi)}\big] \\ &= \sum_{\substack{\sigma\in \mathcal{P}(a) \\ \sigma \preceq \la}} \prod_{j=1}^{t} \bE\big[\mathcal{S}_n^{(\la^j(\sigma))} \,  \mathcal{I}_n^{\ell_j(\pi)}\big] \, .
 		\end{align*}
 		Plugging this into~\eqref{eq:first_equality_for_gamma_n_a_b}, we can change the order of summation and get
 		\begin{align*}
 			\gamma_{n}(a,b) &= \sum_{\la\in \mathcal{P}(a)} \sum_{\pi\in \mathcal{P}_b(|\la|)} \sum_{\substack{\sigma\in \mathcal{P}(a) \\ \sigma \preceq \la}} (|\pi|-1)! (-1)^{|\pi|-1} \prod_{j=1}^{|\pi|} \bE\Big[ \mathcal{S}_n^{(\la_j(\sigma))} \, \mathcal{I}_n^{\ell_j(\pi)}\Big] \\ &= \sum_{\sigma\in \mathcal{P}(a)} \sum_{\substack{\la \in \mathcal{P}(a) \\ \sigma \preceq \la}} \sum_{\pi \in \mathcal{P}_b(|\la|)} (|\pi|-1)! (-1)^{|\pi|-1} \prod_{j=1}^{|\pi|} \bE\Big[ \mathcal{S}_n^{(\la_j(\sigma))} \, \mathcal{I}_n^{\ell_j(\pi)}\Big]
 		\end{align*}
 		which finishes the proof.
 	\end{proof}
 	We are now ready to define our Kac-Rice cumulants density. For $a,b\in \mathbb{Z}_{\ge 0}$ and a partition $\sigma\in \mathcal{P}(a)$, we define the \emph{cumulant Kac-Rice density} $F_{b,\sigma}:\mathbb{C}^{b + |\sigma|} \to \bR$ by
 		\begin{equation}
 			\label{eq:def_kac_rice_cumulant_kernel}
 			F_{b,\sigma} (\mathbf{w},\mathbf{z}) \stackrel{{\rm def}}{=} \sum_{\substack{\la \in \mathcal{P}(a) \\ \sigma \preceq \la}} \sum_{\pi \in \mathcal{P}_b(|\la|)} (|\pi|-1)! (-1)^{|\pi|-1} \prod_{j=1}^{|\pi|} \rho_{\ell_j(\pi),|\la_j|,\la_j(\sigma)} (\mathbf{w}_{\ell_j(\pi)},\mathbf{z}_{|\la_j|})
 		\end{equation}
 		where $\rho_{\ell_j(\pi),|\la_j|,\la_j(\sigma)}$ is given by~\eqref{eq:def_of_kac_rice_densities}.
 	Combining Proposition~\ref{prop:formula_for_joint_moments}, Lemma~\ref{lemma:formula_for_gamma_n_a_b_in_terms_of_partitions} we see that
 	\begin{equation}
 		\label{eq:formula_for_gamma_n_a_b_as_integral}
 		\gamma_{n}(a,b) = \sum_{m=1}^{a} \sum_{\sigma\in \mathcal{P}(a,m)} \int_{\bC^{b}} \int_{\bC^{m}} F_{b,\sigma} (\mathbf{w},\mathbf{z}) \, {\rm d}\mu^{\otimes m}(\mathbf{z}) \, {\rm d}\mu^{\otimes b}(\mathbf{w}) \, ,
 	\end{equation}
 	where $F_{b,\sigma}$ is given by~\eqref{eq:def_kac_rice_cumulant_kernel}.
	\subsection{Convergence of cumulant densities}\label{ss:convergence-cumulant-densities}
	The goal of this section is to give the proof of Theorem~\ref{thm:limit_joint_cumulant}. In view of~\eqref{eq:formula_for_gamma_n_a_b_as_integral}, the desired limit would follow if we can show that
	\begin{equation}
		\label{eq:limit_for_joint_cumulant_single_term}
		\lim_{n\to \infty} \frac{1}{n} \int_{\bC^{b}} \int_{\bC^{m}} F_{b,\sigma} (\mathbf{w},\mathbf{z}) \, {\rm d}\mu^{\otimes m}(\mathbf{z}) \, {\rm d}\mu^{\otimes b}(\mathbf{w})
	\end{equation}
	exists and is finite, for all $a,b\in \mathbb{Z}_{\ge 0}$ and $\sigma\in \mathcal{P}(a)$ fixed. To establish this limit, we record two simple claims that will help us to justify an application of the dominated convergence theorem.
	\begin{claim}
		\label{claim:upper_bound_on_cumulant_density}
		Let $a\ge 1$ and $\sigma\in \mathcal{P}(a,m)$. There exist a constants $C,c>0$, which depend only on $a$ and $b$, such that for all $(\mathbf{w},\mathbf{z}) \in \bC^{b+m}$ we have
		\begin{align*}
			&\big|F_{b,\sigma} (\mathbf{w},\mathbf{z})\big| \\ &\le C n^{m+b} \Big(\sum_{i=1}^{b} \sum_{j=1}^{m} |\log(n |w_i-z_j|^2)|^C\Big) e^{-cn\sum_{j=2}^{m}d_{\bS^2}(z_1,z_j)^2 -cn \sum_{i=1}^{b}d_{\bS^2}(z_1,w_i)^2}\, .
		\end{align*}  
		Furthermore, for $a=0$, we have the similar upper bound
		\begin{equation*}
			\big|F_{b,\emptyset} (\mathbf{w})\big| \le C n^b e^{-cn \sum_{i=2}^{b}d_{\bS^2}(w_1,w_i)^2}.
		\end{equation*}
	\end{claim}
	\begin{proof}
		We will focus on the case where $a\ge 1$, as the other case is easier and follows from the same argument. We first argue there exist partitions of $\mathbf{w}=(\mathbf{w}_{I^\prime},\mathbf{w}_{J^\prime})$ and $\mathbf{z}=(\mathbf{z}_I,\mathbf{z}_J)$ such that
		\begin{equation}
			\label{eq:good_partition_of_z_w_wrt_z_1}
			{\sf dist} \, \Big\{ \big(\mathbf{w}_I^\prime, \mathbf{z}_I\big), \big(\mathbf{w}_J^\prime, \mathbf{z}_J\big) \Big\}^2 \ge c \bigg(\sum_{j=2}^{m} d_{\bS^2} (z_1,z_j)^2 + \sum_{i=1}^{b} d_{\bS^2} (z_1,w_i)^2\bigg)
		\end{equation}
		where ${\sf dist}$ is defined by~\eqref{eq:def_spherical_distance_between_sets} and $c=c(m,b)>0$ is some constant. Indeed, let $X$ be the set of all coordinates in $(\mathbf{w},\mathbf{z})$ (so the cardinality of $X$ is $m+b$) and let $\zeta\in X$ be such that		
		\begin{equation*}
			d_{\bS^2}(z_1,\zeta) = \max_{x\in X\setminus\{z_1\}} d_{\bS^2}(z_1,x) \,. 
		\end{equation*}
		By the pigeonhole principle, we can partition the set $X=X_I \sqcup X_J$ so that
		\[
		{\sf dist} \big\{X_I, X_J\big\} \ge \frac{1}{m+b}\, d_{\bS^2}(z_1,\zeta)\, .
		\]
		This, together with the bound
		\[
		d_{\bS^2}(z_1,\zeta)^2 \ge \frac{1}{m+b} \bigg(\sum_{j=2}^{m} d_{\bS^2} (z_1,z_j)^2 + \sum_{i=1}^{b} d_{\bS^2} (z_1,w_i)^2\bigg)
		\]
		yields~\eqref{eq:good_partition_of_z_w_wrt_z_1}. Furthermore, we note that the sum $\sum_{\sigma\preceq \la} \sum_{\pi \in \mathcal{P}_b(|\la|)}$ in the definition~\eqref{eq:def_kac_rice_cumulant_kernel} of $F_{b,\sigma}$ runs in fact over all partitions $\widetilde{\pi} \in \mathcal{P}(m+b)$, by first partitioning the $m$ blocks of $\sigma$ (this partition is dictated by $\la$) and then adding to each block the elements from $\{m+1,\ldots,m+b\}$. With this identification we have $|\widetilde{\pi}| = |\pi|$, and so we can apply Claim~\ref{claim:approximate_factor_for_moments_implies_small_cumulant} together with $M$ being the bound in Theorem~\ref{thm:local_bounds_on_dentity} and $\eps$ the bound from Theorem~\ref{thm:clustering_of_density}, applied with the partition~\eqref{eq:good_partition_of_z_w_wrt_z_1}. We get
		\begin{equation*}
			\big|F_{b,\sigma} (\mathbf{w},\mathbf{z})\big| \lesssim n^{m+b} \prod_{1\le i\le b} |\log(n\Delta^\prime_i)|^{m+b} \exp\bigg(-cn\Big(\sum_{j=2}^{m} d_{\bS^2} (z_1,z_j)^2 + \sum_{i=1}^{b} d_{\bS^2} (z_1,w_i)^2\Big)\bigg)\, ,
		\end{equation*}
		where $\Delta_i^\prime$ is given by~\eqref{eq:def_of_Delta_and_Delta_prime}. The desired bound now follows from the triangle inequality, which shows that
		\[
		|\log|n\Delta_i^\prime|| \lesssim \sum_{j=1}^{m} |\log(n |w_i-z_j|^2)|
		\]
		and the proof is complete.
	\end{proof}

	\begin{claim}
		\label{claim:pointwise_limit_for_cumulant_density}
		For any $(\mathbf{w},\mathbf{z}) \in \bC_{\neq}^{b+m}$ we have the pointwise limits
		\begin{equation*}
			\lim_{n\to \infty} n^{-b-m} F_{b,\sigma} \Big(\frac{\mathbf{w}}{\sqrt{n}},\frac{\mathbf{z}}{\sqrt{n}}\Big) = F_{b,\sigma}^{G} (\mathbf{w},\mathbf{z}) \, ,
		\end{equation*}
		where $F_{b,\sigma}^G$ is defined via the relation~\eqref{eq:def_kac_rice_cumulant_kernel} with $\rho$ replaced by $\rho^G$ given by~\eqref{eq:def_of_kac_rice_densities_GEF} (the limiting density of the G.E.F.).
	\end{claim}
	\begin{proof}
		In view of the definition~\eqref{eq:def_kac_rice_cumulant_kernel} of $F_{b,\sigma}$, the desired limit will follow once we show that for all $(\mathbf{w},\mathbf{z})\in \bC_{\neq}^{\ell+m}$ fixed we have 
		\begin{equation}
			\label{eq:limiting_intensity_is_GEF}
			\lim_{n \to \infty} n^{-\ell - m} \rho_{\ell,m,p}\left(\frac{\mathbf{w}}{\sqrt{n}},\frac{\mathbf{z}}{\sqrt{n}} \right)= \rho_{\ell,m,p}^G(\mathbf{w},\mathbf{z}) \, .
		\end{equation}
		Indeed, we recall relation~\eqref{eq:Kac_Rice_densities_non_conditional}, which states that
		\begin{equation*}
			n^{-\ell-m}\rho_{\ell,m,p} \Big(\frac{\mathbf{w}}{\sqrt{n}},\frac{\mathbf{z}}{\sqrt{n}}\Big) =  \int_{\bC^{m+\ell}} \varphi_n(\mathbf{0},\eta,\eta^\prime)  \Big(\prod_{t=1}^{\ell} \log|\eta_t^\prime| \Big) \Big(\prod_{j=1}^{m} |\eta_j|^2 \, \log^{p_j}|\eta_j|\Big)  \, {\rm d}m(\eta,\eta^\prime)
		\end{equation*}
		where $\varphi_n$ is the density~\eqref{eq:gaussian_density_formula} of the Gaussian random variables with covariance matrix $\Ga_{n}$ (which corresponds to the Gaussian random variables given by~\eqref{eq:complex_gaussians_we_consider}). By Claim~\ref{claim:L_n_converge_to_L_infty}, we know that the sequence of matrices $\Ga_{n}$ converges, as $n\to \infty$, to a limiting matrix $\Ga_\infty$ which is the covariance matrix of the Gaussian random variables which correspond to the G.E.F. (given by~\eqref{eq:gaussians_we_consider_GEF}). Finally, Lemma~\ref{lemma:rational_function_and_L_n_are_comparable} implies that for each fixed $(\mathbf{w},\mathbf{z})\in \bC_{\neq}^{\ell+m}$, the matrices $\Ga_n$ and $\Ga_\infty$ are uniformly non-singular as a sequence in $n$, and hence~\eqref{eq:limiting_intensity_is_GEF} follows from a simple application of the dominated convergence theorem. 
	\end{proof}

	\begin{proof}[Proof of Theorem~\ref{thm:limit_joint_cumulant}]
		The desired result will follow once we establish the limit~\eqref{eq:limit_for_joint_cumulant_single_term}. By applying the obvious scaling, it is equivalent to showing that
		\begin{equation*}
			\lim_{n\to \infty} \frac{1}{n^{m+b+1}} \int_{\bC^m} \int_{\bC^b} F_{b,\sigma} \Big(\frac{\mathbf{w}}{\sqrt{n}},\frac{\mathbf{z}}{\sqrt{n}}\Big) \prod_{j=1}^{m} \frac{{\rm d}m(z_j)}{\pi(1+|z_j|^2/n)^2}  \prod_{i=1}^{b} \frac{{\rm d}m(w_i)}{\pi(1+|w_i|^2/n)^2}
		\end{equation*}
		exists and is finite.  Assuming that $m\ge 1$, we use the rotation invariance of $F_{b,\sigma}$, guaranteed by Lemma~\ref{lemma:spherical_invariance}, and rotate the vector $(\mathbf{w},\mathbf{z})$ so that $z_1 = 0$. If $m=0$, we rotate $\mathbf{w}$ so that $w_1 = 0$ and argue in the same way. Since
		\begin{equation*}
			\int_{\bC}\frac{{\rm d}m(z_1)}{\pi(1+|z_1|^2/n)^2} = n\,,
		\end{equation*}
		we see that existence of the limit~\eqref{eq:limit_for_joint_cumulant_single_term} is equivalent to existence of the limit
		\begin{equation}
			\label{eq:limit_for_single_term_cumulant_after_rotation}
			\lim_{n\to \infty} \frac{1}{n^{m+b}} \int_{\bC^{m-1}} \int_{\bC^b} F_{b,\sigma} \Big(\frac{\mathbf{w}}{\sqrt{n}},\frac{\mathbf{z}^\prime}{\sqrt{n}}\Big) \prod_{j=2}^{m} \frac{{\rm d}m(z_j)}{\pi(1+|z_j|^2/n)^2}  \prod_{i=1}^{b} \frac{{\rm d}m(w_i)}{\pi(1+|w_i|^2/n)^2}\, ,
		\end{equation}
		where $\mathbf{z}^\prime \stackrel{{\rm def}}{=} (0,z_2,\ldots,z_m)$. The convergence claimed in~\eqref{eq:limit_for_single_term_cumulant_after_rotation} will follow from an application of the dominated convergence theorem which we now justify. Let 		
		\begin{equation*}
		g_{n,\sigma,b}(\mathbf{w},\mathbf{z}^\prime) \stackrel{{\rm def}}{=} n^{-b-m} F_{b,\sigma} \Big(\frac{\mathbf{w}}{\sqrt{n}},\frac{\mathbf{z}^\prime}{\sqrt{n}}\Big) \prod_{j=2}^{m} \frac{1}{\pi(1+|z_j|^2/n)^2} \prod_{i=1}^{b} \frac{1}{\pi(1+|w_i|^2/n)^2} \, .
		\end{equation*}
		As we have the pointwise convergence of $g_{n,\sigma,b}$ (Claim~\ref{claim:pointwise_limit_for_cumulant_density}) it is enough to find an integrable majorant. By Claim~\ref{claim:upper_bound_on_cumulant_density}, we have the upper bound
		\begin{align}
			\label{eq:first_bound_on_g_n_sigma_b}
				|g_{n,\sigma,b}(\mathbf{w},\mathbf{z}^\prime)|  \lesssim \Big(\sum_{i=1}^{b} \sum_{j=1}^{m} |\log |w_i-z_j||^C\Big) \prod_{j=2}^{m} \frac{e^{-c|z_j|^2/(1+|z_j|^2/n)}}{(1+|z_j|^2/n)^2} \prod_{i=1}^{b} \frac{e^{-c|w_i|^2/(1+|w_i|^2/n)}}{(1+|w_i|^2/n)^2}\, .
		\end{align}
		It is evident that $\exp(-cx) \lesssim (1+x)^{-2}$
		for all $x\ge0$.
		Applying this simple observation with $x = \frac{t}{1+t/n}$, we see that 
		\begin{equation*}
			\frac{e^{-ct/(1+t/n)}}{(1+t/n)^2} \lesssim \frac{1}{\big((1+t/n)(1+\frac{t}{1+t/n})\big)^2} \lesssim \frac{1}{(1+t)^2}
		\end{equation*}
		uniformly in $n\ge 1$. Plugging $t=|z_j|^2$ and $t=|w_i|^2$ in the above inequality and applying it to the bound~\eqref{eq:first_bound_on_g_n_sigma_b}, we see that
		\begin{equation*}
			|g_{n,\sigma,b}(\mathbf{w},\mathbf{z}^\prime)| \lesssim \Big(\sum_{i=1}^{b} \sum_{j=1}^{m} |\log |w_i-z_j||^C\Big) \prod_{j=2}^{m} \frac{1}{(1+|z_j|^2)^2} \prod_{i=1}^{b} \frac{1}{(1+|w_i|^2)^2}\, .
		\end{equation*}
		As the right hand side of the above is integrable on $\bC^{m-1} \times \bC^{b}$ (with respect to the usual Lebesgue measure), we can apply the dominated convergence theorem and see that the limit~\eqref{eq:limit_for_single_term_cumulant_after_rotation} exists and is equal to 
		\begin{equation*}
			\int_{\bC^{m-1}} \int_{\bC^b} F_{b,\sigma}^{G} (\mathbf{w},\mathbf{z}^\prime) \prod_{j=2}^{m} {\rm d}m(z_j) \prod_{i=1}^{b} {\rm d}m(w_i)
		\end{equation*}
		which is finite. As a consequence, we proved that the limit~\eqref{eq:limit_for_joint_cumulant_single_term} exists and, in view of~\eqref{eq:formula_for_gamma_n_a_b_as_integral}, this proves Theorem~\ref{thm:limit_joint_cumulant}.
	\end{proof}
	\begin{remark} \label{remark:cumulants}
		We note that en route to proving Theorem~\ref{thm:limit_joint_cumulant}, we also identified the limits for the joint cumulants. We proved that for all $a\ge 1$ and $b\in \bZ_{\ge 0}$ we have
		\begin{equation}
			\label{eq:limit_for_joint_cumulants_explicit}
			\lim_{n\to \infty} \frac{\gamma_{n}(a,b)}{n} = \sum_{\sigma\in \mathcal{P}(a)}\frac{1}{\pi^{|\sigma|+b-1}} \int_{\bC^{|\sigma|-1}} \int_{\bC^b} F_{b,\sigma}^{G} (\mathbf{w},\mathbf{z}^\prime) \prod_{j=2}^{|\sigma|} {\rm d}m(z_j) \prod_{i=1}^{b} {\rm d}m(w_i)
		\end{equation}
		where $F_{b,\sigma}^{G}$ is the limiting cumulant density for the G.E.F. (the scaling limit for our polynomial). Furthermore, in the case $a=0$ and $b\ge 1$, we have a similar expression
		\begin{equation}
			\label{eq:limit_for_joint_cumulants_explicit_a_is_0}
			\lim_{n\to \infty} \frac{\gamma_{n}(0,b)}{n} = \frac{1}{\pi^{b-1}} \int_{\bC^{b-1}} F_{b,\emptyset}^{G} (0,w_2,\ldots,w_b) \prod_{i=2}^{b} {\rm d}m(w_i)
		\end{equation}
		In Section~\ref{sec:variance_asymptotics} we simplify~\eqref{eq:limit_for_joint_cumulants_explicit} in the case where $a+b = 2$, and with that give some explicit integral expressions for the limiting variance of the logarithmic energy. In particular, we prove that it is non-degenerate.	
	\end{remark}
	\section{Asymptotic of the variance. Proof of Theorem~\ref{thm:variance_after_split}}
	\label{sec:variance_asymptotics}
	In this section we compute the limiting variance of the logarithmic energy (the expectation was computed in Lemma~\ref{lemma:expectation_of_E_n} above). The strategy for the proof, as mentioned in the statement of Theorem~\ref{thm:variance_after_split}, is to compute separately the limits:
	\begin{align}
		\label{eq:def_of_c_1} 
		c_1 &\stackrel{{\rm def}}{=} \lim_{n\to \infty} \frac{\sf Var\big[\mathcal{I}_n\big]}{n} = \lim_{n\to \infty} \frac{\gamma_{n}(0,2)}{n}\, , \\  \label{eq:def_of_c_2} c_2 &\stackrel{{\rm def}}{=} \lim_{n\to \infty} \frac{\sf Var\big[\mathcal{S}_n\big]}{n} = \lim_{n\to \infty} \frac{\gamma_{n}(2,0)}{n}\, , \\ \label{eq:def_of_c_3} c_3 &\stackrel{{\rm def}}{=} \lim_{n\to \infty} \frac{\sf Cov\big[\mathcal{I}_n,\mathcal{S}_n\big]}{n} = \lim_{n\to \infty} \frac{\gamma_{n}(1,1)}{n} \, .
	\end{align} 
	The limiting variance for the logarithmic energy $\mathcal{E}_n$ will then follow with the aid of Lemma~\ref{lemma:rewrite_energy_as_I_minus_S}.  The computations performed below combined with numerical integration shows that
	\begin{align*}
		&c_1 = \frac{\zeta(3)}{4} \approx 0.300514 \\ &c_2 =  \frac{1}{4} \, \Big( \zeta(2) + \gamma(\gamma-2) + I_1 + I_2 + I_3\Big) \approx 0.476091  \\ & c_3 = \frac{1}{4} \Big(\zeta(2) - J_1\Big) \approx 0.34295	
	\end{align*}
	where $I_1,I_2,I_3$ and $J_1$ are explicit integrals given by~\eqref{eq:def_of_Integrals} and~\eqref{eq:def_of_J_1} and $\zeta(\cdot)$ is the Riemann zeta function. As indicated in Theorem~\ref{thm:variance_after_split}, we have the relation
	\begin{equation}
		\label{eq:explicit_formula_for_c_ast}
		c_\ast = c_1+c_2 - 2c_3 \approx 0.0907056 \, .
	\end{equation}
	Finally, we prove in Appendix~\ref{sec:variance_is_positive} that the limiting variance is  indeed positive, by showing that $c_1\not=c_2$  without using computer aided computations. Note that the existence of the limits~\eqref{eq:def_of_c_1},~\eqref{eq:def_of_c_2} and~\eqref{eq:def_of_c_3} is guaranteed by Theorem~\ref{thm:limit_joint_cumulant}. Furthermore, the limits are the corresponding integrals of the cumulant densities for the limiting G.E.F., see the remark at the end of Section~\ref{sec:cumulant_densities}. This observation is the starting point for the computations we perform in this section.
	\subsection{Wiener chaos decomposition} The basic tool which allows us to compute the limits above is an orthogonal decomposition scheme, which is valid in general Gaussian Hilbert spaces and is commonly known as \emph{Wiener chaos}, see~\cite[Chap.~2]{Janson}. We record below the basic facts that we will use. First, we find the coefficients in the orthogonal expansion of the functions $\log x$ and $x\log x$ in the Hilbert space $L^2(\bR_{\ge0},e^{-x}\, {\rm d}x)$. The Laguerre polynomials
	\begin{equation*}
		\label{eq:def_Laguerre_polynomial}
		{\sf L}_j(x) \stackrel{{\rm def}}{=} \frac{e^x}{j!} \frac{\partial^j}{\partial x^j} \big(x^j e^{-x}\big)\, ,
	\end{equation*}
	for $m\ge 0$, form an orthonormal basis in this Hilbert space. The relevance of these polynomials becomes evident from the following simple observation: if $(Z_1,Z_2) \in \bC^2$ a mean-zero random complex Gaussian vector with $\bE|Z_1|^2 = \bE|Z_2|^2 = 1$ and $\bE[Z_1\overline{Z_2}] = \theta$, then
	\begin{equation}
		\label{eq:orthogonality_of_laguerre_poly_on_gaussians}
		\bE\big[{\sf L}_m(|Z_1|^2)\, {\sf L}_k(|Z_2|^2)\big] = \begin{cases}
		|\theta|^{2m} & m=k, \\ 0 & m\not= k,
		\end{cases}
	\end{equation}
	see for instance~\cite[Lemma~3.5.2]{HKPV} or~\cite[Theorem~3.9]{Janson}. Throughout this section, we use $\gamma$ for the Euler-Mascheroni constant.
		
	\begin{claim}
		\label{claim:chaos_expansion}
		Let $\alpha_0 = -\gamma $ and $\alpha_j = -\frac{1}{j}$ for $j\ge 1$, then 
		\begin{equation}
			\label{eq:chaos_expansion_for_logx}
			\log x  = \sum_{j\ge 0} \alpha_j {\sf L}_j(x)
		\end{equation}
		in the $L^2(\bR_{\ge0},e^{-x} \, {\rm d}x)$ sense. Similarly, if we set $\beta_0 = 1-\gamma$, $\beta_1 = \gamma-2$ and $\beta_j = \frac{1}{j(j-1)}$ for $j\ge 2$, then 
		\begin{equation}
			\label{eq:chaos_expansion_for_xlogx}
			x\log x = \sum_{j\ge 0} \beta_j {\sf L}_j(x)
		\end{equation}
		in the $L^2(\bR_{\ge0},e^{-x} \, {\rm d}x)$ sense.
	\end{claim}
	\begin{proof}
		We will only compute the expansion~\eqref{eq:chaos_expansion_for_xlogx}, as the expansion~\eqref{eq:chaos_expansion_for_logx} follows in a similar way (it also exists in the proof of~\cite[Lemma~3.5.2]{HKPV}). For $j\ge 2$ we integrate by parts and see that
		\begin{align*}
			\beta_j &= \int_{0}^{\infty} (x \log x) \, {\sf L}_j(x) e^{-x} \, {\rm d}x \\ &= \frac{1}{j!}\int_{0}^{\infty} (x \log x) \frac{\partial^j}{\partial x^j} \big(x^j e^{-x}\big) \, {\rm d}x \\ &= \frac{(-1)^j}{j!}\int_{0}^{\infty} \frac{\partial^j(x \log x)}{\partial x^j}  \, x^j e^{-x} \, {\rm d}x = \frac{(j-2)!}{j!} \int_{0}^{\infty} x e^{-x}\, {\rm d}x = \frac{1}{j(j-1)}\, .
		\end{align*}
		The values of $\beta_0$ and $\beta_1$ follow from direct integration and the claim follows.
	\end{proof}
 	\subsection{Computations}
 	Recall that the G.E.F. is given by
 	\begin{equation*}
 		g(z) = \sum_{j\ge 0} a_j\frac{z^j}{\sqrt{j!}} 
 	\end{equation*}
 	where the $\{a_j\}$ are i.i.d.\ standard complex Gaussians, and that $\widehat{g}(z) = e^{-|z|^2/2} g(z)$.
 	\begin{lemma}
 		\label{lemma:value_for_c_1}
 		Let $c_1$ be given by~\eqref{eq:def_of_c_1}. Then $$c_1 = \frac{\zeta(3)}{4} \, ,$$ where $\zeta(\cdot)$ is the Riemann zeta function.
 	\end{lemma}
 	\begin{proof}
 		By~\eqref{eq:limit_for_joint_cumulants_explicit_a_is_0}, we see that
 		\begin{align*}
 			c_1 &= \frac{1}{\pi} \int_{\bC} \bE\big[(\log|\widehat{g}(0)| - \alpha_0/2) (\log|\widehat{g}(z)| - \alpha_0/2) \big] \, {\rm d}m(z) = \frac{1}{\pi} \int_{\bC} {\sf Cov} \big[\log|\widehat{g}(0)|, \log|\widehat{g}(z)| \, \big] \, {\rm d}m(z) 
 		\end{align*}
 		where $\alpha_0 = -\gamma = \bE\log|Z|^2$ with $Z$ being a standard complex Gaussian. It is straightforward to check that
 		\begin{equation*}
 			\bE\big[\widehat{g}(0) \overline{\widehat{g}(z)} \, \big] = e^{-|z|^2/2}\, ,
 		\end{equation*}
 		and so, by combining~\eqref{eq:orthogonality_of_laguerre_poly_on_gaussians} with the expansion~\eqref{eq:chaos_expansion_for_logx} we see that
 		\begin{align*}
 			c_1 &= \frac{1}{4\pi} \int_{\bC} {\sf Cov} \big[\log|\widehat{g}(0)|^2, \log|\widehat{g}(z)|^2 \, \big] \, {\rm d}m(z) \\&= \frac{1}{4\pi} \int_{\bC} \bigg(\sum_{j\ge 1} \alpha_j^2 e^{-j|z|^2} \bigg) \, {\rm d}m(z) \\ &= \sum_{j\ge 1} \frac{1}{4j^2} \bigg(\frac{1}{\pi} \int_{\bC} e^{-j|z|^2} \, {\rm d}m(z)\bigg) = \sum_{j\ge 1} \frac{1}{4j^3} = \frac{\zeta(3)}{4}\, .		
 		\end{align*}
 	\end{proof}
 	We continue our computations and simplify the expression for the limiting covariance between $\mathcal{I}_n$ and $\mathcal{S}_n$.  Set
 	\begin{equation}
 		\label{eq:def_of_J_1}
 		J_1 \stackrel{{\rm def}}{=} \int_{0}^\infty \Psi\Big(\frac{s}{e^s - 1}\Big) \, {\rm d}s \, ,
 	\end{equation}
 	 with $\displaystyle \Psi(x) = \sum_{j=2}^\infty \frac{x^j}{j^2(j-1)}$ for $x\in [0,1]$.
 	\begin{lemma}
 		\label{lemma:value_of_c_3}
 		Let $c_3$ be given as in~\eqref{eq:def_of_c_3}. Then
 		\[
 		c_3 = \frac{\pi^2}{24} - \frac{1}{4} \, J_1  \, ,
 		\]
 		where $J_1$ is  given by~\eqref{eq:def_of_J_1}.
 	\end{lemma}
 	Before proving Lemma~\ref{lemma:value_of_c_2}, we will need the following simple claim.
 	\begin{claim}
 		\label{claim:expression_for_covariance_of_gaussians_logx_xlogx}
 		Let $(Z_1,Z_2)$ be a complex Gaussian vector with $\bE|Z_1|^2 = 1$, $\bE|Z_2| = \sigma^2$ and $\bE\big[Z_1\overline{Z_2}\big] = \theta$. Then
 		\begin{equation*}
 			\bE\big[|Z_1|^2 \log|Z_1|^2 \, \log|Z_2|^2\big] = \sum_{j=0}^{\infty} \alpha_j\beta_j |\theta/\sigma|^{2j} + \beta_0\log \sigma^2\, ,
 		\end{equation*}
 		where $\{\alpha_j\}$ and $\{\beta_j\}$ are as in Claim~\ref{claim:chaos_expansion}.
 	\end{claim}
 	\begin{proof}
 		Let $Z=Z_1$ and $W=Z_2/\sigma$. Using Claim~\ref{claim:chaos_expansion} and the orthogonality relation~\eqref{eq:orthogonality_of_laguerre_poly_on_gaussians}, we see that
 		\begin{align*}
 			\bE\big[|Z_1|^2 \log|Z_1|^2 \, \log|Z_2|^2\big] &= \bE\big[|Z|^2 \log|Z|^2 \, \log|W|^2\big] + \bE\big[|Z|^2 \log|Z|^2 \big]  \log\sigma^2 \\ &= \sum_{j\ge 0} \alpha_j\beta_j |\theta/\sigma|^{2j} + \beta_0 \log \sigma^2
 		\end{align*}
 		as desired.
 	\end{proof}
 	\begin{proof}[Proof of Lemma~\ref{lemma:value_of_c_3}]
 		By relation~\eqref{eq:limit_for_joint_cumulants_explicit} with $a=b=1$, we see that
 		\begin{equation}
 			\label{eq:first_expression_for_c_3}
 			c_3 = \frac{1}{4\pi} \int_{\bC} \Big(\bE\Big[|D\widehat{g}(0)|^2 \log|D\widehat{g}(0)|^2 \log|\widehat{g}(z)|^2 \mid g(0) = 0 \, \Big] - \alpha_0\beta_0\Big)\, {\rm d}m(z)
 		\end{equation}
 		with $\alpha_0 = \bE\big[\log|Z|^2\big]$ and $\beta_0 = \bE\big[|Z|^2\log|Z|^2\big]$. A straightforward inspection yields that
 		\[
 		\Cov\big(\widehat{g}(0),D\widehat{g}(0),\widehat{g}(z)\big) = \begin{bmatrix}
 		1 & 0 & e^{-|z|^2/2} \\ 0 & 1 & ze^{-|z|^2/2} \\ e^{-|z|^2/2} & \overline{z}e^{-|z|^2/2} & 1
 		\end{bmatrix} \, .
 		\]
		Therefore, conditioned on the event $\{g(0) = 0\}$, the vector $\big(D\widehat{g}(0),\widehat{g}(z)\big)$ has a complex Gaussian law with the covariance matrix
		\begin{equation*}
			\begin{bmatrix}
			1 & ze^{-|z|^2/2} \\ \overline{z} e^{-|z|^2/2} & 1 - e^{-|z|^2}
			\end{bmatrix}\, .
		\end{equation*}
		We apply Claim~\ref{claim:expression_for_covariance_of_gaussians_logx_xlogx} with $Z_1 = D\widehat{g}(0)$ and $Z_2 = \widehat{g}(z)$, conditioned on the event $\{g(0) = 0\}$, and see that
		\begin{align*}
			\bE\Big[|D\widehat{g}(0)|^2 &\log|D\widehat{g}(0)|^2 \log|\widehat{g}(z)|^2 \mid g(0) = 0 \, \Big] - \alpha_0\beta_0 = \sum_{m=1}^{\infty} \alpha_m \beta_m \bigg(\frac{|z|^2}{e^{|z|^2}-1}\bigg)^m + \beta_0 \log\big(1-e^{-|z|^2}\big)\, .
		\end{align*}
		Plugging this relation into~\eqref{eq:first_expression_for_c_3} and integrating we see that
		\begin{align*}
			c_3 &= \frac{1}{4\pi} \int_{\bC} \bigg(\sum_{j=1}^{\infty} \alpha_j \beta_j \bigg(\frac{|z|^2}{e^{|z|^2}-1}\bigg)^j + \beta_0 \log\big(1-e^{-|z|^2}\big)\bigg) \, {\rm d}m(z) \\ &= \frac{1}{4} \bigg(\int_{0}^{\infty} \sum_{j=1}^{\infty} \alpha_j \beta_j \bigg(\frac{s}{e^{s}-1}\bigg)^j \, {\rm d}s + \beta_0 \int_{0}^{\infty} \log\big(1-e^{-s}\big) \, {\rm d}s \bigg) \, .
		\end{align*}
		By solving the elementary integrals
		\begin{equation*}
			\int_{0}^{\infty} \log\big(1-e^{-s}\big) \, {\rm d}s = -\sum_{j\ge 1} \int_{0}^{\infty} \frac{e^{-sj}}{j} \, {\rm d}s = -\sum_{j=1}^{\infty} \frac{1}{j^2} = -\frac{\pi^2}{6} \, ,
		\end{equation*}
		and
		\begin{equation*}
			 \int_{0}^{\infty} \frac{s}{e^s-1} \, {\rm d}s = \sum_{j\ge 0} \int_{0}^{\infty} s e^{-s(j+1)} \, {\rm d}s = \sum_{j=0}^\infty \frac{1}{(j+1)^2} = \frac{\pi^2}{6} \, ,
		\end{equation*}
		we see that
		\begin{align*}
			c_3 &= \frac{1}{4} \bigg( -\frac{\pi^2}{6} \big(\beta_0 + \beta_1\big) + \int_{0}^{\infty} \sum_{j=2}^\infty \alpha_j\beta_j \bigg( \frac{s}{e^s -1}\bigg)^j \, {\rm d}s\bigg) \\ &= \frac{\pi^2}{24} - \frac{1}{4}\int_{0}^{\infty} \sum_{j=2}^\infty \frac{1}{j^2(j-1)} \bigg( \frac{s}{e^s -1}\bigg)^j \, {\rm d}s
		\end{align*}
		and the lemma follows.
 	\end{proof}
 	At last, we are ready to give an expression for the limiting variance of $\mathcal{S}_n$. It will be given in term of these explicit integrals:
 	\begin{align}
 		\label{eq:def_of_Integrals}
 		I_1 &= \int_{0}^{\infty} \bigg[\big(1-e^{-s}\big)^{-1} \Big(1-\frac{s}{e^s-1}\Big)^2\Big(\beta_0+\log\Big(1-\frac{s}{e^s-1}\Big)\Big)^2 - \beta_0^2 \, \bigg]\, {\rm d}s \, ; \\ \nonumber I_2 &= \int_{0}^{\infty} \big(1-e^{-s}\big)^{-1} \Big(1-\frac{s}{e^s-1}\Big)^2 \Big(\beta_1 - \log\Big(1-\frac{s}{e^s-1}\Big)\Big)^2 \bigg(\frac{e^{-s/2}(1-s-e^{-s})}{1-(s+1)e^{-s}}\bigg)^2  \, {\rm d}s \, ; \\  \nonumber I_3 &= \int_{0}^{\infty} \big(1-e^{-s}\big)^{-1} \Big(1-\frac{s}{e^s-1}\Big)^2 \Bigg(\sum_{j=2}^{\infty} \beta_j^2 \, \bigg(\frac{e^{-s/2}(1-s-e^{-s})}{1-(s+1)e^{-s}}\bigg)^{2j} \, \Bigg) \, {\rm d}s \, .
 	\end{align}
 	It is easy to see that all three integrals are absolutely convergent. Furthermore, numerical integration gives the approximations:
 	\begin{align*}
 		I_1 & \approx -0.570754 \\ I_2 & \approx 1.53694 \\ I_3 & \approx 0.114499
 	\end{align*}
 	In Appendix~\ref{sec:variance_is_positive} we give some rigorous bounds on these integrals, which are not computer aided.
 	\begin{lemma}
 		\label{lemma:value_of_c_2}
 		Let $c_2$ be given as~\eqref{eq:def_of_c_2}. Then
 		\begin{equation*}
 			c_2 = \frac{1}{4} \, \bigg( \frac{\pi^2}{6} + \gamma(\gamma-2) + I_1 + I_2 + I_3\bigg)
 		\end{equation*}
 		where $I_1,I_2,I_3$ are given by~\eqref{eq:def_of_Integrals}.
 	\end{lemma}
 	Similarly to the proof of Lemma~\ref{lemma:value_of_c_3}, we record the relevant correlation computation as a separate claim.
 	\begin{claim}
 		\label{claim:expression_for_covariance_of_gaussians_xlogx_xlogx}
 		Let $(Z_1,Z_2)$ be a complex Gaussian vector with $\bE|Z_1|^2 = \bE|Z_2| = \sigma^2$ and $\bE\big[Z_1\overline{Z_2}\big] = \theta \, \sigma^2$. Then
 		\begin{align*}
 			\bE\left[|Z_1|^2 \log|Z_1|^2 \, |Z_2|^2\log|Z_2|^2\right] = \sigma^4 \, \left( \left(\beta_0 + \log \sigma^2\right)^2 + \left(\beta_1 - \log \sigma^2\right)^2 |\theta|^2 + \sum_{j=2}^\infty \beta_j^2 |\theta|^{2j} \right) \, .
 		\end{align*}
 	\end{claim} 
 	\begin{proof}
 		Put $W_1 = Z_1/\sigma$ and $W_2 = Z_2 / \sigma$, then
 		\begin{align}
 			\label{eq:opening_brackets_for_correlations_ZlogZ}
 			\nonumber\bE\big[|Z_1|^2 \log|Z_1|^2 \, |Z_2|^2\log|Z_2|^2\big] &= \sigma^4 \, \bE\Big[ |W_1|^2 \big(\log|W_1|^2 + \log \sigma^2 \big) \, |W_2|^2 \big(\log|W_2|^2 + \log \sigma^2 \big)\Big] \\  &= \sigma^4 \bigg((\log\sigma^2)^2 \, \bE\big[|W_1|^2 |W_2|^2 \big] +2(\log\sigma^2) \, \bE\big[|W_1|^2\log|W_1|^2 |W_2|^2 \big] \nonumber \\ & \qquad \qquad \qquad \qquad + \bE\big[|W_1|^2\log|W_1|^2 |W_2|^2\log|W_2|^2\big] \bigg) \, .
 		\end{align}
 		We expand in the basis of Laguerre polynomials with~\eqref{eq:chaos_expansion_for_xlogx} and $x= L_0(x) - L_1(x)$ and apply the orthogonality relation~\eqref{eq:orthogonality_of_laguerre_poly_on_gaussians}. This yields
 		\begin{align*}
 			&\bE\big[|W_1|^2 |W_2|^2 \big] = 1 + |\theta|^2 \, ; \\ &\bE\big[|W_1|^2\log|W_1|^2 |W_2|^2 \big] = \beta_0 - \beta_1 |\theta|^2 \, ;\\ & \bE\big[|W_1|^2\log|W_1|^2 |W_2|^2\log|W_2|^2\big]= \sum_{j=0}^\infty \beta_j^2 |\theta|^{2j}\, ,
 		\end{align*}
 		and plugging into~\eqref{eq:opening_brackets_for_correlations_ZlogZ} gives the claim.
 	\end{proof}
 	\begin{proof}[Proof of Lemma~\ref{lemma:value_of_c_2}]
 		By the case $a=2$ and $b=0$ of the limit~\eqref{eq:limit_for_joint_cumulants_explicit} (note that in this case there are two possible partitions), we see that
 		\begin{equation}
 			\label{eq:split_for_c_2}
 			c_2 = c_2^\prime + c_2^\dprime\, ,
 		\end{equation}
 		where
 		\begin{align*}
 			c_2^\prime &\stackrel{{\rm def}}{=} \bE\Big[(\log|D\widehat{g}(0)|)^2 \, |D\widehat{g}(0)|^2 \mid g(0) = 0 \Big] \, ; \\ c_2^\dprime &\stackrel{{\rm def}}{=} \frac{1}{4\pi} \int_{\bC} \bigg( \frac{\bE\big[\log|D\widehat{g}(0)|^2|D\widehat{g}(0)|^2\, \log|D\widehat{g}(z)|^2|D\widehat{g}(z)|^2 \mid g(0) = g(z) = 0 \big]}{\det\big[\Cov\big(\widehat{g}(0),\widehat{g}(z)\big)\big]} - \beta_0^2\bigg) \, {\rm d}m(z)\, .
 		\end{align*}
 		Starting with the term $c_2^\prime$, we note that $\widehat{g}(0)$ and $D\widehat{g}(0)$ are independent, with $\bE|D\widehat{g}(0)|^2 = 1$. A routine computation gives
 		\begin{align*}
 			c_2^\prime &= \bE\Big[(\log|D\widehat{g}(0)|)^2 \, |D\widehat{g}(0)|^2 \Big] = \frac{1}{\pi} \int_{\bC} \log^2|z| \, |z|^2 \, e^{-|z|^2} \, {\rm d}m(z) = \frac{1}{4}\int_{0}^\infty \log^2(s)\, s\, e^{-s} \, {\rm d}s  \\ &= \frac{1}{4} \int_{0}^{\infty} e^{-s} \Big(\log^2(s) + 2 \log s\Big)\, {\rm d}s = \frac{1}{4} \big(\Ga^\dprime(1) + 2 \Ga^\prime(1)\big) = \frac{1}{4} \big(\gamma^2 - 2\gamma + \zeta(2) \big) \, .
 		\end{align*}
 		To deal with the integral term in~\eqref{eq:split_for_c_2} we need to do some linear algebra. For $z\in \bC$ we have
 		\begin{equation*}
 			\Cov\big(\widehat{g}(0), \widehat{g}(z) , D\widehat{g}(0),D \widehat{g}(z) \big) = \begin{bmatrix}
 			1 & e^{-|z|^2/2} & 0 & 0 \\ e^{-|z|^2/2} & 1  & z e^{-|z|^2/2} & z \\ 0 & \overline{z} e^{-|z|^2/2} & 1 & e^{-|z|^2/2} \\ 0 & \overline{z} & e^{-|z|^2/2} & 1 + |z|^2 
 			\end{bmatrix}\, .
 		\end{equation*}
 		Thus, conditioned on the event $\{g(0) = g(z) = 0\}$, the vector $\big(D\widehat{g}(0),D \widehat{g}(z) \big)$ has a Gaussian law with mean zero and covariance matrix
 		\begin{equation*}
 			\begin{bmatrix}
 			\frac{1-(1+|z|^2)\, e^{-|z|^2}}{1-e^{-|z|^2}} & \frac{e^{-|z|^2/2}\big(1-|z|^2-e^{-|z|^2}\big)}{1-e^{-|z|^2}} \\ \frac{e^{-|z|^2/2}\big(1-|z|^2-e^{-|z|^2}\big)}{1-e^{-|z|^2}} & \frac{1-(1+|z|^2)\, e^{-|z|^2}}{1-e^{-|z|^2}}
 			\end{bmatrix}\, .
 		\end{equation*}
 		Claim~\ref{claim:expression_for_covariance_of_gaussians_xlogx_xlogx} shows that
 		\begin{align*}
 			\bE\Big[\log|D\widehat{g}(0)|^2|D\widehat{g}(0)|^2\,& \log|D\widehat{g}(z)|^2|D\widehat{g}(z)|^2 \mid g(0) = g(z) = 0 \Big] \\ &= \sigma^4 \, \Bigg( \big(\beta_0 + \log \sigma^2\big)^2 + \big(\beta_1 - \log \sigma^2\big)^2 |\theta|^2 + \sum_{j=2}^\infty \beta_j^2 |\theta|^{2j} \Bigg)\, ,
 		\end{align*}
 		with
 		\begin{equation*}
 			\sigma^2 = \frac{1-(1+|z|^2)\, e^{-|z|^2}}{1-e^{-|z|^2}} \qquad \text{and} \qquad \theta = \frac{e^{-|z|^2/2}\big(1-|z|^2-e^{-|z|^2}\big)}{1-(1+|z|^2)e^{-|z|^2}}\, .
 		\end{equation*}
 		Finally, we observe that $$\det \big[\Cov\big(\widehat{g}(0),\widehat{g}(z)\big)\big] = 1-e^{-|z|^2}$$ and plug everything into the definition~\eqref{eq:split_for_c_2} of $c_2^\dprime$ to get
 		\begin{align*}
 			c_2^\dprime &= \frac{1}{4\pi} \int_{\bC} \frac{\sigma^4}{1-e^{-|z|^2}} \bigg(\big(\beta_0 + \log \sigma^2\big)^2 - \beta_0^2 + \big(\beta_1 - \log \sigma^2\big)^2 |\theta|^2 + \sum_{j=2}^\infty \beta_j^2 |\theta|^{2j}\bigg) \, {\rm d}m(z) \\ & = \frac{1}{4} \, \bigg(I_1 + I_2 + I_3\bigg) \, ,
 		\end{align*}
 		with $I_1,I_2,I_3$ given by~\eqref{eq:def_of_Integrals}. Plugging this into~\eqref{eq:split_for_c_2} gives what we wanted.
 	\end{proof}
	\begin{proof}[Proof of Theorem~\ref{thm:variance_after_split}]
		The proof follows immediately by combining Lemma~\ref{lemma:value_for_c_1}, Lemma~\ref{lemma:value_of_c_2} and Lemma~\ref{lemma:value_of_c_3}.  
	\end{proof}
	\section*{Acknowledgments} 
	We would like to thank Alon Nishry, Aditya Potukuchi and Mikhail Sodin for very helpful discussions. 
	The research of M.M. is supported in part by an NSF CAREER grant DMS-2336788 as well as grants DMS-2137623 and DMS-2246624.  The research of O.Y. is supported by the ERC Advanced Grant 692616 and by ISF Grants 1903/18 and 1288/21.

	\pagebreak
		\appendix
	
	\section{Proof of Proposition~\ref{prop:formula_for_joint_moments}}
	
	\label{sec:proof_of_kac_rice_formula}
	The starting point is a sufficiently general version for the Kac-Rice formula, given by the following theorem from~\cite{Azais-Wschebor}.
	\begin{theorem}{\cite[Theorem 6.4]{Azais-Wschebor}}
		\label{theorem:kac_rice_from_AW}
		Let $Z:U \to \bR^d$ be a Gaussian random field with $U \subset \bR^d$ open.  Assume that:
		\begin{itemize}
			\item the function $t \mapsto Z(t)$ is almost surely smooth;
			\item for each $t \in U$ we have $\det \Cov(Z(t)) \neq 0$;
			\item $\bP\big(\exists t \in U, Z(t) = 0, \det(Z'(t)) = 0 \big) = 0$.
		\end{itemize}
		Suppose further that one has another random field $Y^t : U \times \bR^{\ell}\to \bR^m$ satisfying the following conditions: 
		\begin{itemize}
			\item $Y^t$ is measurable and almost surely $(t,w) \mapsto Y^t(w)$ is continuous; 
			\item for each $t$ the random process $(s,w) \mapsto (Z(s),Y^t(w))$ is (real) Gaussian.
		\end{itemize}
		Let $H:U \times \mathcal{C}(\bR^\ell,\bR^m) \to \bR$ be a bounded and continuous function, where the topology on $\mathcal{C}(\bR^\ell,\bR^m)$ is the topology of uniform convergence on compact sets. Then, for each compact $I\subset U$, we have 
		\begin{equation*}
		\bE\left(\sum_{t \in I, Z(t) = 0} H(t,Y^t)  \right) = \int_I \bE \left(|\det(Z'(t))| \, H(t,Y^t) \mid Z(t) = 0  \right) p_{Z(t)}(0)\,dt
		\end{equation*}
		where $p_{Z(t)}(0)$ is the density of the random vector $Z(t)$ evaluated at $0$.
	\end{theorem}
	With Theorem~\ref{theorem:kac_rice_from_AW} at our disposal, we can prove the following version of it for Gaussian analytic functions, which will take us most of the way towards Proposition~\ref{prop:formula_for_joint_moments}. 
	\begin{lemma}
		\label{lemma:Kac_Rice_for_GAFs}
		Let $f$ be a Gaussian analytic function so that for any $k$ distinct points $w_1,\ldots,w_k$ we have $\det \big[\Cov (f(w_1),\ldots,f(w_k))\big] \neq 0$. For a compact set $K \subset \bC$, let $\mathcal{Z}_k(K)$ denote the set of distinct $k$-tuples of zeros of $f$ in $K$. Then for bounded and continuous functions $h_j : \bC^2 \to \bR$, $1\le j \le k+\ell$, we have 
		\begin{align*}
			\bE\bigg[ \bigg(\int_{K^\ell} \prod_{i = 1}^\ell h_{k + i}(f(z_i),z_i)\, {\rm d}m(\mathbf{z}) \bigg) & \sum_{(y_1,\ldots,y_k) \in \mathcal{Z}_k(K)} \prod_{j = 1}^k h_j(f'(y_j),y_j) \bigg] \\
			&= \int_{K^\ell} \int_{K^k}  \frac{\La(\mathbf{u},\mathbf{z};h,f)}{\pi^k \det \big[\Cov\big( (f(u_j))_{j=1}^k \big) \big]}\,{\rm d} m(\mathbf{u}) \, {\rm d} m(\mathbf{z}) \, ,
		\end{align*}
		where
		\begin{equation*}
			\La(\mathbf{u},\mathbf{z};h,f) \stackrel{{\rm def}}{=} \bE \bigg[\Big(\prod_{j = 1}^k |f'(u_j)|^2 h_j(f'(u_j),u_j)\Big) \Big( \prod_{i = 1}^\ell h_{k + i}(f(z_i),z_i) \Big)  \, \Big\vert \,  f(u_j) = 0, \ \forall j \in [k]  \bigg]\, .
		\end{equation*}
		
	\end{lemma}
	\begin{proof}
		We can assume without loss of generality that the $h_j$'s are non-negative. Let $\delta>0$ and consider the set
		\[
		U_\delta \stackrel{{\rm def}}{=} \big\{(y_1,\ldots,y_k) \in K^k \, : \, |y_i-y_j| \ge \delta, \ \forall i\not=j \big\} \subset \bC^k \, .
		\]
		We define $Z:U_\delta\to \bC^k$ by
		\[
		Z(\mathbf{u}) \stackrel{{\rm def}}{=} \big(f(u_1),\ldots,f(u_k)\big)
		\]
		and note that $Z$ is a smooth Gaussian process. It is straightforward to check that
		\[
		|\det\big(Z^\prime(\mathbf{u})\big)| = \prod_{j=1}^{k} |f^\prime (u_k)|\, ,
		\]
		and so, if $t\in U_\delta$ is such that $Z(t) = \det Z^\prime(t) =  0$ we must have a $z\in K$ with $f(z) = f^\prime(z) = 0$. By \cite[Lemma~2.4.2]{HKPV}, this event has probability zero.
		
		Define $Y^\mathbf{u}:U_\delta\times \bC^\ell \to \bC^{\ell+k}$ by  
		\begin{equation*}
			Y^\mathbf{u}(\mathbf{z}) \stackrel{{\rm def}}{=} \big(f^\prime(u_1),\ldots f^\prime(u_k),f(z_1),\ldots,f(z_\ell)\big)
		\end{equation*}
		and note that $Y^\mathbf{u}$ is measurable, continuous and Gaussian. Finally, we set
		\begin{equation*}
			H(\mathbf{u},Y^\mathbf{u}) \stackrel{{\rm def}}{=} \bigg(\prod_{j=1}^{k} h_j(f^\prime(u_j),u_j)\bigg) \bigg(\int_{K^\ell} \prod_{i=1}^{\ell} h_{k+i} (f(z_i),z_i) \, {\rm d}m(\mathbf{z}) \bigg)
		\end{equation*}
		and note that $H$ is bounded and continuous. By the above, we may apply Theorem~\ref{theorem:kac_rice_from_AW} and see that
		\begin{align*}
			\bE&\left[ \left(\int_{K^\ell} \prod_{i = 1}^\ell h_{k + i}(f(z_i),z_i)\, {\rm d}m(\mathbf{z}) \right)\sum_{(y_1,\ldots,y_k) \in \mathcal{Z}_k(K) \cap U_\delta } \prod_{j = 1}^k h_j(f'(y_j),y_j) \right] \\
			&=\int_{U_\delta} \bE \Big[\, |\det[Z'(\mathbf{u})]| \,  H(\mathbf{u},Y^\mathbf{u}) \mid Z(\mathbf{z}) = 0  \Big] \, p_{Z(\mathbf{z})}(0)\, {\rm d}m(\mathbf{u}) \,.
		\end{align*}
		where $p_{Z(\mathbf{z})}$ is the density function of the Gaussian vector $\big(f(u_1),\ldots,f(u_k)\big)$, and hence
		\[
		p_{Z(\mathbf{z})}(0)^{-1} = \pi^k \det\big[\Cov\big( (f(u_j))_{j=1}^k \big) \big]\, .
		\]
		Since we assumed that the $h_j$'s are non-negative, we can send $\delta\downarrow 0$ and apply the monotone convergence theorem, which implies the lemma. 
	\end{proof}
	\begin{proof}[Proof of Proposition~\ref{prop:formula_for_joint_moments}]
		Let $\{\psi_m\}$ be a sequence of continuous functions so that $|\psi_m|\le |\psi_{m+1}|$ and $\psi_m(s)\xrightarrow{m\to \infty} \log s$ for all $s>0$. For a compact set $K\subset \bC$ we apply Lemma~\ref{lemma:Kac_Rice_for_GAFs} and observe that
		\begin{align}
				\label{eq:KR-first}
				&n^{-a}\bE\left[ \left(\int_{K^b} \prod_{i = 1}^\ell \psi_m(|\widehat{f}(z_i)| )\,{\rm d}\mu(\mathbf{z})\right)\sum_{(y_1,\ldots,y_a) \in \mathcal{Z}_a(K)} \prod_{j = 1}^a \psi_m(|D\widehat{f} (y_j)|)^{p_j} \right]  \\
				&= n^{-a}\int_{K^b} \int_{K^a}  \frac{\bE \left[\left(\prod_{j = 1}^a |f'(u_j)|^2 \psi_m(|D\widehat{f}(u_j)|)^{p_j}\right) \left( \prod_{i = 1}^b \psi_m(|f(z_i)|) \right) \,|\, f(u_j) = 0\, \forall j  \right] }{\pi^a \det \Cov\left( (f(u_j))_{j = 1}^a \right)} \, {\rm d}m(\mathbf{u}) \, {\rm d}\mu^{\otimes b}(\mathbf{z}) \nonumber  \\
				&=\int_{K^b} \int_{K^a}  \frac{\bE \left[\left(\prod_{j = 1}^a |D\widehat{f}(u_j)|^2 \psi_m(|D\widehat{f}(u_j)|)^{p_j}\right) \left( \prod_{i = 1}^b \psi_m(|f(z_i)|) \right) \,|\, f(u_j) = 0\, \forall j  \right] }{ \det \Cov\left( (\widehat{f}(u_j))_{j = 1}^a \right) }\, {\rm d} \mu^{\otimes a}(\mathbf{u}) \, {\rm d}\mu^{\otimes b}(\mathbf{z}) \nonumber
				\end{align}
		where $f= f_n$ is our random polynomial~\eqref{eq:def_random_polynomial}. We claim that
		\begin{equation}
			\label{eq:KR_bound_to_justify_DOM}
			\bE\left[ \left(\int_{\bC^b} \prod_{i = 1}^b |\log|\widehat{f}(z_i)| |\, {\rm d} \mu^{\otimes b}(\mathbf{z})\right)\sum_{(y_1,\ldots,y_a) \in \mathcal{Z}_a(\bC)} \prod_{j = 1}^a |\log|D\widehat{f} (y_j)||^{p_j} \right] < \infty \, . 
		\end{equation}
		Assuming that~\eqref{eq:KR_bound_to_justify_DOM} holds, we can finish the proof of the proposition. Simply take $K\uparrow \bC$ and $m\to \infty$ in~\eqref{eq:KR-first} and apply the dominated convergence theorem, which is justified by~\eqref{eq:KR_bound_to_justify_DOM}. To prove~\eqref{eq:KR_bound_to_justify_DOM}, simply apply Lemma~\ref{lemma:Kac_Rice_for_GAFs} as in~\eqref{eq:KR-first}, but with $|\psi_m|$ instead of $\psi_m$. As both sides are non-decreasing when $m$ increase and as $K$ exhaust $\bC$, we can apply the monotone convergence theorem to obtain equality of the limits. Finally, by the assumption of Proposition~\ref{prop:formula_for_joint_moments}, the right-hand-side is finite, and hence~\eqref{eq:KR_bound_to_justify_DOM} holds and we are done. 
	\end{proof}

	\section{Positivity of the variance}
	
	\label{sec:variance_is_positive}
	In order to show that the limiting variance is non-trivial, it is sufficient to prove that $c_1 \neq c_2$.  In particular, we will prove \begin{fact}\label{fact:c2-c1}
		We have 		$c_2 - c_1 > 0.025$.
	\end{fact}
	
	\noindent We will tackle $I_1$ and $I_2$ separately, beginning with $I_1$.
	\begin{fact}\label{fact:I2}
		$I_2 \geq 1.348$.
	\end{fact}
	\begin{proof}
		Note that for $s \geq 1$ we have $|\beta_1| > \left| \log(1 - \frac{s}{e^s - 1})\right|$ and have \begin{equation}
			\left(\beta_1 - \log\left(1- \frac{s}{e^s - 1} \right) \right)^2 \geq 	\left(\beta_1 + \frac{3}{2}\frac{s}{e^s - 1} \right)^2\,.
		\end{equation}
		In particular, we have \begin{align*}
			I_2 \geq \int_1^\infty  \big(1-e^{-s}\big)^{-1} \Big(1-\frac{s}{e^s-1}\Big)^2 	\left(\beta_1 + \frac{3}{2}\frac{s}{e^s - 1} \right)^2 \bigg(\frac{e^{-s/2}(1-s-e^{-s})}{1-(s+1)e^{-s}}\bigg)^2  \, {\rm d}s\,.
		\end{align*}
		Letting $h_2$ denote the integrand appearing on the right-hand side, we note that $h_2$ is monotone increasing for $s \in [0,1]$ and so we may bound \begin{equation}
			I_2 \geq \int_0^\infty h_2(s)\,{\rm d}s - h_2(1)\,.
		\end{equation}
		Exact computation gives \begin{align*}
			\int_0^\infty h_2(s)\,{\rm d}s &= \frac{\pi^2}{2}\left(\frac{1}{12} + \frac{\pi^2}{10} - \frac{\gamma}{3}(1-\gamma) \right) - 3 \zeta(3)\left(\frac{1}{4} + \gamma \right) - \frac{\gamma}{2}(3 - \gamma)  - \frac{19}{16}> 1.408 \\
			h_2(1) &= \frac{(2 \gamma e - 4d - 2\gamma + 7)^2}{4(e - 1)^5} \leq .06\,.
		\end{align*}
	\end{proof}
	
	To handle $I_1$ we first simplify slightly.
	\begin{fact}\label{fact:I1}
		\begin{align*}I_1 &= -\frac{(1-\gamma)^2}{2}\left(1 + \frac{\pi^2}{3}\right)  + I_1' > -.384 + I_1'
		\end{align*}
		where $$I_1' = \int_{0}^{\infty} \big(1-e^{-s}\big)^{-1} \Big(1-\frac{s}{e^s-1}\Big)^2\Big[2\beta_0\log\Big(1-\frac{s}{e^s-1}\Big) +\log^2\Big(1-\frac{s}{e^s-1}\Big)\Big] \,{\rm d}s$$
	\end{fact}
	\begin{proof}
		This follows from computing $$\int_0^\infty \left[1 - \big(1-e^{-s}\big)^{-1} \Big(1-\frac{s}{e^s-1}\Big)^2\right]\,{\rm d}s = \frac{1}{2} + \frac{\pi^2}{6} $$
		along with recalling that $\beta_0 =  1- \gamma$.
	\end{proof}
	
	We now similarly lower bound $I_1'$ in a similar fashion to $I_2$.
	
	\begin{fact} \label{fact:I1p} $I_1' \geq -.472$
	\end{fact}
	\begin{proof}
		Note that the integrand of $I_1'$ is positive for $s \in [0,1]$.  For $s \geq 1$ we may bound $$I_1' \geq \int_1^\infty (1 - e^{-s})^{-1}\left(1- \frac{s}{e^s - 1} \right)^2\left(-2\beta_0\cdot \left(\frac{s}{e^s - 1} + \left(\frac{s}{e^s - 1}\right)^2 \right)+ \left(\frac{s}{e^s - 1} \right)^2 \right)\,{\rm d}s\,.$$
		
		Letting $h_1$ denote the integrand on the right-hand side we see that $h_1 \leq 0$ and so we may bound $$I_1' \geq \int_0^\infty h_1(s)\,{\rm d}s - |h_1(1)|\,.$$
		Exact computation again gives \begin{align*}
			\int_0^\infty h_1(s)\,{\rm d}s &= \pi^2\left(\frac{1}{2} - \frac{2\gamma}{3} + \frac{\pi^2}{15}(1-2\gamma) \right) + \zeta(3)(18\gamma - 11) - \frac{\gamma}{2} + \frac{5}{12} > -.472
		\end{align*}
	\end{proof}
	
	\begin{proof}[Proof of Fact \ref{fact:c2-c1}]
		Write \begin{align*}
			4(c_2 - c_1) &= I_1 + I_2 + I_3 + \frac{\pi^2}{6} + \gamma(\gamma - 2) - \zeta(3) \\
			&\geq -0.384 -0.472 + 1.348 +  \frac{\pi^2}{6} + \gamma(\gamma - 2) - \zeta(3)  \\
			&\geq 0.1
		\end{align*}
		where on the first inequality we used $I_3 \geq 0$ and Facts \ref{fact:I2}, \ref{fact:I1}, and \ref{fact:I1p}.
	\end{proof}
	
	\vspace{2cm}
\end{document}